\documentclass[11pt,reqno,a4paper]{article}
\usepackage[T1]{fontenc}
\usepackage[utf8]{inputenc}
\usepackage[linktocpage]{hyperref}
\usepackage{graphicx}
\usepackage{slashed}

\usepackage{amsmath,amssymb,amsthm,amsfonts,amscd,authblk,dsfont,cancel,diagbox,
enumerate,epsfig,fancyhdr,latexsym,lmodern,mathrsfs,mathtools,multirow,
musicography,skull,simplewick,subcaption,tikz-cd}
\usepackage{amscd}
\usepackage{amssymb}
\usepackage[noadjust]{cite}
\usepackage{esint}
\usepackage{enumitem}
\usepackage{amsmath}
\usepackage{dsfont}

\usepackage[mathscr]{eucal}
\usepackage{pst-all}

\numberwithin{equation}{section}
\allowdisplaybreaks[1]
\definecolor{labelkey}{gray}{.65}

\usepackage{tikz-cd}
\usetikzlibrary{matrix,arrows,decorations.pathmorphing}

\input{xy}
\xyoption{all}

\usepackage{hyperref}
\hypersetup{
pdftitle={},%
pdfauthor={},%
pdfsubject={},%
pdfkeywords={},%
colorlinks=true,%
linkcolor=black,%
urlcolor=black,
linktoc={},
linktocpage={false},%
pageanchor={},
citecolor={black},
}

\usepackage[english]{babel}
\usepackage[utf8]{inputenc}
\usepackage[margin=2.35cm]{geometry}

\usepackage{url}

\usepackage{cite}

\usepackage{titlesec}
\titleformat{\section}{\large\bfseries\filcenter}{\thesection}{1em}{}
\titleformat{\subsection}{\bfseries}{\thesubsection}{1em}{}
\let\oldbibliography\thebibliography
\renewcommand{\thebibliography}[1]{\oldbibliography{#1}
\setlength{\itemsep}{0\baselineskip}}
\makeatletter
\renewcommand\@makefntext[1]{\leftskip=2em\hskip-0.5em\@makefnmark#1}
\makeatother
\newtheorem{Def}{Definition}[section]
\newtheorem{Thm}[Def]{Theorem}
\newtheorem{Prp}[Def]{Proposition}
\newtheorem{Lemma}[Def]{Lemma}
\newtheorem{Remark}[Def]{Remark}
\newtheorem{Corollary}[Def]{Corollary}
\newtheorem{Question}[Def]{Question}
\newtheorem{Example}[Def]{Example}

\newtheorem{Setup}[Def]{Set-up}

\newcommand{\beq}{\begin{equation}}
\newcommand{\eeq}{\end{equation}}
\newcommand{\Proof}{\begin{proof}}
	\newcommand{\QED}{\end{proof} \noindent}

\newcommand{\la}{\langle}
\newcommand{\ra}{\rangle}

\newcommand{\bra}{\mathopen{<}}
\newcommand{\ket}{\mathclose{>}}
\newcommand{\Sl}{\mbox{$\prec \!\!$ \nolinebreak}}
\newcommand{\Sr}{\mbox{\nolinebreak $\succ$}}

\newcommand{\C}{\mathbb{C}}
\newcommand{\R}{\mathbb{R}}
\renewcommand{\1}{\mbox{\rm 1 \hspace{-1.05 em} 1}}

\newcommand{\N}{\mathbb{N}}

\newcommand{\V}{\mathcal{V}}

\newcommand{\Sact}{{\mathcal{S}}}

\newcommand\B{{\mathscr{B}}}

\newcommand{\Dir}{{\mathcal{D}}}
\renewcommand{\H}{\mathscr{H}}

\newcommand{\F}{{\mathscr{F}}}

\newcommand{\scrA}{\mycal A}
\newcommand{\scrB}{\mycal B}
\newcommand{\scrM}{\mycal M}
\newcommand{\scrS}{\mycal S}
\newcommand{\scrE}{\mycal E}
\newcommand{\scrF}{\mycal F}

\newcommand{\scrN}{\mycal N}

\newcommand{\bitem}{\begin{itemize}[leftmargin=2.5em]}
\newcommand{\eitem}{\end{itemize}}

\newcommand{\macro}{{\mathrm{\tiny{macro}}}}

\DeclareFontFamily{OT1}{rsfso}{}
\DeclareFontShape{OT1}{rsfso}{m}{n}{ <-7> rsfso5 <7-10> rsfso7 <10-> rsfso10}{}
\DeclareMathAlphabet{\mycal}{OT1}{rsfso}{m}{n}

\begin{document}
%
%
%
%
\begin{center}
\vspace*{-1cm}

{\Large\bf  The Cauchy Problem for Symmetric Hyperbolic Systems\\[3mm] with Nonlocal Potentials} 

\vspace{5mm}

{\bf by}

\vspace{3mm}
\noindent
{  \bf Felix Finster$^1$, Simone Murro$^2$ and Gabriel Schmid$^2$}\\[2mm]

\noindent  {\it $^1$ Fakult\"at f\"ur Mathematik, Universit\"at Regensburg, Regensburg, Germany}\\[1mm]

\noindent  {\it $^2$Dipartimento di Matematica, Universit\`a di Genova \& GNFM-INdAM \& INFN, Italy}\\[2mm]

Emails: \ {\tt  finster@ur.de, simone.murro@unige.it, gabriel.schmid@edu.unige.it}
\\[10mm]
\end{center}
\begin{abstract}
In this paper, we investigate the initial value problem for symmetric hyperbolic systems on globally hyperbolic Lorentzian manifolds with potentials that are both nonlocal in time and space.  When the potential is retarded and uniformly bounded in time, we establish well-posedness of the Cauchy problem on a time strip, proving existence, uniqueness, and regularity of solutions. If the potential is not retarded but has only short time range, we show that strong solutions still exist, under the additional assumptions that the uniform bound in time is sufficiently small compared to the range in time and that its kernel decays sufficiently fast in time with respect to the zero-order terms of the system. Furthermore, we present a counterexample demonstrating that when the uniform bound is too large compared to the time range, solutions may fail to exist. As an application, we discuss Maxwell's equations in linear dispersive media on ultrastatic spacetimes, as well as the Dirac equation with nonlocal potential naturally arising in the theory of causal fermion systems. Our paper aims to represent the starting point for a rigorous study for the Cauchy problem for the semiclassical Einstein equations.
\end{abstract}
\tableofcontents

\section{Introduction} \label{secintro}

\medskip

Symmetric hyperbolic systems are a fundamental class of partial differential equations that are widely employed in mathematical physics due to their well-posedness properties and structural compatibility with physical conservation laws. Traditionally studied in the context of local interactions, it would be highly desirable to extend these systems to incorporate nonlocal interactions, thereby providing a more flexible and encompassing framework for modelling complex phenomena where long-range correlations or memory effects play a significant role. 

Incorporating nonlocality opens up a broad spectrum of applications across both classical and quantum theories. Viscoelastic materials, for instance, exhibit memory-dependent stress-strain relations, which naturally lead to equations with nonlocal temporal kernels (see e.g.~\cite{viscoelastic1,viscoelastic2}). In  semiclassical response theory, nonlocalities  emerge due to the interaction of classical matter with quantum fields. For example, in semiclassical Maxwell equations (see e.g.~\cite{Galanda, semiMax1,semiMax2,semiMax3}) and Einstein equations (see e.g.~\cite{FlanaganWald,MedaPinamonti1, MedaPinamonti2}) the nonlocality arises from the fact that the expectation values of quantum observables are obtained using correlation functions (that are nonlocal in time and space) of a (given) quantum state. Collapse models in quantum mechanics often involve stochastic nonlocal modifications to Schr\"odinger dynamics to resolve the measurement problem, see e.g.~\cite{collapse,heatcoll}.  Notably, in the context of quantum spacetimes, the dynamics of classical fields acquire an inherently nonlocal character in both time and space. This nonlocality stems from the non-commutative structure of the underlying coordinate algebra: when the coordinate ring is no longer abelian, standard notions of locality break down. As a result, differential operators acting on fields must be reinterpreted in terms of non-commutative geometry, often leading to modified dispersion relations and integral-type evolution equations. These features reflect the fact that field values at a given point are influenced by extended regions of spacetime, as encoded by the non-commuting coordinates, see e.g.~\cite{Fewster,LechnerVerch,Verch}. Last but not least, nonlocal potentials play a key role in the theory of causal fermion systems, a non-perturbative approach to fundamental physics that aims to unify general relativity and quantum field theory within a single mathematical structure (see e.g.~\cite{intro,cfsweblink}). More precisely, the nonlocal potential arises in the analysis of the causal action principle for causal fermion systems in Minkowski space~\cite{nonlocal}, and they are also crucial for the quantum field theory limit of causal fermion systems~\cite{fockdynamics}. Let us remark that, beyond these physical applications, nonlocal PDEs have also become a vibrant area of research in mathematical analysis, offering new challenges and insights. Notable examples include nonlocal conservation laws (see e.g.~\cite{Colombo,Radici,DiFrancesco,Betancourt}), nonlocal diffusion-reaction equations (see e.g.~\cite{PalMelnik,Tarasov,Mezzanotte}), and aggregation-diffusion models (see e.g.~\cite{Fagioli,Daneri}), which describe systems ranging from population dynamics to granular flows and self-organizing agents. These models often feature integral operators, fractional derivatives, or memory terms, and they illustrate the mathematical richness and physical relevance of nonlocal phenomena.

Despite the diversity of these systems, a unifying mathematical treatment based on symmetric hyperbolic systems with nonlocal potentials that provides a coherent and powerful language to investigate their dynamics is still missing. In the present paper we aim to provide a systematic study of the Cauchy problem for a symmetric hyperbolic system which involves a nonlocal
potential of the form
\begin{align}\label{nonlocintro}
	\big( \B\psi \big)(x)=\int_{\scrM}k_{\B}(x,y)\psi(y)\,d\mu_{\scrM}(y)
\end{align}
with a distributional integral kernel~$k_{\B}(x,y)$ defined on a globally hyperbolic manifold~$(\scrM, g)$ (the general setup will be introduced in Section~\ref{secprelim}). As we shall see, the nonlocal potential typically leads to a breakdown of the finite propagation speed and, as a consequence, Einstein causality as well. This is reflected on the mathematical side by the fact that the standard energy estimates are no longer available and they must be adapted to include the effect of nonlocal interactions. We shall show that the Cauchy problem is well-posed, provided that the integral kernel is sufficiently small and its time range is sufficiently short. It is the main purpose of this paper to make this statement mathematically precise and to quantify it by suitable energy estimates. One example are time-delay systems, in which case the derivative of the unknown function at a certain time is given in terms of the values of the function at previous times. In the relativistic setting, such time delays can be described by the case of
\begin{itemize}
\item[] {\em{retarded kernels}}: $k_{\B}(x,y)=0$ unless~$y$ lies in the causal past of~$x$.
\end{itemize}
This case is of particular interest for the semiclassical linear response theory. It turns out that causality and finite propagation speed are respected. But our results also apply to the case that the kernels are not retarded. In this case, we still obtain solutions to the Cauchy problem, provided that the nonlocal potential has a short time range and is sufficiently small. However, it is unknown whether the solutions are unique. A particular case where we do get uniqueness is the case of
\begin{itemize}
\item[] {\em{symmetric kernels}}: $k_{\B}(x,y)^\dagger = k_{\B}(y,x)$ for all~$x,y \in \scrM$
\end{itemize}
This case is of particular interest for the Dirac equation, where the symmetry of the potential is related to the probabilistic interpretation. This is why the uniqueness statement is worked out only for the Dirac equation in Section~\ref{sec:Dirac}.

\subsection{Summary of Main Results}
The basic set-up considered in this paper (cf.~Set-up~\ref{Setup}) is as follows: We consider a globally hyperbolic Lorentzian manifold 
\begin{align*}
	\scrM=\R\times\scrN,\qquad g=-\beta^2 dt\otimes dt+h_{t}\, ,
\end{align*}
where $\scrN$ is a smooth and spacelike Cauchy surface, $\beta\in C^{\infty}(\scrM,(0,\infty))$ a lapse function and $(h_{t})_{t\in\R}$ a one-parameter family of Riemannian metrics on $\scrN$. Furthermore, let
\begin{align*}
	\scrS\colon C^{\infty}(\scrM,\scrE)\to C^{\infty}(\scrM,\scrE)\, ,
\end{align*}
be a \textit{symmetric hyperbolic system} (see~Definition~\ref{Def:SHS}), that is, a first-order linear differential operator acting on smooth sections of a Hermitian vector bundle $(\scrE\xrightarrow{\pi}\scrM,\Sl\cdot|\cdot\Sr_{\scrE})$ whose principal symbol $\sigma_{\scrS}\colon T^{\ast}\scrM\to\mathrm{End}(\scrE)$ is Hermitian with respect to $\Sl\cdot|\cdot\Sr_{\scrE}$ and which is hyperbolic in the sense that $\sigma_{\scrS}(\tau)$ is positive-definite with respect to $\Sl\cdot|\cdot\Sr_{\scrE}$ for future-directed timelike covectors $\tau$. Considering the foliation $(\scrN_{t}:=\{t\}\times\scrN)_{t\in\R}$ of $\scrM$, there is a natural Hilbert space $\H_t$ associated to each spatial slice $\scrN_{t}$, which is obtained as the completion of the space of initial data $C^{\infty}_{\mathrm{c}}(\scrN_{t},\scrE_{t}\vert_{\scrN_{t}})$ with respect to the positive-definite inner product
\begin{align*}
	(\psi|\varphi)_{t}=\int_{\scrN_{t}}\,\Sl\sigma_{\scrS}(\eta_{t})\psi|\varphi\Sr_{\scrE}\,d\mu_{\scrN_{t}}
\end{align*}
with $\eta_{t}=\beta dt\vert_{\scrN_{t}}$ denoting the future-directed timelike normal covector of $\scrN_{t}$ in $\scrM$. On a time strip $\overline{\scrM}_T = [0,T]\times\scrN$ with $T>0$, we obtain correspondingly a natural Hilbert space $\H_{\overline{\scrM}_{T}}$ as the corresponding direct integral.

In addition, we consider a \textit{nonlocal potential}, that is, a linear and continuous operator 
\begin{align*}
	\B\colon C^{\infty}_{\mathrm{c}}(\scrM,\scrE)\to C^{\infty}(\scrM,\scrE)
\end{align*}
of the type~\eqref{nonlocintro} with distributional Schwartz kernel $k_{\B}\in\mathcal{D}^{\prime}(\scrM\times\scrM,\scrE\boxtimes\scrE^{\ast})$. Under some additional assumptions on $\B$, we can associate to it a \textit{time kernel}, namely a two-parameter family of operators $\B_{t,\tau}\colon C^{\infty}_{\mathrm{c}}(\scrN_{\tau},\scrE\vert_{\scrN_{\tau}})\to C^{\infty}(\scrN_{t},\scrE\vert_{\scrN_{t}})$ such that formally
\begin{align*}
	(\B\psi)(t,\vec{x})=\int_{\R}(\B_{t,\tau}\psi_{\tau})(\vec{x})\,dt\, ,
\end{align*}
where $\psi_{\tau}:=\psi(\tau,\cdot)$, as we explain in detail in Section~\ref{subsec:TimeKernel}. Now, as mentioned in the introduction, we consider the following two important special cases of nonlocal potentials in this article (see~Definitions~\ref{Def:NonLocPot} and~\ref{Def:Potentials}):
\begin{itemize}
	\item[(R)]\quad$\B_{t,\tau}=0$ for $\tau>t$.\hfill (\textit{retarded case})
	\item[(S)]\quad$\B_{t,\tau}=0$ for $\vert t-\tau\vert>\delta$ for some fixed $\delta>0$.\hfill (\textit{short time range case})
\end{itemize}
Our main results can be summarized as follows: Let $\scrS$ be a symmetric hyperbolic system on a Hermitian vector bundle $\scrE$ over a globally hyperbolic spacetime $(\scrM, g)$ and $\B$ a nonlocal potential with time kernel.  We consider the Cauchy problem
\begin{align} \label{eq:Caucy}
\begin{cases}
(\scrS - \B)\psi &= \phi \\
\psi|_{\scrN_0} &= \mathfrak{f}
\end{cases}
\end{align}
on the time strip $\overline{\scrM}_T$, where $\phi \in C^\infty_{\mathrm{c}}(\overline{\scrM}_T, \scrE)$ is a compactly supported smooth source and $\mathfrak{f}\in C^{\infty}_{\mathrm{c}}(\scrN_{0},\scrE\vert_{\scrN_{0}})$ some initial datum. Furthermore, define the modified potential $\V:=\beta\sigma_{\scrS}(\eta)^{-1}\B$.
\begin{itemize}
    \item[(i)]Suppose $\B$ is a \emph{retarded} nonlocal potential (as in (R) above) whose time kernel is uniformly bounded in time on $\scrM_T$ in the sense that
	\begin{align*}
		\exists C_{T}>0:\quad \Vert\V_{t,\tau}\psi_{\tau}\Vert_{t}\leq C_{T}\Vert\psi_{\tau}\Vert_{\tau}\,\qquad\forall \psi\in C^{\infty}_{\mathrm{sc}}(\scrM,\scrE),\, t,\tau\in [0,T]\, .
	\end{align*}
	Furthermore, we assume that $\B$ has \emph{past compact support} in the sense that $\B_{t,\tau}=0$ for $\tau<0$. Then, the Cauchy problem~\eqref{eq:Caucy} admits a strong solution on $\overline{\scrM}_{T}$ (see Definition~\ref{Def:StrongSol}). If, in addition, the covariant derivatives of $\B$ are also uniformly bounded in time in a suitable sense (explained in details in Section~\ref{Subsec:ExtSHS}), then the solution $\psi$ is smooth and hence in particular a classical solution on $\overline{\scrM}_{T}$ (see Definition~\ref{Def:ClassSol}). Moreover, the solution is unique and propagates at most with the speed of light, i.e.
    \begin{align*}
    \mathrm{supp}(\psi)\cap J^{+}(\scrN_{0}) \subset J^+(\mathrm{supp}(\phi)\cup\mathrm{supp}(\mathfrak{f}))\,.
    \end{align*}
    This is the content of Theorem~\ref{Thm:Ret}.
    \item[(ii)]Suppose that $\B$ has \emph{short time range} $\delta > 0$ (as in (S) above) and suppose that the zero-order terms of $\scrS$ are uniformly bounded in time in the sense that 
    \begin{align*}
    		\exists D>0:\quad \Vert(\mathcal{Z}_{\scrS}\psi)_{t}\Vert_{t}\leq D\Vert\psi_{t}\Vert_{t}\,\qquad\forall \psi\in C^{\infty}_{\mathrm{sc}}(\scrM,\scrE),\, t\in\R\, ,
    \end{align*}
    where $\mathcal{Z}_{\scrS}:=\beta\sigma_{\scrS}(\eta)^{-1}(\scrS+\scrS^{\dagger})$ is a zero-order operator with $\scrS^{\dagger}$ being the formal adjoint of $\scrS$ with respect to the sesquilinear form obtained by integrating over the bundle metric. Furthermore, suppose that $\B$ is uniformly bounded on all of $\R$ with a bound that is ``small'' compared to $\delta$ in the sense that
    \begin{align*}
		\exists C\,\,\text{ with }\,\,0<C<(8e\delta^{2})^{-1}:\quad \Vert\V_{t,\tau}\psi_{\tau}\Vert_{t}\leq Ce^{-\frac{1}{2}D\vert\tau\vert}\Vert\psi_{\tau}\Vert_{\tau}\,\qquad\forall \psi\in C^{\infty}_{\mathrm{sc}}(\scrM,\scrE),\, t,\tau\in\R\, .
	\end{align*}
    Then the Cauchy problem~\eqref{eq:Caucy} admits a strong solution on $\overline{\scrM}_{T}$. This is the content of Theorem~\ref{Thm:Small}.
    \item[(iii)]There exist nonlocal potentials $\B$ with short time range $\delta$ that are uniformly bounded in time, but with constant $C > (8e\delta^2)^{-1}$, for which the Cauchy problem~\eqref{eq:Caucy} admits \emph{no strong solution}, even in the case in which $\scrS+\scrS^{\dagger}=0$ (and hence $\mathcal{Z}_{\scrS}=0$). The corresponding counterexample is discussed in Example~\ref{Counterexample}.
\end{itemize}

As an application, the Cauchy problem for symmetric hyperbolic systems with a retarded nonlocal potential is illustrated using Maxwell's equations on ultrastatic manifolds in linear dispersive media (see Proposition~\ref{Thm:Maxwell}). In addition, the case of nonlocal potentials with short time range is applied to the Dirac operator (see Corollary~\ref{Cor.Dirac}).

\subsection{Structure of the Paper} 
In Section~\ref{secprelim}, we review the necessary geometric and analytic preliminaries on globally hyperbolic Lorentzian manifolds as well as symmetric hyperbolic systems in Hermitian vector bundles defined thereon.

Section~\ref{Sec:NonLocPot} begins by recalling the framework of distributions in vector bundles, after which we introduce the class of nonlocal potentials considered in this work. In particular, we define the notion of a \textit{time kernel} for general bi-distributional kernels, which enables us to interpret nonlocal potentials as linear operators acting on Hilbert spaces associated to fixed Cauchy time slices. 

After defining the notion of \emph{strong} and \emph{classical solutions} in the nonlocal context and after a discussion of key energy estimates and regularity techniques, the main results are presented in Section~\ref{Sec:SHSNonLoc}, where we study the Cauchy problem on a time strip for symmetric hyperbolic systems with nonlocal potentials that are uniformly bounded in time. We address both the case of retarded nonlocal potentials as well as those with small range in time. Additionally, we present a counterexample demonstrating that when the uniform bound is too large compared to the range in time, solutions may fail to exist at all.

In Section~\ref{sec:Maxwell} we discuss an example of a symmetric hyperbolic system with nonlocal retarded potential, namely Maxwell's equations on ultrastatic globally hyperbolic manifolds in the presence of a linear dispersive medium.

Finally, in Section~\ref{sec:Dirac}, we apply our results to the case of the Dirac operator coupled to a nonlocal potential with short time range and illustrate some recent applications to causal fermion systems.
In this section, we also give the uniqueness proof
for symmetric potentials, in the sense that, given Cauchy data, there exists a unique solution
in the Hilbert space endowed with a scalar product which is defined nonlocally by integrating
over a surface layer.

\paragraph*{Acknowledgments:}
We are grateful to  Simone Di Marino, Stefano Galanda, Christian Gérard, Paolo Meda and Nicola Pinamonti for helpful discussions. GS gratefully acknowledges the hospitality of the University of Regensburg, where part of this work was carried out. We also gratefully acknowledge the MathOverflow community for valuable discussions related to the content of this article. This work has been produced within the activities of the INdAM-GNFM.

\paragraph*{Funding:}
 The research was supported in part by the MIUR Excellence Department Project 2023–2027 awarded to the Department of Mathematics of the University of Genoa. We are also grateful for support from the ``Universit\"atsstiftung Hans Vielberth'' in Regensburg.

\section{Geometric and Analytic Preliminaries} \label{secprelim}

The aim of this section is to recall basic definitions and results concerning symmetric hyperbolic systems on globally hyperbolic Lorentzian manifolds.

\subsection{Globally Hyperbolic Manifolds}
Let $\scrM$ be a connected, oriented and smooth manifold (without boundary) of dimension~$k+1$ with $k\geq 1$. We assume that $\scrM$ is either non-compact or compact with Euler characteristic zero, so that it can be equipped with a smooth Lorentzian metric $g$ of signature~$(- ,+, \ldots, +)$. In the class of Lorentzian manifolds, those called \textit{globally hyperbolic} provide a suitable background for analyzing the Cauchy problem for hyperbolic equations.
 
\begin{Def}[Globally Hyperbolic Manifolds]
	A connected, oriented, time-oriented and smooth Lorentzian manifold $(\scrM,g)$ is called \emph{globally hyperbolic} if the following holds true:
	\begin{itemize}
		\item[(i)] $(\scrM,g)$ is \emph{causal}, i.e.~there are no closed causal curves;
		\item[(ii)] for every $p, q \in \scrM$, the set $J^+ (p) \cap J^{-}(q)$ is compact, where $J^+(\cdot)$ (\textit{resp.}~$J^-(\cdot)$) denotes the causal future (\textit{resp.}~past) in $\scrM$.
	\end{itemize}
\end{Def}
 
In the seminal paper~\cite{Geroch}, Geroch showed that global hyperbolicity of a Lorentzian manifold is equivalent to the existence of a \textit{Cauchy hypersurface}. 

\begin{Def}[Cauchy Hypersurfaces]
	A subset $\scrN\subset\scrM$ of a spacetime $(\scrM,g)$ is a {\em Cauchy hypersurface}, if it is intersected exactly once by any inextendible future-directed smooth timelike curve.
\end{Def}

Let us remind the reader that a smooth curve $\gamma \colon I \to \scrM$, where $I \subset \mathbb{R}$ is an open interval, is said to be {\em inextendible}, if there does not exist any {\em continuous} curve $\gamma'\colon J \to \scrM$ defined on  an open interval  $J \subset \mathbb{R}$ such that $I \subsetneq J$ and $\gamma'|_I =\gamma$.

\begin{Thm}[\protect{\cite[Thm.~11]{Geroch}}]\label{thm:Geroch} 
	A Lorentzian manifold $(\scrM,g)$ is globally hyperbolic if and only if it admits a Cauchy hypersurface.
\end{Thm}

It turns out that the Cauchy hypersurfaces of a globally hyperbolic manifold $(\scrM,g)$ are co-dimension-one topologically embedded submanifolds of $\scrM$ that are homeomorphic to each other. As a by-product of Geroch’s theorem, it also follows that a globally hyperbolic manifold $(\scrM,g)$ admits a continuous foliation by Cauchy hypersurfaces, namely $\scrM$ is homeomorphic to a product $\R\times \scrN$ for some Cauchy hypersurface $\scrN$. This can be established by finding a {\em Cauchy time function}, i.e.~a continuous function $t\colon\scrM\rightarrow\mathbb{R}$ which is strictly increasing on any future-directed timelike curve and whose level sets $t^{-1}(t_0)$, $t_0\in \mathbb{R}$, are Cauchy hypersurfaces homeomorphic to $\scrN$. Geroch's splitting is topological in nature and the fact that it can be done in a smooth way has been part of the mathematical folklore whilst remaining an open problem for many years. Only recently Bernal and S\'anchez \cite{bernal+sanchez2} obtained a smooth version of the result of Geroch by introducing the notion of a \emph{Cauchy temporal function}.

\begin{Def}[Cauchy Temporal Functions] \label{def:Cauchy temporal}
	Given a connected and time-oriented smooth Lorentzian manifold $(\scrM,g)$,  we say that a smooth function  $t\colon\scrM \to \R$ is a {\em Cauchy temporal function}  if the embedded submanifolds of co-dimension one given by its level sets are smooth Cauchy hypersurfaces and if $(dt)^\sharp$ is past-directed timelike. 
\end{Def} 

\begin{Thm}[\protect{\cite[Thm.~1.1~and~1.2]{bernal+sanchez2}}, \protect{\cite[Thm.~1.2]{bernal+sanchez_further_results}}]\label{thm:Sanchez}
	For any globally hyperbolic manifold $(\scrM,g)$  there exists a Cauchy temporal function. In particular, it is (smoothly) isometric to the (globally orthogonal) product 
	\begin{align*} 
		(\R\times\scrN, - \beta^2 dt\otimes dt + h_t)\, ,
	\end{align*}
	where $t:\R\times \scrN \to \R$ is a Cauchy temporal function acting as a projection onto the first factor, $\scrN$ is a smooth, spacelike Cauchy hypersurface, $\beta\in C^{\infty}(\R\times\scrN,(0,\infty))$ is called the {\em lapse function}, and $h_t$ is a one-parameter family of Riemannian metrics on $\scrN$ depending smoothly on $t$.

	Moreover, if $\scrN \subset \scrM$ is any spacelike Cauchy hypersurface, then there exists a Cauchy temporal function $t$ such that $\scrN$ belongs to the foliation $t^{-1}(\R)$.
\end{Thm}
 
The class of globally hyperbolic manifolds is non-empty and besides Minkowski spacetime includes many significant spacetimes in general relativity and cosmology. Notable examples include \textit{black hole spacetimes} such as Schwarzschild and Kerr, as well as \textit{de Sitter spacetime}, a maximally symmetric solution of Einstein's equations with a positive cosmological constant. Of course, given any $n$-dimensional compact manifold $\scrN$ and a smooth one-parameter family of Riemannian metrics $\{h_t\}_{t\in\R}$, the Lorentzian manifold $(\R\times\scrN, g=-dt\otimes dt+h_t)$ provides an example of a globally hyperbolic manifold.

\subsection{Sobolev Spaces}\label{SobolevSpaces}
Let $(\scrM,g)$ be a Riemannian manifold with Levi-Civita connection $\nabla$ and $\scrE\xrightarrow{\pi}\scrM$ be a finite-rank $\C$-vector bundle over $(\scrM,g)$ equipped with a \textit{positive-definite} fibre metric 
\begin{align*}
	\Sl\cdot|\cdot\Sr_{\scrE}\in C^{\infty}(\scrM,(\overline{\scrE}\otimes\scrE)^{\ast}),\qquad \Sl\cdot\vert\cdot\Sr_{\scrE_{x}}\colon\scrE_{x}\times\scrE_{x}\to\C,\qquad x\in\scrM\, .
\end{align*}
We denote the space of smooth sections by $C^{\infty}(\scrM,\scrE)$ and the subspace of sections with compact support by $C^{\infty}_{\mathrm{c}}(\scrM,\scrE)$. Furthermore, we denote the natural $L^{2}$-Hilbert space by
\begin{align*}
	L^{2}(\scrM,\scrE):=\overline{C^{\infty}_{\mathrm{c}}(\scrM,\scrE)}^{\Vert\cdot\Vert_{L^{2}}},\qquad \bra\psi\vert\varphi\ket_{L^{2}}:=\int_{\scrM}\Sl\psi(x)\vert\varphi(x)\Sr_{\scrE_{x}}\,d\mu_{\scrM}(x)\, ,
\end{align*}
noting that it depends on both the Riemannian and bundle metric. Let us further choose a (metric compatible) connection 
\begin{align*}
	\nabla^{\scrE}\colon C^{\infty}(\scrM,\scrE)\to C^{\infty}(\scrM,\scrE\otimes T^{\ast}\scrM)
\end{align*} 
of $\scrE$. Writing $\scrE_{i}:=\scrE\otimes T^{\ast}\scrM^{\otimes i}$ for $i\in\N$, we denote the \textit{$i$th covariant derivative} by\footnote{By $\nabla^{\scrE}\otimes\nabla$ we mean the tensor product connection, i.e.~$\nabla^{\scrE}\otimes\nabla:=\mathrm{id}_{\scrE}\otimes\nabla+\nabla^{\scrE}\otimes\mathrm{id}_{T^{\ast}\scrM}$.}
\begin{align*}
	\nabla^{\scrE,i}\colon C^{\infty}(\scrM,\scrE)\xrightarrow{\nabla^{\scrE}}C^{\infty}(\scrM,\scrE_{1})\xrightarrow{\nabla^{\scrE}\otimes\nabla}C^{\infty}(\scrM,\scrE_{2})\xrightarrow{}\dots\xrightarrow{}C^{\infty}(\scrM,\scrE_{i})\, ,
\end{align*}
see e.g.~\cite[Sec.~12.8]{LeeBookManifolds}. We equip the bundles $\scrE_{i}$ with the obvious bundle metric induced by $\Sl\cdot|\cdot\Sr_{\scrE}$ and $g$, which we denote by $\Sl\cdot|\cdot\Sr_{\scrE_{i}}$. With this notation, we define the \textit{Sobolev space} of order $i\in\N$ by
\begin{align*}
	W^{2,i}(\scrM,\scrE)&:=\overline{\{\psi\in C^{\infty}(\scrM,\scrE)\mid \Vert\psi\Vert_{W^{2,i}}<\infty\}}^{\Vert\cdot\Vert_{W^{2,i}}},\\ \bra\psi\vert\varphi\ket_{W^{2,i}}&:=\sum_{j=0}^{i}\bra\nabla^{\scrE,j}\psi(x)\vert\nabla^{\scrE,j}\varphi(x)\ket_{L^{2}}\, .
\end{align*}
We stress that this definition strictly depends on the choice of bundle metric and connection. Furthermore, we remind the reader that for non-compact $\scrM$, this definition does in general neither coincide with the definition obtained by using distributional (weak) covariant derivatives nor with the space $W^{2,i}_{0}(\scrM,\scrE)$ obtained by taking the completion of $C_{\mathrm{c}}^{\infty}(\scrM,\scrE)$ with respect to $\Vert\cdot\Vert_{W^{2,i}}$. We refer to \cite{HebeyBook} as well as \cite[Chap.~2]{AubinBook} and \cite[Sec.~1.3]{Eichhorn} for further details.

\subsection{Symmetric Hyperbolic Systems}\label{subsec:SHS}
Consider a \textit{Hermitian vector bundle} over a globally hyperbolic manifold $(\scrM,g)$, that is, a $\mathbb{C}$-vector bundle $\scrE\xrightarrow{\pi}\scrM$ of finite rank $N\in\N$ together with a non-degenerate Hermitian fibre metric 
\begin{align*}
	\Sl\cdot|\cdot\Sr_{\scrE}\in C^{\infty}(\scrM,(\overline{\scrE}\otimes\scrE)^{\ast}),\qquad \Sl\cdot\vert\cdot\Sr_{\scrE_{x}}\colon\scrE_{x}\times\scrE_{x}\to\C,\qquad x\in\scrM\, .
\end{align*}
The bundle metric $\Sl\cdot\vert\cdot\Sr_{\scrE}$ induces a non-degenerate Hermitian sesquilinear form on the level of sections, which we shall denote by
\begin{align}\label{eq:SectionMetric}
	\bra\psi\vert\varphi\ket_{\scrE}:=\int_{\scrM}\Sl\psi(x)\vert\varphi(x)\Sr_{\scrE_{x}}\,d\mu_{\scrM}(x)
\end{align}
for all $\psi,\varphi\in C^{\infty}(\scrM,\scrE)$ with compactly overlapping support, where $d\mu_{\scrM}$ denotes the Lorentzian volume measure on $(\scrM,g)$. The \textit{formal adjoint} of a linear differential operator $\scrS\colon C^{\infty}(\scrM,\scrE)\to C^{\infty}(\scrM,\scrE)$ with respect to $\bra\cdot\vert\cdot\ket_{\scrE}$ is denoted by $\scrS^{\dagger}$.

To clarify our convention, we recall that the \textit{principal symbol} $\sigma_{\scrS}\colon T^{\ast}\scrM\to\mathrm{End}(\scrE)$ of a linear first-order differential operator $\scrS\colon C^{\infty}(\scrM,\scrE)\to C^{\infty}(\scrM,\scrE)$ can be defined in accordance to the Leibniz rule
\begin{align*}\scrS(f\psi)=f\scrS\psi+\sigma_{\scrS}(df)\psi
\end{align*}
for all $f\in C^{\infty}(\scrM)$ and $\psi\in C^{\infty}(\scrM,\scrE)$, as in \cite[Sec.~5]{baergreen}. Since $\scrS$ is a \textit{first-order} differential operator, we note that $\sigma_{\scrS}(\xi)\in\mathrm{End}(\scrE_{x})$ is also $\R$-linear in $\xi\in T^{\ast}_{x}\scrM$ for all $x\in\scrM$, i.e.~$\sigma_{\scrS}$ defines a vector bundle morphism from $T^{\ast}\scrM$ to $\mathrm{End}(\scrE)$ and hence can be viewed as a section of $\mathrm{End}(\scrE)\otimes T\scrM$. 

Furthermore, let us choose a (metric) connection $\nabla^{\scrE}$ on $(\scrE,\Sl\cdot|\cdot\Sr_{\scrE})$ and a local coordinate chart $(x^{\mu})_{\mu=0,\dots,k}$ on some open subset of $\scrM$. Then, using the simplified notation $\nabla^{\scrE}_{\mu}:=\nabla^{\scrE}_{\partial_{\mu}}$, we can write $\scrS$ locally in the form
\begin{align*}
	\scrS=\sigma_{\scrS}(d x^{\mu})\nabla^{\scrE}_{\mu}-\scrS_{0}\, ,
\end{align*}
where we used the \textit{Einstein summation convention} and where $\scrS_{0}\in C^{\infty}(\scrM,\mathrm{End}(\scrE))$ is a zero-order operator, which only depends on the choice of connection on $\scrE$. 

In this article, we are mainly interested in studying linear first-order differential operators of the following type:

\begin{Def}[Symmetric Hyperbolic Systems]\label{Def:SHS}
	 A \emph{symmetric hyperbolic system} is a linear first-order differential operator $\scrS\colon C^{\infty}(\scrM,\scrE)\to C^{\infty}(\scrM,\scrE)$ whose principal symbol $\sigma_{\scrS}\colon T^{\ast}\scrM\to\mathrm{End}(\scrE)$ has the following two properties:
	\begin{itemize}
		\item[(S)]$\sigma_{\scrS}(\xi)\in\mathrm{End}(\scrE_{x})$ is Hermitian with respect to $\Sl\cdot|\cdot\Sr_{\scrE_{x}}$ for all $\xi\in T^{\ast}_{x}\scrM,\,x\in\scrM$;
		\item[(H)]The sesquilinear form $\Sl\sigma_{\scrS}(\tau)\cdot\vert\cdot\Sr_{\scrE_{x}}$ is positive definite for every future-directed\footnote{A timelike covector $\xi\in T^{\ast}\scrM$ is \textit{future/past-directed} if $\xi(v)>0$ for all future/past-directed $v\in T\scrM$. In particular, the musical isomorphism map future-directed into past-directed objects and vice versa.} timelike $\tau\in T_{x}^{\ast}\scrM$ and all $x\in\scrM$.
	\end{itemize}
\end{Def}

\begin{Remark}
	More generally, one may consider \emph{weakly symmetric hyperbolic systems}, in which condition (H) in Definition~\ref{Def:SHS} is relaxed, see e.g.~\cite[Sec.~2.5]{DiracMIT}.
\end{Remark}

The following proposition shows how the formal adjoint $\scrS^{\dagger}$ of a first-order operator $\scrS$ can be written in terms the principal symbol and zero-order part of $\scrS$.

\begin{Prp}[Formal Adjoint of a Symmetric Hyperbolic System]\label{Prop:AdjointSHS} 
	Consider a symmetric hyperbolic system 
	\begin{align*}
		\scrS=\sigma_{\scrS}(d x^{\mu})\nabla^{\scrE}_{\mu}-\scrS_{0}
	\end{align*}		
	on a Hermitian bundle $(\scrE,\Sl|\cdot\Sr_{\scrE})$ with metric connection $\nabla^{\scrE}$. Let $\nabla^{\mathrm{End}(\scrE)}$ be the connection on $\mathrm{End}(\scrE)$ induced by $\nabla^{\scrE}$ and denote by $\mathscr{D}$ the divergence-like operator defined as 
	\begin{align*}
		\mathscr{D}\colon C^{\infty}(\scrM,\mathrm{End}(\scrE)\otimes T\scrM)\xrightarrow{\nabla^{\mathrm{End}(\scrE)}\otimes\nabla} C^{\infty}(\scrM,\mathrm{End}(\scrE)\otimes T\scrM\otimes T^{\ast}\scrM)\xrightarrow{c}C^{\infty}(\scrM,\mathrm{End}(\scrE))\, ,
	\end{align*}
where $c\colon T\scrM\otimes T^{\ast}\scrM\ni v\otimes \xi\mapsto \xi(v)=v^{\mu}\xi_{\mu}$ denotes the dual pairing. Then, the formal adjoint $\scrS^{\dagger}$ of $\scrS$ with respect to $\bra\cdot\vert\cdot\ket_{\scrE}$ is given by
	\begin{align*}
		\scrS^{\dagger}=-\sigma_{\scrS}(d x^{\mu})\nabla^{\scrE}_{\mu}-\scrS_{0}^{\prime}-\scrS_{0}^{\dagger}\, ,
	\end{align*}
	where $\scrS_{0}^{\prime}:=\mathscr{D}\sigma_{\scrS}$ is a zero-order differential operator.
\end{Prp}

\begin{proof}
	Fix $\psi,\varphi\in C^{\infty}_{\mathrm{c}}(\scrM,\scrE)$ and consider the $k$-form $\omega\in\Omega^{k}(\scrM,\C)$ defined by
	\begin{align*}
		\omega:=\Sl\sigma_{\scrS}(dx^{\mu})\psi|\varphi\Sr_{\scrE} \partial_{\mu}\lrcorner d\mu_{\scrM}\,.
	\end{align*}
	Defining the vector field $V:=\Sl\sigma_{\scrS}(dx^{\mu})\psi|\varphi\Sr_{\scrE} \partial_{\mu}\in C^{\infty}(\scrM,T\scrM)$, then $V\lrcorner d\mu_{\scrM}$ denotes the insertion of $V$ into the first slot of the volume form $d\mu_{\scrM}$. Cartan's magic formula implies 
	\begin{align*}
		d\omega=d(V\lrcorner d\mu_{\scrM})=\mathcal{L}_{V}d\mu_{\scrM}=\mathrm{div}_{g}(V)d\mu_{\scrM}\, ,
	\end{align*}	
	where $\mathcal{L}_{V}$ denotes the Lie derivative along $V$ and $\mathrm{div}_{g}(V)$ the divergence of $V$. Furthermore, using the fact that the connection $\nabla^{\scrE}$ is metric with respect to $\Sl\cdot|\cdot\Sr_{\scrE}$, a straightforward computation yields
	\begin{align*}
		\mathrm{div}_{g}(V)=&\partial_{\mu}\Sl \sigma_{\scrS}(dx^{\mu})\psi|\varphi\Sr_{\scrE}+\Gamma_{\mu\nu}^{\mu}\Sl \sigma_{\scrS}(dx^{\nu})\psi|\varphi\Sr_{\scrE}\\=&\Sl \nabla^{\scrE}_{\mu}(\sigma_{\scrS}(dx^{\mu})\psi)|\varphi\Sr_{\scrE}+\Sl \sigma_{\scrS}(dx^{\mu})\psi|\nabla^{\scrE}_{\mu}\varphi\Sr_{\scrE}+\Gamma_{\mu\nu}^{\mu}\Sl \sigma_{\scrS}(dx^{\nu})\psi|\varphi\Sr_{\scrE}\\=&\Sl \underbrace{(\nabla^{\mathrm{End}(\scrE)}_{\mu}\sigma_{\scrS}(dx^{\mu})+\Gamma_{\mu\nu}^{\mu}\sigma_{\scrS}(dx^{\nu}))\psi}_{=:\scrS_{0}^{\prime}\psi}+\sigma_{\scrS}(dx^{\mu})\nabla_{\mu}^{\scrE}\psi|\varphi\Sr_{\scrE}+\Sl \psi|\sigma_{\scrS}(dx^{\mu})\nabla^{\scrE}_{\mu}\varphi\Sr_{\scrE}\\=&\Sl\psi\vert\scrS\varphi\Sr_{\scrE}+\Sl(\sigma_{\scrS}(d x^{\mu})\nabla^{\scrE}_{\mu}+\scrS_{0}^{\prime}+\scrS_{0}^{\dagger})\psi\vert\varphi\Sr_{\scrE} \, ,
	\end{align*}
	where we added and subtracted $\Sl\psi\vert\scrS_{0}\varphi\Sr_{\scrE}$ in the last step and used the fact that the adjoint relation for $\scrS_{0}$ actually holds pointwise, i.e.~$\Sl\psi\vert\scrS_{0}\varphi\Sr_{\scrE}=\Sl\scrS_{0}^{\dagger}\psi\vert\varphi\Sr_{\scrE}$, since it is an operator of order zero. Combining everything, we have shown that
	\begin{align}\label{eq:AdjointProof}
		d\omega=\Sl\psi\vert\scrS\varphi\Sr_{\scrE}-\Sl\scrS^{\dagger}\psi\vert\varphi\Sr_{\scrE}\, .
	\end{align}
	where $\scrS^{\dagger}:=-\sigma_{\scrS}(d x^{\mu})\nabla^{\scrE}_{\mu}-\scrS_{0}^{\prime}-\scrS_{0}^{\dagger}$. Integrating equation~\eqref{eq:AdjointProof} over $\scrM$ and using Stokes' theorem, we conclude that $\scrS^{\dagger}$ is indeed the formal adjoint of $\scrS$. 
\end{proof}

\begin{Remark}\label{Remark:Zero}	
	As a by-product, Proposition~\ref{Prop:AdjointSHS} implies that $\scrS+\scrS^{\dagger}$ is an operator of order zero and
	\begin{align*}
		\Sl(\scrS+\scrS^{\dagger})\psi\vert\psi\Sr_{\scrE}=-2\Re \Sl\scrS_{0}\psi\vert\psi\Sr_{\scrE}-\Sl\scrS^{\prime}_{0}\psi\vert\psi\Sr_{\scrE}
	\end{align*}
	for all $\psi\in C^{\infty}(\scrM,\scrE)$. In particular, since $\Sl\scrS^{\prime}_{0}\psi\vert\psi\Sr_{\scrE}$ is clearly real-valued, it follows that $\Sl(\scrS+\scrS^{\dagger})\psi\vert\psi\Sr_{\scrE}$ is real-valued as well. 
\end{Remark}

The following basic examples shows how Definition~\ref{Def:SHS} is related to the notion of symmetric hyperbolic systems in the usual PDE sense (see e.g.~\cite{friedrichs, friedrichs2}).

\begin{Example}\label{Ex:MinkSHS}
	Consider $(1+k)$-dimensional Minkowski spacetime $\scrM=\mathbb{R}\times\mathbb{R}^{k}$ and let $\scrE:=\scrM\times \C^N$ be the trivial vector bundle equipped with the standard inner product on its fibres. Any linear first-order differential operator $\scrS\colon C^{\infty}(\scrM,\scrE) \to C^{\infty}(\scrM,\scrE)$ can be written as
	\begin{align*} 
		\scrS:=  A^0 \partial_t + \sum_{j=1}^k A^j\partial_{j} - C 
	\end{align*}
	with coefficients $A^0, A^j,C\in C^{\infty}(\scrM,\C^{N\times N})$. In the notation used before, it holds that $\sigma_{\scrS}(dt)=A^{0}$ and $\sigma_{\scrS}(dx^{j})=A^{j}$ as well as $\scrS_{0}=C$. In these coordinates, Condition~(S) in Definition~\ref{Def:SHS} hence reduces to 
	\begin{align*}
		\text{(S)}&\qquad A^0=(A^0)^\dagger \qquad \text{and} \qquad A^j=(A^j)^\dagger
	\end{align*}
	for $j=1,\dots, k$, where $\dagger$ is the (pointwise) adjoint matrix, while (H) can be stated as 
	\begin{align*}
		\text{(H)}\qquad A^0 + \sum_{j=1}^{k} \alpha_j A^j \quad\text{ is positive definite for \quad
$\sum_{j=1}^{k} \alpha_j^2 <1$}\, ,
	\end{align*}
	since any timelike future-directed covector is (up to multiplication by a positive function) of the form $\tau=dt+\sum_{j=1}^{k}\alpha_{j}d x^{j}$ with coefficients $\{\alpha_{j}\}_{j}$ such that $\sum_{j}\alpha_{j}^{2}<1$. The latter requirement seems a bit stronger to the one found in the classical PDE literature, in which one usually just requires $A^{0}$ to be positive-definite. However, since the principal symbol $\sigma_{\scrS}(\xi)=A^{\mu}\xi_{\mu}$ depends smoothly on $\xi$, it is clear that positive definiteness of $\sigma_{\scrS}(d t)=A^{0}$ extends to small perturbations of $dt$, i.e.~to $\sigma_{\scrS}(d t+\alpha_{i}d x^{i})=A^{0}+\alpha_{j}A^{j}$ for sufficiently small $\{\alpha_{j}\}_{j}$. 
\end{Example}

As another example, also \textit{normally hyperbolic operators}, namely second-order differential operators, which up to the addition of a zero-order term are given by the d'Alembertian of some connection, can be reduced to symmetric hyperbolic systems, see~\cite[Sec.~1.5]{baer+ginoux} or \cite[Sec.~6.3]{FriedrichsTimelikeBoundary} for details. Furthermore, also Maxwell's equations on ultrastatic manifolds provide an example of a symmetric hyperbolic system (see Section~\ref{sec:Maxwell}). Last but not least, as we shall see in Section~\ref{sec:Dirac}, a prototypical example of a symmetric hyperbolic system is the \textit{Dirac operator}. In this setting, the naturally defined fibre metric on the spinor bundle is indefinite. However, it turns out that assuming the fibre metric of a symmetric hyperbolic system to be positive-definite is not a loss of generality. 
 
\begin{Lemma}[\protect{\cite[Lemma~2.8]{FriedrichsTimelikeBoundary}}]\label{lem:scalprod}
	Let $t\colon\scrM\to\R$ be a Cauchy temporal function and $\beta\in C^{\infty}(\scrM,(0,\infty))$ the corresponding lapse function, i.e.
	\begin{align*}
		\scrM=\R\times\scrN,\qquad g=-\beta^{2}dt\otimes dt+h_{t}
	\end{align*}
	for a one-parameter family $(h_{t})_{t\in\R}$ of Riemannian metrics on $\scrN$. Furthermore, let $\scrS\colon C^{\infty}(\scrM,\scrE)\to C^{\infty}(\scrM,\scrE)$ be a symmetric hyperbolic system over $(\scrM,g)$. Then, $\scrS_\beta:=\sigma_\scrS(dt)^{-1} \scrS$ is a symmetric hyperbolic system with respect to the positive-definite Hermitian fibre metric
	\begin{align}\label{eq:scalarprod}
\Sl\cdot|\cdot\Sr_{\beta,\scrE}:=\Sl{\sigma_{\scrS}(\eta)\cdot}\vert{\cdot}\Sr_\scrE\, ,
	\end{align}
	where $\eta:=\beta dt$ is the future-directed normal covector of $\scrN$ in $\scrM$. Moreover, the Cauchy problems for $\scrS_\beta$ and $\scrS$ are equivalent.
\end{Lemma}

It turns out that the Cauchy problem for a symmetric hyperbolic systems on a globally hyperbolic manifolds is well-posed.

\begin{Thm}[\protect{\cite[Cor.~5.4 and 5.5, Thm.~5.6 and Prop.~5.7]{baergreen}}]\label{Thm:WellPosedCauchy}
	Let $\scrS\colon C^{\infty}(\scrM,\scrE)\to C^{\infty}(\scrM,\scrE)$ be a symmetric hyperbolic system over a globally hyperbolic manifold $(\scrM,g)$ and let $t\colon\scrM\to\R$ be a Cauchy temporal function with corresponding foliation $(\scrN_{t})_{t\in\R}$. For every $t_{0}\in\R$ and every pair $(\phi,\mathfrak{f})\in C_{\mathrm{c}}^{\infty}(\scrM,\scrE)\times C_{\mathrm{c}}^{\infty}(\scrN_{t_{0}},\scrE|_{\scrN_{t_{0}}})$, there exists a unique solution $\psi\in C_{\mathrm{sc}}^{\infty}(\scrM,\scrE)$ of the Cauchy problem
	\begin{align}\label{Cauchysmooth}
		\begin{cases}{}
			{\scrS}\psi &=\phi \\
			\psi_{|_{\scrN_{t_{0}}}} &= \mathfrak{f}
		\end{cases} 
	\end{align}
	and the solution map $(\phi,\mathfrak{f})\mapsto\psi$ is continuous in the respective LF-space topologies. Furthermore, it holds that
	\begin{align*}
		\mathrm{supp}(\psi)\cap J^{\pm}(\scrN_{t_0})\subset J^{\pm}((\mathrm{supp}(\phi)\cap J^{\pm}(\scrN_{t_0}))\cup\mathrm{supp}(\mathfrak{f}))\, ,
	\end{align*}
	which in particular implies that $\psi$ propagates with at most the speed of light, i.e. 
	\begin{align*}
		\mathrm{supp}(\psi)\subset J(\mathrm{supp}(\phi)\cup\mathrm{supp}(\mathfrak{f}))\, .
	\end{align*}
\end{Thm}

As a by-product of the well-posedness of the Cauchy problem, one obtains the existence of the advanced and the retarded Green operators.

\begin{Prp}[\protect{\cite[Thm.~5.9, Cor.~3.11]{baergreen}}]\label{prop:Green}
	For any symmetric hyperbolic system $\scrS$ on a globally hyperbolic manifold $(\scrM,g)$, there exist linear maps 
	\begin{align*}
		G_{\mathrm{ret}/\mathrm{adv}}\colon C^{\infty}_{\mathrm{c}}(\scrM,\scrE)\to C^{\infty}_{\mathrm{sc}}(\scrM,\scrE)\, ,
	\end{align*}
	called \emph{retarded and advanced Green operators}, with the following properties:
	\begin{align*}
		\text{(i)}&\quad G_{\mathrm{ret}/\mathrm{adv}}\circ\scrS= \scrS\circ G_{\mathrm{ret}/\mathrm{adv}}=\mathrm{id}\quad\text{on}\quad  C^{\infty}_{\mathrm{c}}(\scrM,\scrE)\\
		\text{(ii)}&\quad \mathrm{supp}(G_{\mathrm{ret}}\varphi)\subset J^{+}(\mathrm{supp}(\varphi))\quad\text{and}\quad \mathrm{supp}(G_{\mathrm{adv}}\varphi)\subset J^{-}(\mathrm{supp}(\varphi))
	\end{align*}
	Furthermore, the retarded and advanced Green operators are unique and continuous.
\end{Prp}

Since the Cauchy problem for symmetric hyperbolic systems is well-posed, the space $\mathrm{ker}(\scrS\vert_{C^{\infty}_{\mathrm{sc}}})$ of solutions with spatially compact supports is isomorphic to the space of initial data $C^{\infty}_{\mathrm{c}}(\scrN,\scrE\vert_{\scrN})$ on a fixed smooth (spacelike) Cauchy hypersurface $\scrN$ of $\scrM$. We equip this space with the (positive-definite) inner product
\begin{align}\label{eq:InnerProduct} 
	(\psi | \phi)_{\scrN} := \int_\scrN \Sl \sigma_{\scrS}(\eta) \psi | \phi\Sr_{\scrE}\: d\mu_{\scrN}\:,
\end{align}
where $\eta$ denotes the \textit{future-directed} timelike normal covector of $\scrN$ and $d\mu_{\scrN}$ the volume measure on the Riemannian manifold $(\scrN,i^{\ast}g)$ with $i\colon\scrN\to\scrM$ denoting the natural embedding.

Now, if $t\colon\scrM\to\R$ is a Cauchy temporal function for $(\scrM,g)$ with corresponding lapse function $\beta\in C^{\infty}(\scrM,(0,\infty))$ such that
\begin{align*}
		\scrM=\R\times\scrN,\qquad g=-\beta^{2}dt\otimes dt+h_{t}
\end{align*}
for a one-parameter family $(h_{t})_{t\in\R}$ of Riemannian metrics on $\scrN$, the future-directed timelike normal vector of $\scrN_{t}:=\{t\}\times\scrN$ in $\scrM$ is $\nu_{t}=\beta^{-1}\partial_{t}\vert_{\scrN_{t}}$ and hence $\eta_{t}:=-\nu_{t}^{\flat}=\beta dt\vert_{\scrN_{t}}$ is the corresponding future-directed timelike normal covector. In particular, for any $t\in\R$, we use the inner product \eqref{eq:InnerProduct} to obtain a natural family of Hilbert spaces on the foliation $(\scrN_{t})_{t\in\R}$, i.e.
\begin{align}\label{eq:Ht}
	\H_t:=\overline{C^{\infty}_{\mathrm{c}}(\scrN_{t},\scrE\vert_{\scrN_{t}})}^{\Vert\cdot\Vert_{t}},\qquad (\psi | \phi)_t := \int_{\scrN_{t}} \Sl \psi | \phi \Sr_{\beta,\scrE}\: d\mu_{\scrN_{t}} \: ,
\end{align}
where $d\mu_{\scrN_{t}}$ is the Riemannian volume measure of $(\scrN_{t},h_{t})$ and where $\Sl\psi|\phi\Sr_{\beta,\scrE}:=\Sl \sigma_{\scrS}(\eta)\psi | \phi \Sr_{\scrE}$, as in Lemma~\ref{lem:scalprod}. In other words, using the notation of Section~\ref{SobolevSpaces}, 
\begin{align*}
	\H_{t}=L^{2}(\scrN_{t},\scrE_{\beta}\vert_{\scrN_{t}})\,,
\end{align*} 
where the subscript $\beta$ in $\scrE_{\beta}$ indicates that $\scrE$ is equipped with the positive-definite bundle metric $\Sl\cdot|\cdot\Sr_{\beta,\scrE}$. For a fixed time $T>0$, we consider the corresponding \textit{open} and \textit{closed time strip} defined by
\begin{align*}
	\scrM_{T}:=t^{-1}((0,T))\cong (0,T)\times\scrN\,,\qquad \overline{\scrM}_{T}:=t^{-1}([0,T])\cong [0,T]\times\scrN\, .
\end{align*}
On such a time strip, we introduce the Hilbert space
\begin{align}\label{eq:HMt}
	\H_{\scrM_{T}}:=\int^{\oplus}_{(0,T)}\H_t\,dt,\qquad \qquad (\psi | \phi)_{\scrM_{T}} = \int_{0}^{T} (\psi_{t} | \phi_{t})_t\, dt\, .
\end{align}
Introducing the auxiliary \textit{Riemannian} metric $g_{0}:=dt\otimes dt+h_{t}$ on $\scrM$, Fubini's theorem implies that the Hilbert space $\H_{\scrM_{T}}$ can be identified with the $L^{2}$-space
\begin{align*}
	\H_{\scrM_{T}}\cong L^{2}(\scrM_{T},\scrE_{\beta})\, ,
\end{align*}
where $\scrM_{T}$ is equipped with the auxiliary Riemannian metric $g_{0}$. We will also need the corresponding Sobolev spaces defined on the open time strip, which following the notation of Section~~\ref{SobolevSpaces}, will be denoted by $W^{2,i}(\scrM_{T},\scrE_{\beta})$ for $i\in\N$.
\section{Nonlocal Potentials}
\label{Sec:NonLocPot}
In this section, we establish the framework for nonlocal potentials on globally hyperbolic Lorentzian manifolds. We begin with some preliminaries on distributions and kernels for operators acting on vector bundles. We then introduce the concept of time kernels, allowing us to interpret a nonlocal potential as a family of linear operators acting on the Hilbert spaces associated with the Cauchy surfaces of a given foliation.

\subsection{Preliminaries on Distributions and Kernels}
Let $(\scrM,g)$ be a smooth pseudo-Riemannian manifold and $(\scrE,\Sl\cdot|\cdot\Sr_{\scrE})$ be a Hermitian vector bundle. We equip the spaces of test sections $C_{\rm{c}}^{\infty}(\scrM,\scrE)$ and $C_{\rm{c}}^{\infty}(\scrM,\scrE^{\ast})$ with their natural LF-space topologies and define the space of distributions on $\scrM$ in $\scrE$ as the \textit{topological} dual space
\begin{align*}
	\mathcal{D}^{\prime}(\scrM,\scrE):=(C^{\infty}_{\mathrm{c}}(\scrM,\scrE^{\ast}))^{\prime}\, ,\end{align*}
equipped with its strong dual topology. The space of smooth sections is embedded in the space of distributions via
\begin{align}\label{eq:regular}
	C^{\infty}(\scrM,\scrE)\hookrightarrow\mathcal{D}^{\prime}(\scrM,\scrE),\qquad \psi\mapsto\bigg(C^{\infty}_{\mathrm{c}}(\scrM,\scrE^{\ast})\ni\varphi\mapsto \int_{\scrM}\varphi(\psi)\,d\mu_{\scrM}\bigg)
\end{align}
and we shall call the distributions contained in the image of this map \textit{regular}. In the same way, also sections of lower regularity such as $C^{0}(\scrM,\scrE)$ and $L^1_{\mathrm{loc}}(\scrM,\scrE)$ can be embedded into $\mathcal{D}^{\prime}(\scrM,\scrE)$. We remark that we will always use the volume measure $d\mu_{\scrM}$ on $(\scrM,g)$ for the above embedding as opposed to some authors which allow for different measures by defining $\mathcal{D}^{\prime}(\scrM,\scrE)$ as the dual of \textit{densitized} sections.

If $\scrE$ and $\scrF$ are two $\C$-bundles over $\scrM$ and $\mathrm{pr}_{1,2}\colon\scrM\times\scrM\to\scrM$ the natural projections onto the first and second entry, respectively, we denote by
\begin{align*}
	\scrE\boxtimes\scrF:=(\mathrm{pr}_{1}^{\ast}\scrE)\otimes (\mathrm{pr}^{\ast}_{2}\scrF)
\end{align*}
the \textit{external tensor product}, i.e.~the $\C$-bundle over $\scrM\times\scrM$ with fibre $\scrE_{x}\otimes\scrF_{y}$ at $(x,y)\in\scrM\times\scrM$. The space of sections of $\scrE\boxtimes\scrF$ is locally generated by sections of the form $(\varphi\boxtimes\psi)(x,y):=\varphi(x)\otimes\psi(y)$ for $\varphi\in C^{\infty}(\scrM,\scrE)$ and $\psi\in C^{\infty}(\scrM,\scrF)$. 

For the sake of completeness and to fix the notation, let us recall the well-known \textit{kernel theorem of Schwartz} in this setting:

\begin{Thm}[\protect{\cite[Thm.~1.5.1]{Tarkhanov}, \cite[Thm.~51.7]{treves}}]\label{Thm:Schwartz} 
	A $\mathbb{C}$-linear operator of the form $\B\colon C^{\infty}_{\rm{c}}(\scrM,\scrE)\to\mathcal{D}^{\prime}(\scrM,\scrE)$ is continuous if and only if it has a kernel, i.e.~a distribution $k_{\B}\in \mathcal{D}^{\prime}(\scrM\times\scrM,\scrE\boxtimes\scrE^{\ast})$ such that
	\begin{align}\label{Eq:Schwartz} 
		\langle \B\varphi,\psi\rangle_{\scrM}=\langle k_{\B},\psi\boxtimes\varphi\rangle_{\scrM\times\scrM}
	\end{align}
	for all $\varphi\in C_{\mathrm{c}}^{\infty}(\scrM,\scrE)$ and $\psi\in C^{\infty}_{\mathrm{c}}(\scrM,\scrE^{\ast})$, where $\langle\cdot,\cdot\rangle_{\scrM}$ and $\langle\cdot,\cdot\rangle_{\scrM\times\scrM}$ denote the bilinear distributional pairings on $\scrM$ and $\scrM\times\scrM$, respectively.
\end{Thm}

A distributional kernel $k_{\B}\in \mathcal{D}^{\prime}(\scrM\times\scrM,\scrE\boxtimes\scrE^{\ast})$ is called \textit{semi-regular}, if it lives in the subspace 
\begin{align*}
	C^{\infty}(\scrM,\scrE)\hat{\otimes}\mathcal{D}^{\prime}(\scrM,\scrE^{\ast})\subset \mathcal{D}^{\prime}(\scrM,\scrE)\hat{\otimes}\mathcal{D}^{\prime}(\scrM,\scrE^{\ast})\cong \mathcal{D}^{\prime}(\scrM\times\scrM,\scrE\boxtimes\scrE^{\ast})\,,
\end{align*} 
where $\hat{\otimes}$ denotes the tensor product of nuclear spaces. In this case, the corresponding linear and continuous operator $\B$ is of the form $\B\colon C^{\infty}_{\rm{c}}(\scrM,\scrE)\to C^{\infty}(\scrM,\scrE)$. If the two projections $\scrM\times\scrM\supset\mathrm{supp}(k_{\B})\to\scrM$ are in addition \textit{proper} maps, i.e.~if preimages of compact sets are compact, we call $\B$ \textit{properly supported} and in this case, it holds that $\mathrm{ran}(\B)\subset C^{\infty}_{\rm{c}}(\scrM,\scrE)$. Furthermore, a properly supported operator extents continuously to an operator from $C^{\infty}(\scrM,\scrE)$ to $C^{\infty}(\scrM,\scrE)$.

If the operator $\B$ is a \textit{smoothing operator}, which means that its kernel $k_{\B}$ is a regular distribution $k_{\B}\in C^{\infty}(\scrM\times\scrM,\scrE\boxtimes\scrE^{\ast})$, we obtain $\C$-linear maps $k_{\B}(x,y)\colon\scrE_{y}\to\scrE_{x}$ for all $x,y\in\scrM$. Furthermore, using the (anti-linear) bijection $\varphi\mapsto\Sl\varphi|\cdot\Sr_{\scrE}$ to identify sections of $\scrE$ with sections of $\scrE^{\ast}$ and vice versa, the defining relation~\eqref{Eq:Schwartz} reads
\begin{align*}
	\int_{\scrM}\Sl\psi(x)|\B\varphi(x)\Sr_{\scrE_{x}}\,d\mu_{\scrM}(x)=\int_{\scrM}\int_{\scrM}\Sl\psi(x)| k_{\B}(x,y)\varphi(y)\Sr_{\scrE_{x}}\,d\mu_{\scrM}(y)\,d\mu_{\scrM}(x)
\end{align*}
for all $\varphi,\psi\in C^{\infty}_{\mathrm{c}}(\scrM,\scrE)$. We refer to \cite{Tarkhanov,treves} for more details on distributions and kernels.

\subsection{Time Kernels}\label{subsec:TimeKernel}
Let us now consider a globally hyperbolic Lorentzian manifold 
\begin{align*}
	\scrM=\R\times\scrN,\qquad g=-\beta^{2}dt\otimes dt+h_{t}
\end{align*}
with corresponding foliation $(\scrN_{t}:=\{t\}\times\scrN)_{t\in\R}$ and let us denote by $i_{t}\colon\scrN_{t}\to\scrM$ the natural embeddings. In this setting, it is convenient to introduce the concept of a \textit{time kernel}. As motivational example, let us consider the following simple case:
\begin{itemize}
	\item[$\bullet$]The trivial line bundle $\scrE=\scrM\times\mathbb{C}$ equipped with standard inner product on its fibres.
	\item[$\bullet$]A continuous kernel $k_{\B}\in C^{0}(\scrM\times\scrM)$ with corresponding integral operator 
	\begin{align*}
		\B\psi(x):=\int_{\scrM}k_{\B}(x,y)\psi(y)\,d\mu_{\scrM}(y),\qquad\forall\psi\in C^{\infty}_{\mathrm{c}}(\scrM)\, .
	\end{align*}
\end{itemize}
Now, if instead of integrating the kernel $k_{\B}(x,y)$ over \textit{spacetime} we just integrate over the \textit{spatial} variables, we obtain a two-parameter family of linear and continuous operators $\B_{t,\tau}:C^{\infty}_{\mathrm{c}}(\scrN_\tau)\to C^{\infty}(\scrN_{t})$ defined by
\begin{align*}
	(\B_{t,\tau}\varphi)(\vec{x}):=\int_{\scrN_{\tau}}k_{\B}(t,\vec{x};\tau,\vec{y})\beta(\tau,\vec{y})\varphi(\vec{y})d\mu_{\scrN_{\tau}}(\vec{y})\,,\quad\forall\varphi\in C^{\infty}_{\mathrm{c}}(\scrN_{\tau})\, .
\end{align*}
In other words, the operator $\B_{t,\tau}$ for $t,\tau\in\R$ is the integral operator with kernel 
\begin{align*}
	k_{t,\tau}:=(\mathrm{pr}_{2}^{\ast}\beta_{\tau})(i_{t}\times i_{\tau})^{\ast}k_{\B}\,,\end{align*} 
where $\beta_{\tau}(\cdot):=\beta(\tau,\cdot)$ and where $\mathrm{pr}_{2}\colon \scrN_{t}\times\scrN_{\tau}\to\scrN_{\tau}$ denotes the projection onto the second factor. With this definition, it follows from Fubini's theorem as well as from $d\mu_{\scrM}(t,\vec{x})=\beta(t,\vec{x})d\mu_{\scrN_{t}}(\vec{x})dt$ that the nonlocal potential $\B$ can be obtained from $\B_{t,\tau}$ by integrating over time, i.e.
\begin{align}\label{eq:time-kernel}
	(\B\psi)_{t}(\vec{x})=\int_{\R} (\B_{t,\tau}\psi_{\tau})(\vec{x})\, d\tau\, ,
 \end{align}		
for all $\psi\in C^{\infty}_{\mathrm{c}}(\scrM,\scrE)$, where we used the obvious notation $\psi_{\tau}(\cdot):=\psi(\tau,\cdot)$. The two-parameter family of operators $(\B_{t,\tau})_{t,\tau\in\R}$ will be referred to as the \textit{time kernel} of the nonlocal potential $\B$.

Now, the aim of the following discussion is to generalize the notion of a time kernel in a broader distributional context. If $\scrE$ is an arbitrary $\C$-bundle and we consider a kernel $k_{\B}\in\mathcal{D}^{\prime}(\scrM\times\scrM,\scrE\boxtimes\scrE^{\ast})$, we would like to define the pull-back of $k_{\B}$ to $\scrN_{t}\times\scrN_{\tau}$ for $t,\tau\in\R$ via the natural embeddings $i_{t}$. However, a priori, it is not clear whether this can be defined in a consistent way for arbitrary bi-distributional kernels. As a simple toy model, we consider a distribution $u\in\mathcal{D}^{\prime}(\scrM,\scrE)$. The guiding question in the following discussion is for which class of distribution the following can be achieved:

\begin{Question}\label{Question}
	Given a distribution $u\in\mathcal{D}^{\prime}(\scrM,\scrE)$, can we find a one-parameter family $(u_{t})_{t\in\R}$ of distributions on $\scrN_{t}$ recovering $u$ in the sense that
	\begin{align*}
		\langle u,\varphi\rangle_{\scrM}=\int_{\R}\langle u_{t},\varphi_{t}\rangle_{\scrN_{t}}\,dt
	\end{align*}
	for all test sections $\varphi\in C^{\infty}_{\mathrm{c}}(\scrM,\scrE^{\ast})$?
\end{Question}

As a first example, if $u$ is sufficiently regular, i.e.~$u\in C^{0}(\scrM,\scrE)$, then we can define \begin{align}\label{eq:Pullback}
	u_{t}:=\beta_{t}(i_{t}^{\ast}u)\in C^{0}(\scrN_{t},\scrE\vert_{\scrN_{t}})\, ,
\end{align}	
where $i_{t}^{\ast}u:=u\circ i_{t}$ is the usual pull-back of a continuous function. Then, by Fubini's theorem,
\begin{align}\label{eq:Fubini}
	\langle u,\varphi\rangle_{\scrM}=\int_{\R}\langle u_{t},\varphi_{t}\rangle_{\scrN_{t}}\,dt\, .
\end{align}
More generally, for $u\in\mathcal{D}^{\prime}(\scrM,\scrE)$, we can always find a sequence $\{u_{n}\}_{n\in\N}\subset C^{\infty}_{\mathrm{c}}(\scrM,\scrE)$ converging to $u$ in $\mathcal{D}^{\prime}(\scrM,\scrE)$ and since $C^{\infty}_{\mathrm{c}}(\scrM,\scrE)$ is a \textit{Montel space}, convergence in the strong topology is equivalent to convergence in the weak topology, which implies
\begin{align*}
	\langle u_{n},\varphi\rangle_{\scrM}\xrightarrow{n\to\infty} \langle u,\varphi\rangle_{\scrM}\qquad\forall\varphi\in C^{\infty}_{\mathrm{c}}(\scrM,\scrE^{\ast})\, .
\end{align*}
Now, suppose that we can arrange $(u_{n})_{n\in\mathbb{N}}$ in such way that also $(i^{\ast}_{t}u_{n})_{n\in\N}$ converges in $\mathcal{D}^{\prime}(\scrN_{t},\scrE\vert_{\scrN_{t}})$. Denote the corresponding limit by 
\begin{align*}
	i_{t}^{\ast}u:=\lim_{n\to\infty}i_{t}^{\ast}u_{n}\qquad\text{in}\quad\mathcal{D}^{\prime}(\scrN_{t},\scrE\vert_{\scrN_{t}})\, .
\end{align*}
Now, it is clear that Equality~\eqref{eq:Fubini} will still hold in this more general situation whenever the right-hand side is well-defined. For instance, if for some arbitrary $u\in\mathcal{D}^{\prime}(\scrM,\scrE)$, the pull-back $i_{t}^{\ast}u$ can be defined as above and the map $t\mapsto \langle u_{t},\psi_{t}\rangle_{\scrN_{t}}$ with $u_{t}:=\beta_{t}i_{t}^{\ast}u$ is locally integrable such that
\begin{align*}
	\int_{\R}\langle \beta_{t}i^{\ast}_{t}u_{n},\psi_{t}\rangle_{\scrN_{t}}\,dt\xrightarrow{n\to\infty} \int_{\R}\langle u_{t},\psi_{t}\rangle_{\scrN_{t}}\,dt\, ,
\end{align*}
then, since $u_{n}$ is smooth, the left-hand side is nothing else than 
\begin{align*}
\int_{\R}\langle \beta_{t} i^{\ast}_{t}u_{n},\psi_{t}\rangle_{\scrN_{t}}\,dt =\langle u_{n},\psi\rangle_{\scrM}\xrightarrow{n\to\infty} \langle u,\psi\rangle_{\scrM}\, ,
\end{align*}
which establishes the validity of Equation~\eqref{eq:Fubini} by uniqueness of limits. Hence, the answer to Question~\ref{Question} is positive whenever the pull-back of $u$ via $i_{t}$ can be defined by means of an approximate sequence and if $t\mapsto \langle u_{t},\psi_{t}\rangle_{\scrN_{t}}$ is at least locally integrable.

An example of a condition that guarantees that this is the case is provided by a certain condition on the \textit{wavefront set}. Let $u \in \mathcal{D}^{\prime}(\scrM, \scrE)$ be a distribution on $\scrM$ satisfying the condition
\begin{align}\label{eq:WF}
   \bigg(\bigcup_{t\in\R}N^{\ast}\scrN_t \bigg)\cap \mathrm{WF}(u) = \emptyset \, ,
\end{align}
where $N^{\ast}\scrN_t \to \scrN_t$ denotes the \textit{co-normal bundle} over $\scrN_{t}$ in $\scrM$ and $\mathrm{WF}(u)$ the \textit{wavefront set} of $u$, which for distributions in $\scrE$ is defined as the union of the wavefront sets of its components in some fixed local trivialisation (the definition is independent of this choice), see e.g.~\cite[Chap.~8]{hormanderI} and \cite{WFSet}. If this condition is satisfied, then there is a unique way to define the pull-back $i_{t}^{\ast}u\in \mathcal{D}^{\prime}(\scrN_t, \scrE\vert_{\scrN_t})$ in such a way that 
\begin{align*}
	\mathrm{WF}(i^{\ast}_{t}u)\subset i_{t}^{\ast}\mathrm{WF}(u)\,,
\end{align*}	
see e.g.~\cite[Theorem~8.2.4, Corollary 8.2.7]{hormanderI}. In fact, the way $i^{\ast}_{t}u$ is defined is as follows: If $\Gamma\subset T^{\ast}\scrM\backslash\{0\}$ is an arbitrary closed cone such that $\Gamma\cap N^{\ast}\scrN_{t}=\emptyset$, we consider the subspace 
\begin{align}\label{eq:HormanderSpace}
	\mathcal{D}_{\Gamma}^{\prime}(\scrM,\scrE)=\{u\in\mathcal{D}^{\prime}(\scrM,\scrE)\mid \mathrm{WF}(u)\subset\Gamma\}
\end{align} 
equipped with the \textit{Hörmander topology}, see e.g.~\cite[Sec.~8.2]{hormanderI}.\footnote{The Hörmander topology is finer than the weak topology on $\mathcal{D}^{\prime}(\scrM,\scrE)$. In particular, if $u_{n}\xrightarrow{n\to\infty} u$ in the Hörmander topology, then also $\langle u_{n},\varphi\rangle_{\scrM}\to\langle u,\varphi\rangle_{\scrM}$ for all $\varphi\in C^{\infty}_{\mathrm{c}}(\scrM,\scrE^{\ast})$.}  Now, if $(u_{n})_{n\in\N}$ is a sequence in $C^{\infty}_{\mathrm{c}}(\scrM,\scrE)$ converging to $u$ in $\mathcal{D}^{\prime}_{\Gamma}(\scrM,\scrE)$, then in \cite[Theorem~8.2.4]{hormanderI} it is shown that $i_{t}^{\ast}u_{n}$ has a limit in the space $\mathcal{D}^{\prime}_{i_{t}^{\ast}\Gamma}(\scrM,\scrE)$, which provides a definition of the pull-back $i_{t}^{\ast}u$.

Now, if $u$ satisfies the condition~\eqref{eq:WF}, we can define the one-parameter family of distributions as in~\eqref{eq:Pullback} as
\begin{align}\label{eq:Family}
	u_{t}:=\beta_{t}(i_t^{\ast} u) \in \mathcal{D}^{\prime}(\scrN_t, \scrE\vert_{\scrN_t})\, ,
\end{align}
where $\beta_{t}(\cdot):=\beta(t,\cdot)\in C^{\infty}(\scrN_{t},(0,\infty))$. Now, in this case, Equation~\eqref{eq:Fubini} can be obtained as follows.

\begin{Prp}\label{Prop:Distr} 
	Let $u\in\mathcal{D}^{\prime}(\scrM,\scrE)$ be such that condition~\eqref{eq:WF} is satisfied and define $(u_{t})_{t\in\R}$ as above in~\eqref{eq:Family}. Then,
	\begin{align*}
		\R\ni t\mapsto \langle u_{t},\psi_{t}\rangle_{\scrN_{t}}\in\C
	\end{align*}
	is continuous for all $\psi\in C^{\infty}_{\mathrm{c}}(\scrM,\scrE^{\ast})$ and it holds that
	\begin{align*}
		\langle u,\psi\rangle_{\scrM}=\int_{\R}\langle u_{t},\psi_{t}\rangle_{\scrN_{t}}\,dt\, .
	\end{align*}
\end{Prp}

\begin{proof}
	As a first step, let us denote by $\delta_{t}\in\mathcal{D}(\scrM)$ the $\delta$-distribution on $\scrN_{t}\subset\scrM$, that is, the distribution defined by
	\begin{align*}
		\langle\delta_{t},\varphi\rangle_{\scrM}:=\int_{\scrN_{t}}\varphi(t,\vec{x})\,d\mu_{\scrN_{t}}(\vec{x})
	\end{align*}
	for all $\varphi\in C^{\infty}_{\mathrm{c}}(\scrM)$. It follows from standard arguments that the wavefront set of $\delta_{t}$ is given by $\mathrm{WF}(\delta_{t})=N^{\ast}\scrN_{t}\backslash\{0\}=:\Gamma$. Now, let $\Gamma^{\prime}\subset T^{\ast}\scrM\backslash\{0\}$ be a closed cone such that $\Gamma^{\prime}\cap N^{\ast}\scrN_{t}=\emptyset$ for all $t\in\R$. The multiplication theorem of Hörmander implies that the point-wise multiplication of smooth function extends to a continuous bilinear map $\mathcal{D}^{\prime}_{\Gamma^{\prime}}(\scrM)\times\mathcal{D}^{\prime}_{\Gamma}(\scrM,\scrE)\to\mathcal{D}^{\prime}(\scrM,\scrE)$, see~ \cite[Theorem~8.2.10]{hormanderI}. In particular, the product $\delta_{t} u$ is well-defined and it holds that
	\begin{align*}
		\langle u_{t},\psi_{t}\rangle_{\scrN_{t}}=\langle \beta_{t}(\delta_{t} u),\psi\rangle_{\scrM}\, .
	\end{align*}
	for all $\psi\in C^{\infty}_{\mathrm{c}}(\scrM,\scrE)$, by uniqueness of the pull-back. The map $t\mapsto \delta_{t}\in\mathcal{D}^{\prime}_{\Gamma}(\scrM)$ can easily be seen to be continuous, which shows continuity of $t\mapsto \langle u_{t},\psi_{t}\rangle_{\scrN_{t}}$. The fact that the integral of $\langle u_{t},\psi_{t}\rangle_{\scrN_{t}}$ coincides with $\langle u,\psi\rangle_{\scrM}$ then follows from the same arguments as above.
\end{proof}

Hence, condition~\eqref{eq:WF} provides another class of distribution for which the answer to Question~\ref{Question} is positive. However, we stress that \eqref{eq:WF} is a rather strong and it is clear that the pull-back can be defined for much more general distributions. For instance, whenever $u\in C^{0}(\scrM,\scrE)$, there is an obvious way to define $i_{t}^{\ast}u$, even though Condition~\eqref{eq:WF} is not necessarily satisfied.

\begin{Remark}[A Measure-Theoretic Point of View]
	The previous discussion is closely related to the \emph{disintegration of measures}. Let $u\in\mathcal{D}^{\prime}(\scrM)$ be a distribution of order zero, that is, $u$ continuously extends to a linear form on $C^{0}_{\mathrm{c}}(\scrM)$. By the \emph{theorem of Riesz-Markov}, there exists a complex Radon measure $\nu$ on the Borel $\sigma$-algebra of $\scrM$ such that 
	\begin{align*}
		\langle u,\varphi\rangle_{\scrM}=\int_{\scrM}\,\varphi\,d\nu
	\end{align*}
	for all $\varphi\in C^{0}_{\mathrm{c}}(\scrM)$, see e.g.~\cite[Thm.~6.19]{Rudin}. As an example, the $\delta$-distribution centred at $p\in\scrM$, i.e.~the distribution defined by $\langle\delta_{p},\varphi\rangle_{\scrM}:=\varphi(p)$, is of order zero and induces the \emph{Dirac measure} defined by $\delta_{p}(M):=\chi_{M}(p)$ for all Borel sets $M\in\mathcal{B}(\R)$, where $\chi_{M}$ denotes the characteristic function.
	
	 Now, by the \emph{disintegration theorem} (see e.g.~\cite[Chap.~45]{FremlinMeasures} for a detailed review on that subject), one can find a (non-unique) family of measures $\nu_{t}$ on $\scrN_{t}$ such that
	\begin{align*}
		\langle u,\varphi\rangle_{\scrM}=\int_{\scrM}\,\varphi\,d\nu=\int_{\R}\int_{\scrN_{t}}\varphi(t,\vec{x})\,d\nu_{t}(\vec{x})\,d\pi(t)\, ,
	\end{align*}
	where $\pi:=(\mathrm{pr}_{\R})_{\ast}\nu$ is the complex measure on $\R$ obtained by the push-forward via the projection $\mathrm{pr}_{\R}\colon\scrM\to\mathbb{R},\,(t,\vec{x})\mapsto t$. Using the Lebesgue decomposition theorem, we can uniquely write $\pi=\pi_{0}+\pi_{\mathrm{sing}}$, where $\pi_{0}$ is absolutely continuous and $\pi_{\mathrm{sing}}$ singular with respect to the Lebesgue measure $dt$ on $\R$, i.e.~$\pi_{0}\ll dt$ and $\pi_{\mathrm{sing}}\perp dt$. In particular, by the Radon-Nikodým theorem, there is a unique $h\in L^{1}(\R,dt)$, the \emph{Radon-Nikodým derivative} of $\pi$ with respect to $dt$, such that
	\begin{align*}
		\pi(M)=\int_{M} h\,dt+\pi_{\mathrm{sing}}(M)
	\end{align*}
	for all Borel sets $M\in\mathcal{B}(\R)$, see e.g.~\cite[Thm.~6.10]{Rudin}. Hence, in this case, the question whether the answer to Question~\ref{Question} is positive for some given (zero-order) distribution $u\in\mathcal{D}^{\prime}(\scrM)$ is related to the requirement that the singular part $\pi_{\mathrm{sing}}$ vanishes.	
\end{Remark}

After this preliminary discussion, let us come back to the notion of time kernels. 

\begin{Def}[Nonlocal Potentials]\label{Def:NonLocPot}
	Let $(\scrM,g)$ be a globally hyperbolic manifold with Cauchy temporal function $t\colon\scrM\to\R$  and $\B\colon C^{\infty}_{\mathrm{c}}(\scrM,\scrE)\to C^{\infty}(\scrM,\scrE)$ be a linear, continuous and semi-regular operator with Schwarz kernel $k_{\B}$ such that there exists a two-parameter family of distributional kernels
	\begin{align*}
		k_{t,\tau}\in\mathcal{D}^{\prime}(\scrN_{t}\times\scrN_{\tau},\scrE\vert_{\scrN_{t}}\boxtimes\scrE^{\ast}\vert_{\scrN_{\tau}}),\qquad t,\tau\in\R
	\end{align*}
	with corresponding operators $B_{t,\tau}\colon C^{\infty}_{\mathrm{c}}(\scrN_{\tau},\scrE\vert_{\scrN_{\tau}})\to C^{\infty}(\scrN_{t},\scrE\vert_{\scrN_{t}})$, such that the following holds for all $\psi\in C^{\infty}_{\mathrm{c}}(\scrM,\scrE^{\ast})$ and $\varphi\in C^{\infty}_{\mathrm{c}}(\scrM,\scrE)$:
	\begin{itemize}
		\item[(i)]The map
		\begin{align*}
			\R\times\R\ni (t,\tau)\mapsto \langle k_{t,\tau},\psi_{t}\boxtimes\varphi_{\tau}\rangle_{\scrN_{t}\times\scrN_{\tau}}=\langle \B_{t,\tau}\varphi_{\tau},\psi_{t}\rangle_{\scrN_{t}}\in\C
		\end{align*}
		is locally integrable.
		\item[(ii)]$\B$ is related to $(\B_{t,\tau})_{t,\tau\in\R}$ via
		\begin{align*}
			\langle(\B\varphi)_{t},\psi_{t}\rangle_{\scrN_{t}}=\int_{\R}\langle \B_{t,\tau}\varphi_{\tau},\psi_{t}\rangle_{\scrN_{t}}\,d\tau\, .
\end{align*}
	\end{itemize}
	Then, $\B$ is called a \emph{nonlocal potential} with time kernel $(k_{t,\tau})_{t,\tau\in\R}$.
\end{Def}

We will use the term \textit{time kernel} interchangeably for both the family of kernels $(k_{t,\tau})_{t,\tau\in\R}$ and linear operators $(\B_{t,\tau})_{t,\tau\in\R}$, if there is no risk of confusion.

\begin{Example}
	If we consider a kernel $k_{\B}\in C^{0}(\scrM\times\scrM,\scrE\boxtimes\scrE^{\ast})$, then the time kernel can directly be defined by
		\begin{align*}
			\langle\B_{t,\tau}\psi_{\tau},\varphi_{t}\rangle_{\scrN_{t}}:=\int_{\scrN_{t}}\int_{\scrN_{\tau}}\Sl\varphi_{t}(\vec{x})|k_{\B}(t,\vec{x};\tau,\vec{y})\psi_{\tau}(\vec{y})\Sr_{\scrE_{(t,\vec{x})}}\beta(\tau,\vec{y})d\mu_{\scrN_{\tau}}(\vec{y})\,d\mu_{\scrN_{t}}(\vec{x})
		\end{align*}
		for all $\varphi,\psi\in C_{\mathrm{c}}^{\infty}(\scrM,\scrE)$, where we identified $\scrE$ with $\scrE^{\ast}$ using $\Sl\cdot|\cdot\Sr_{\scrE}$, as usual. This generalizes the motivating example at the beginning of this section to operators in vector bundles.
\end{Example}

\begin{Prp}
	Let $\B\colon C^{\infty}_{\mathrm{c}}(\scrM,\scrE)\to C^{\infty}(\scrM,\scrE)$ be linear and continuous. If its Schwartz kernel satisfies the condition 
	\begin{align}\label{eq:TimeKernelCondition}
   		\bigg( \bigcup_{t,\tau\in\R} N^{\ast}(\scrN_t\times \scrN_{\tau})\bigg) \cap \mathrm{WF}(k_{\B}) = \emptyset \, ,
	\end{align}
	then $\B$ is a nonlocal potential with time kernel defined by 
	\begin{align*}
		k_{t,\tau}:=(\mathrm{pr}_{2}^{\ast}\beta_{\tau})(i_{t}\times i_{\tau})^{\ast}k_{\B}\qquad t,\tau\in\R\, ,
	\end{align*}
	where $\mathrm{pr}_{2}\colon \scrN_{t}\times\scrN_{\tau}\to\scrN_{\tau}$ denotes the projection onto the second factor.
\end{Prp}

\begin{proof}
	Let $k_{\B}\in\mathcal{D}^{\prime}(\scrM\times\scrM,\scrE\boxtimes\scrE^{\ast})$ be an arbitrary bi-distributional kernel. If Condition~\eqref{eq:TimeKernelCondition} is satisfied, we can define the two-parameter family of distributions
	\begin{align*}
		k_{t,\tau}:=(\mathrm{pr}_{2}^{\ast}\beta_{\tau})(i_{t}\times i_{\tau})^{\ast}k_{\B}\in\mathcal{D}^{\prime}(\scrN_t\times \scrN_{\tau},\scrE\vert_{\scrN_{t}}\boxtimes\scrE\vert_{\scrN_{\tau}}^{\ast})
\end{align*}
	for $t,\tau\in\R$ and using the kernel theorem of Schwartz (Theorem~\ref{Thm:Schwartz}), we associate for every $(k_{\B})_{t,\tau}$ a unique linear and continuous operator
	\begin{align*}
		\B_{t,\tau}\colon C^{\infty}_{\mathrm{c}}(\scrN_{\tau},\scrE\vert_{\scrN_{\tau}})\to\mathcal{D}^{\prime}(\scrN_{t},\scrE\vert_{\scrN_{t}})
	\end{align*}	
	If $\B$ in addition semi-regular operator, then so is $\B_{t,\tau}$ for all $t,\tau\in\R$ and we obtain a two-parameter family of linear operators $\B_{t,\tau}\colon C^{\infty}_{\mathrm{c}}(\scrN_{\tau},\scrE\vert_{\scrN_{\tau}})\to C^{\infty}(\scrN_{t},\scrE\vert_{\scrN_{t}})$. Following similar steps as in Proposition~\ref{Prop:Distr}, one shows $(\B_{t,\tau})_{t,\tau\in\R}$ is related to $\B$ by the formula
	\begin{align*}
		\langle\beta_{t}(\B\varphi)_{t},\psi_{t}\rangle_{\scrN_{t}}=\int_{\R}\langle \beta_{t}\B_{t,\tau}\varphi_{\tau},\psi_{t}\rangle_{\scrN_{t}}\,d\tau\, ,
	\end{align*} 
	where the integrand on the right-hand side is continuous in $t,\tau$. In other words, $(\B_{t,\tau})_{t,\tau\in\R}$ defines a time kernel in the sense of Definition~\ref{Def:NonLocPot}.
\end{proof}

\begin{Example}[Pseudodifferential Operators]
	Recall that a \emph{pseudo-differential operator} on $\scrM$ with values in $\scrE$ is a linear operator $\mathrm{Op}(b)\colon C^{\infty}_{\mathrm{c}}(\scrM,\scrE)\to C^{\infty}(\scrM,\scrE)$ whose Schwartz kernel is locally of the form
	\begin{align*}
		k_{b}(x,y)=\int_{\mathbb{R}^{k+1}}e^{i(x-y)k}b(x,y,\xi)\,d^{k+1}\xi
	\end{align*}
	for some \emph{(uniform Kohn-Nirenberg) symbol} $b\in\mathcal{S}^{m}(\mathcal{U}\times\mathcal{U},\C^{N\times N})$ with $m\in\R$ and $\mathcal{U}\subset\scrM$ open, where the integral makes sense as an oscillatory integral, i.e.~with an appropriate regularisation. We refer to \cite{grigis,Shubin} for details. Any such operator is \emph{pseudo-local}, which means that $\mathrm{sing}\,\mathrm{supp}(k_{b})$ is contained in the diagonal. Furthermore, it holds that
	\begin{align*}
		\mathrm{WF}(k_{b})=\mathrm{WF}^{\prime}(\mathrm{Op}(b))\cap N^{\ast}\Delta\backslash\{0\}\, ,
	\end{align*}
	where $\Delta:=\{(x,x)\mid x\in\scrM\}$ denotes the diagonal and $\mathrm{WF}^{\prime}(\mathrm{Op}(b))$ the \emph{operator wavefront set}, which is locally defined as the subset of covectors $(x,\xi)\in T^{\ast}\scrM\backslash\{0\}$ around which the symbol $b$ fails to be a smoothing symbol. Now, note that
	\begin{align*}
		\bigg(\bigcup_{t,\tau\in\R} N^{\ast}(\scrN_{t}\times\scrN_{\tau})\bigg) \cap N^{\ast}\Delta=\{((t,\vec{x};\xi_{t},0),(t,\vec{x};-\xi_{t},0))\mid (t,\vec{x})\in \scrM, \xi_{t}\in T^{\ast}_{t}\R\}\, .
	\end{align*}
	Hence, if the operator wavefront set $\mathrm{WF}^{\prime}(\mathrm{Op}(b))$ does not intersect this set, or in other words, if the operator $\mathrm{Op}(b)$ is a smoothing operator in time, then $\mathrm{Op}(b)$ defines a nonlocal potential with time kernel. In fact, it is not too hard to see that the operators $\mathrm{Op}(b)_{t,\tau}$ are again pseudo-differential operators.
\end{Example}

More precisely, we will restrict our attention to two important special cases, namely \textit{retarded} nonlocal potentials and nonlocal potentials with \textit{short time range}.

\begin{Def}\label{Def:Potentials} 
	Let $\B\colon C^{\infty}_{\mathrm{c}}(\scrM,\scrE)\to C^{\infty}(\scrM,\scrE)$ be a nonlocal potential with time kernel $(\B_{t,\tau})_{t,\tau\in\R}$.
	\begin{itemize}
		\item[(R)]$\B$ is called \emph{retarded}, if $\mathrm{supp}(\B\varphi)\subset J^{+}(\mathrm{supp}(\varphi))$ for all $\varphi\in C^{\infty}_{\mathrm{c}}(\scrM,\scrE)$ and hence also $\B_{t,\tau}=0$ for $\tau>t$.
		\item[(S)]$\B$ has \emph{short time range}, if there is a $\delta>0$ such that $\B_{t,\tau}=0$ for $\vert t-\tau\vert>\delta$. 
	\end{itemize}
\end{Def}

\begin{Remark} \label{Rem:MappingProp}
	The two types of nonlocal potentials introduced in Definition~\ref{Def:Potentials} have the following mapping properties, cf.~\cite[Lemma~2.13]{baergreen}:
	\begin{itemize}
		\item[(i)]If $\B$ is a retarded nonlocal potential, then it is easy to see that
		\begin{align*}
			\mathrm{supp}(k_{\B})\subset\{(x,y)\in\scrM\times\scrM\mid y\in J^{-}(x)\}\, .
		\end{align*}
		In particular, $\B$ can be extended uniquely to a linear and continuous operator $\B\colon C^{\infty}_{\mathrm{pc}}(\scrM,\scrE)\to C^{\infty}_{\mathrm{pc}}(\scrM,\scrE)$, where $C^{\infty}_{\mathrm{pc}}(\scrM,\scrE)$ denotes the space of section with \emph{past-compact support}, i.e.~
	\begin{align*}
		C^{\infty}_{\mathrm{pc}}(\scrM,\scrE):=\{\psi\in C^{\infty}(\scrM,\scrE)\mid\exists\text{ Cauchy surface }\scrN\colon\mathrm{supp}(\psi)\subset J^{+}(\scrN)\}\, .
	\end{align*}		
		\item[(ii)]If $\B$ has short time range $\delta>0$, then 
		\begin{align*}
			\mathrm{supp}(k_{\B})\subset\{(t,\vec{x};\tau,\vec{y})\in\scrM\times\scrM\mid \vert t-\tau\vert\leq \delta\}\, .
		\end{align*}	
		In particular, $\B$ can be extended uniquely to a linear and continuous operator $\B\colon C^{\infty}_{\mathrm{sc}}(\scrM,\scrE)\to C^{\infty}(\scrM,\scrE)$. Furthermore, $\B$ is well-defined as a linear operator of the form 
		\begin{align*}
			\B\colon C^{\infty}([0,T]\times\scrN,\scrE)\to C^{\infty}([-\delta,T+\delta]\times\scrN,\scrE)
		\end{align*}
		for all $T>0$.
	\end{itemize}
\end{Remark}

\begin{Example}\label{Exam:RetUltraStat}
	Consider an ultrastatic globally hyperbolic manifold 
	\begin{align*}
		(\scrM=\R\times\scrN,\,g=-dt\otimes dt+h)\,.\end{align*}
	Furthermore, let $k\in\mathcal{D}(\scrN\times\scrN,\scrE\vert_{\scrN}\times\scrE\vert_{\scrN}^{\ast})$ be an arbitrary kernel on $(\scrN,h)$ and consider a map $\chi\in C^{0}(\R)$ with the property that $\chi(t)=0$ for $t< 0$. Then,
	\begin{align*}
		k_{\B}(t,\vec{x};\tau,\vec{y}):=\chi(t-\tau)k(\vec{x},\vec{y})\in\mathcal{D}^{\prime}(\scrM\times\scrM,\scrE\boxtimes\scrE^{\ast})
	\end{align*}
	is an example of a retarded nonlocal potential. Kernels of this form appear often in \emph{linear response theory}. An example in Maxwell's theory will be discussed in Section~\ref{sec:Maxwell}.
\end{Example}

\begin{Example}[Retarded Green's Operator] 
	Consider $(1+k)$-dimensional Minkowski spacetime $\scrM:=\R\times\R^{k}$ and a linear evolution equation of the form 
	\begin{align*}
		\scrS:=\partial_{t}-P(t)\colon C^{\infty}(\scrM)\to C^{\infty}(\scrM)\, ,
	\end{align*}
	where $P(t)$ is one-parameter family of linear and elliptic first-order differential operators on $\scrN_{t}\cong \R^{k}$. We denote by 
	\begin{align*}
		\mathcal{U}(t,\tau)\colon C^{\infty}_{\mathrm{c}}(\scrN_{\tau})\to C^{\infty}_{\mathrm{c}}(\scrN_{t})\, ,\qquad \mathfrak{f}\mapsto \psi\vert_{\scrN_{t}}
	\end{align*}
	the corresponding \emph{Cauchy evolution operator}, where $\psi\in C^{\infty}_{\mathrm{sc}}(\scrM)$ is the unique solution to $\scrS\psi=0$ with initial datum $\psi\vert_{\scrN_{\tau}}=\mathfrak{f}$. Then, its retarded Green operator $G_{\mathrm{ret}}\colon C^{\infty}_{\mathrm{c}}(\scrM)\to C^{\infty}_{\mathrm{sc}}(\scrM)$ is a retarded nonlocal potential in the sense of Definition~\ref{Def:Potentials} with time kernel
	\begin{align*}
		(G_{\mathrm{ret}})_{t,\tau}:=\theta(t-\tau)\mathcal{U}(t,\tau)\, .
	\end{align*}
	Indeed, noting that the Cauchy evolution operator satisfies the operator equations	
	\begin{align*}
		\begin{cases}
		\partial_{t}\mathcal{U}(t,\tau)&=P(t)\mathcal{U}(t,\tau)\\
		\partial_{\tau}\mathcal{U}(t,\tau)&=-\mathcal{U}(t,\tau)P(\tau)
		\end{cases}
	\end{align*}
	one immediately verifies the defining properties of $G_{\mathrm{ret}}$ (see~Proposition~\ref{prop:Green}). For instance, for all $\psi\in C^{\infty}_{\mathrm{c}}(\scrM)$, it holds that 
	\begin{align*}
		\scrS(G_{\mathrm{ret}}\psi)(t,\vec{x})&=(\partial_{t}-P(t))\int_{-\infty}^{t}(\mathcal{U}(t,\tau)\psi_{\tau})(\vec{x})\,d\tau=\\&=(\underbrace{\mathcal{U}(t,t)}_{=\mathrm{id}}\psi_{t})(\vec{x})+\int_{-\infty}^{t}\underbrace{(\partial_{t}-P(t))(\mathcal{U}(t,\tau)\psi_{\tau})(\vec{x})}_{=0}\,d\tau=\psi(t,\vec{x})\, .
	\end{align*}
	With a similar computation, one shows that $G_{\mathrm{ret}}(\scrS\psi)=\psi$. Last but not least, by definition, it is clear that $\mathrm{supp}(G_{\mathrm{ret}}\psi)\subset J^{+}(\mathrm{supp}(\psi))$, which proves the claim by uniqueness of the retarded and advances Green operators.
	
	More generally, if $\scrS\colon C^{\infty}(\scrM,\scrE)\to C^{\infty}(\scrM,\scrE)$ is a symmetric hyperbolic system over an arbitrary globally hyperbolic manifold $(\scrM,g)$, then its retarded Green operator 
	\begin{align*}
		G_{\mathrm{ret}}\colon C^{\infty}_{\mathrm{c}}(\scrM,\scrE)\to C^{\infty}_{\mathrm{sc}}(\scrM,\scrE)
	\end{align*}
	as defined in Proposition~\ref{prop:Green} is a retarded nonlocal potential in the sense of Definition~\ref{Def:Potentials} with time kernel
	\begin{align*}
		(G_{\mathrm{ret}})_{t,\tau}:=\theta(t-\tau)\mathcal{U}(t,\tau)\, ,
	\end{align*}
	where $\mathcal{U}(t,\tau)\colon\H_{\tau}\to\H_{t}$ is the unitary Cauchy evolution operator defined as above.
	
	We remark that examples of hyperbolic equations with nonlocal operators constructed out of the retarded Green's operator are relevant for semiclassical equations.
\end{Example}

	Last but not least, let us comment on adjoints of nonlocal potentials. To start with, for a semiregular operator $\B\colon C^{\infty}_{\mathrm{c}}(\scrM,\scrE)\to C^{\infty}(\scrM,\scrE)$ with Schwartz kernel $k_{\B}$ one can always define its formal adjoint $\B^{\dagger}\colon C^{\infty}_{\mathrm{c}}(\scrM,\scrE)\to C^{\infty}(\scrM,\scrE)$ with respect to the non-degenerate sesquilinear form $\bra\cdot|\cdot\ket_{\scrE}$ defined in~\eqref{eq:SectionMetric}. The operator $\B^{\dagger}$ is continuous as well and its Schwartz kernel $k_{\B^{\dagger}}$ is related to $k_{\B}$ via
	\begin{align*}
		\langle k_{\B^{\dagger}},\psi^{\ast}\boxtimes\varphi\rangle_{\scrM\times\scrM}=\overline{\langle k_{\B},\varphi^{\ast}\boxtimes\psi\rangle_{\scrM\times\scrM}}\, ,
	\end{align*}
	for all $\psi,\varphi\in C^{\infty}_{\mathrm{c}}(\scrM,\scrE)$ with $\psi^{\ast}\in C^{\infty}_{\mathrm{c}}(\scrM,\scrE^{\ast})$ defined by $\psi^{\ast}:=\Sl\psi|\cdot\Sr_{\scrE}$.  If $\B$ is a smoothing operator and we write its smooth kernel as $k_{\B}(x,y)\colon\scrE_{y}\to\scrE_{x}$, then
	\begin{align*}
		k_{\B^{\dagger}}(x,y)=k_{\B}(y,x)^{\dagger}\, ,
	\end{align*}
	where $\dagger$ on the right-hand side denotes the adjoint with respect to the fibre metric. 
	
	Now, if $\B$ is a nonlocal potential with time kernel $\B_{t,\tau}$, also $\B^{\dagger}$ is nonlocal potential with time kernel given by $(\B^{\dagger})_{t,\tau}:=(\B_{\tau,t})^{\dagger}$. In particular, if $\B$ is retarded in the sense of Definition~\ref{Def:Potentials}(i), its formal adjoint $\B^{\dagger}$ will be \emph{advanced}, i.e.~$\mathrm{supp}(\B^{\dagger}\psi)\subset J^{-}(\mathrm{supp}(\psi))$ for all $\psi\in C^{\infty}_{\mathrm{c}}(\scrM,\scrE)$ and hence also $(\B^{\dagger})_{t,\tau}=0$ for $t<\tau$. If $\B$ has short time range $\delta$ in the sense of Definition~\ref{Def:Potentials}(ii), the same holds true for $\B^{\dagger}$.
\section{Symmetric Hyperbolic Systems with Nonlocal Potentials}
\label{Sec:SHSNonLoc}
In this section, we discuss the Cauchy problem for symmetric hyperbolic systems with nonlocal potentials in two different situations, namely for \textit{retarded} nonlocal potentials and nonlocal potential with \textit{short time range}. To start with, let us fix the basic set-up for the following discussion.

\begin{Setup}\label{Setup} Throughout this section, we consider the following data:
\begin{itemize}
	\item[$\bullet$]A globally hyperbolic Lorentzian manifold $(\scrM=\R\times\scrN,g=-\beta^{2}dt\otimes dt+h_{t})$ of dimension $1+k$ and a time strip $\overline{\scrM}_{T}=t^{-1}([0,T])$ for $T>0$;
	\item[$\bullet$]A Hermitian vector bundle $(\scrE,\Sl\cdot\vert\cdot\Sr_{\scrE})$ over $\scrM$ and a metric connection $\nabla^{\scrE}$;
	\item[$\bullet$]A symmetric hyperbolic system $\scrS=\sigma_{\scrS}(dx^{\mu})\nabla_{\mu}^{\scrE}-\scrS_{0}$ on $\scrE$;
	\item[$\bullet$]A nonlocal potential $\B\colon C^{\infty}_{\rm{c}}(\scrM,\scrE)\to C^{\infty}(\scrM,\scrE)$ with time kernel $(\B_{t,\tau})_{t,\tau\in\R}$.
\end{itemize}
\end{Setup}

The goal of this section is to study the Cauchy problem for the \textit{nonlocal symmetric hyperbolic system} $\scrS-\B$, i.e.
\begin{align}\label{eq:SymNonLoc}
	\begin{cases}
		(\scrS-\B)\psi &=\phi\\
		\psi\vert_{\scrN_{0}}&=\mathfrak{f}
	\end{cases}
\end{align}
on the time strip $\overline{\scrM}_{T}=[0,T]\times\scrN$ for a source $\phi\in C^{\infty}_{\mathrm{c}}(\scrM,\scrE)$ with $\mathrm{supp}(\phi)\subset\overline{\scrM}_{T}$ and initial datum $\mathfrak{f}\in C^{\infty}_{\mathrm{c}}(\scrN_{0},\scrE\vert_{\scrN_{0}})$.

We restrict our attention to the following nonlocal potential (see also Definition~\ref{Def:Potentials}):
\begin{itemize}
	\item[(R)] The case of a \emph{retarded nonlocal potential}, i.e.~$\B_{t,\tau} = 0$ for $\tau > t$. In addition, we assume that the time kernel has \emph{past-compact support} in the sense that there exists $t_0 \geq 0$ such that $\B_{t,\tau} = 0$ for all $\tau \leq -t_0$. Hence,
	\begin{align*}
		\mathrm{supp}(k_{\B})\subset\{(x,y)\in\scrM\times\scrM\mid y\in J^{-}(x)\}\cap (J^{+}(\scrN_{t_{0}})\times J^{+}(\scrN_{t_{0}}))\, .
	\end{align*}		
	In other words, the nonlocality is \emph{switched on} only after a finite time in the past.
	\item[(S)] The case of a nonlocal potential with \emph{short time range}, i.e.~there exists a constant $\delta > 0$ such that $\B_{t,\tau} = 0$ whenever $\vert t - \tau \vert > \delta$.
\end{itemize}

We shall prove that for these nonlocal potentials the Cauchy problem~\eqref{eq:SymNonLoc} admits strong solutions (as in \cite{friedrichs}), clearly under suitable assumptions. In the context of symmetric hyperbolic systems with nonlocal potentials, strong solutions are defined as follows.
			
\begin{Def}\label{Def:StrongSol} \emph{(Strong Solutions)}
	\begin{itemize}
		\item[(i)]In case (R), assume in addition that $\B$ is uniformly bounded in time on $[-t_{0},T]$, i.e.~there is a constant $C_{T}>0$ such that $\Vert\B_{t,\tau}\varphi_{\tau}\Vert\leq C_{T}\Vert\varphi_{\tau}\Vert_{\tau}$ for all $\varphi\in C^{\infty}_{\mathrm{sc}}(\scrM,\scrE)$ and $t,\tau\in [-t_{0},T]$. Then, a \emph{strong solution on $\overline{\scrM}_{T}$} is an element $\psi\in\H_{[-t_{0},T]\times\scrN}$ for which there exists a sequence $\psi_{k}\in C^{\infty}(\scrM,\scrE)\cap \H_{[-t_{0},T]\times\scrN}$ such that
		\begin{align*}
			\psi_{k}\vert_{\scrN_{0}}\xrightarrow{k\to\infty}\mathfrak{f}&\quad\text{in}\quad\H_{0}\,,\qquad\qquad\psi_{k}\xrightarrow{k\to\infty}\psi\quad\text{in}\quad\H_{[-t_{0},T]\times\scrN}\,,\\ &\quad(\scrS-\B)\psi_{k}\xrightarrow{k\to\infty} \phi\quad\text{in}\quad \H_{\overline{\scrM}_{T}}\, .
		\end{align*}
		\item[(ii)]In case (S), assume in addition that $\B$ is uniformly bounded in time on $[-\delta,T+\delta]$, i.e.~there is a constant $C_{T}>0$ such that $\Vert\B_{t,\tau}\varphi_{\tau}\Vert\leq C_{T}\Vert\varphi_{\tau}\Vert_{\tau}$ for all $\varphi\in C^{\infty}_{\mathrm{sc}}(\scrM,\scrE)$ and $t,\tau\in [-\delta,T+\delta]$. Then, a \emph{strong solution on $\overline{\scrM}_{T}$} is an element $\psi\in\H_{[-\delta,T+\delta]\times\scrN}$ for which there exists a sequence $\psi_{k}\in C^{\infty}(\scrM,\scrE)\cap \H_{[-\delta,T+\delta]\times\scrN}$ such that
		\begin{align*}
			\psi_{k}\vert_{\scrN_{0}}\xrightarrow{k\to\infty}\mathfrak{f}&\quad\text{in}\quad\H_{0}\,,\qquad\qquad\psi_{k}\xrightarrow{k\to\infty}\psi\quad\text{in}\quad\H_{[-\delta,T+\delta]\times\scrN}\,,\\ &\quad(\scrS-\B)\psi_{k}\xrightarrow{k\to\infty} \phi\quad\text{in}\quad \H_{\overline{\scrM}_{T}}\, .
		\end{align*}
	\end{itemize}
\end{Def}

Note that, it is important in the above definition that the approximating sequence $\psi_{k}$ is defined on all of $\scrM$ and not just the time strip $\overline{\scrM}_{T}$, since, $\B\psi_{k}$ on $[0,T]\times\scrN$ requires us to specify $\psi_{k}$ on $[-t_{0},T]$ in the retarded case (R) and on $[-\delta,T+\delta]\times\scrN$ in the short time range case (S). Furthermore, we note that the boundedness assumption in time ensures that $\B$ is well-defined on smooth sections that are $L^{2}$ on each time slice. 

While full regularity of the strong solution cannot be expected for a generic short time range nonlocal potential with general initial data, for retarded potentials, we shall show that, under some additional assumption on the regularity of the nonlocal potential, strong solutions are  ``classical''. In the context of symmetric hyperbolic systems with nonlocal potentials, strong solutions are defined as follows. For completeness, we also give a suitable definition for classical solutions in the short time range case.

\begin{Def}\label{Def:ClassSol} \emph{(Classical Solutions)}
	\begin{itemize}
		\item[(i)] In case (R), a \emph{classical solution on $\overline{\scrM}_{T}$} is a section $\psi \in C^1((-t_0, T] \times \scrN, \scrE)$ such that
		\begin{align*}
			(\scrS - \B)\psi = \phi \quad \text{pointwise for all } t \in [0,T]\,, \quad \text{and} \quad \psi\vert_{\scrN_0} = \mathfrak{f}\,.
		\end{align*}
		\item[(ii)] In case (S), a \emph{classical solution on $\overline{\scrM}_{T}$} is a section $\psi \in C^1((-\delta, T+\delta) \times \scrN, \scrE)$ such that
		\begin{align*}
			(\scrS - \B)\psi = \phi \quad \text{pointwise for all } t \in [0,T]\,, \quad \text{and} \quad \psi\vert_{\scrN_0} = \mathfrak{f}\,.
		\end{align*}
	\end{itemize}
\end{Def}

Clearly, if a strong solution converge to a $C^1$-solution $\psi$ then $\psi$ is clearly also a classical solution. Indeed, any strong solution is a weak solution and  if it is additional $C^1$ then it is also classical.

\subsection{Energy Estimates}
To establish the existence and uniqueness of solutions, it is essential to derive appropriate energy estimates that control the behaviour of the system over time. As a first step in this direction, we derive a fundamental identity governing the evolution of the energy associated with a symmetric hyperbolic system. This identity serves as the foundation for obtaining a priori bounds, which are crucial in the subsequent analysis.

\begin{Prp}[Energy Evolution]\label{Prop:energy}
	Let $(\scrM,g)$ be a globally hyperbolic manifold with Cauchy temporal function $t\colon\scrM\to\R$ and consider a symmetric hyperbolic system $\scrS$ on a Hermitian bundle $(\scrE,\Sl\cdot|\cdot\Sr_{\scrE})$. Then, for any $\psi\in C^{\infty}_{\mathrm{sc}}(\scrM,\scrE)$, it  holds that
	\begin{align*}
	\frac{d}{dt}\Vert\psi_{t}\Vert^{2}_{t}=2\Re\int_{\scrN_{t}}\Sl \scrS\psi|\psi\Sr_{\scrE}\beta \,d\mu_{\scrN_{t}}-\int_{\scrN_{t}}\Sl (\scrS+\scrS^{\dagger})\psi|\psi\Sr_{\scrE}\beta\,d\mu_{\scrN_{t}}
\end{align*}
for all $t\in\R$, where we used the notation $\psi_{t}(\cdot):=\psi\vert_{\scrN_{t}}(\cdot)=\psi(t,\cdot)\in C^{\infty}_{\mathrm{c}}(\scrN_{t},\scrE\vert_{\scrN_{t}})$.
\end{Prp}

\begin{proof}
	Let $\psi\in C^{\infty}_{\mathrm{sc}}(\scrM,\scrE)$ be arbitrary. Similar as in the proof of Proposition~\ref{Prop:AdjointSHS}, we consider the real-valued $k$-form $\omega\in\Omega^{k}(\scrM)$ defined by
	\begin{align*}
		\omega=\Sl\sigma_{\scrS}(dx^{\mu})\psi|\psi\Sr_{\scrE} \partial_{\mu}\lrcorner d\mu_{\scrM}\,.
	\end{align*}
	As explained in more details in the proof of 	Proposition~\ref{Prop:AdjointSHS} (see~Equation~\eqref{eq:AdjointProof}), its exterior derivative can be written as
	\begin{align}\label{eq:Leibniz}
		d\omega=\Sl\psi|\scrS\psi\Sr_{\scrE}-\Sl\scrS^{\dagger}\psi|\psi\Sr_{\scrE}=2\Re \Sl\scrS\psi|\psi\Sr_{\scrE}-\Sl (\scrS+\scrS^{\dagger})\psi|\psi\Sr_{\scrE}\, .
	\end{align}
	Now, consider a closed time strip $\overline{\scrM}_{T}:=t^{-1}([0,t])$ for $t\geq 0$. Stokes' theorem implies
	\begin{align}\label{eq:Stokes}
		\int_{\overline{\scrM}_{T}}d\omega=\int_{\scrN_{t}}i_{t}^{\ast}\omega-\int_{\scrN_{t_{0}}}i_{0}^{\ast}\omega\, ,
	\end{align}
	where $i_{\tau}\colon\scrN_{\tau}\hookrightarrow\scrM$ denote the natural embeddings. Using Fubini's theorem and Equation~\eqref{eq:Leibniz}, we rewrite the left-hand side of~\eqref{eq:Stokes} as
	\begin{align}\label{eq:LHSProof}
		\int_{\overline{\scrM}_{T}}d\omega=\int_{0}^{t}\int_{\scrN_{\tau}}\bigg\{2\Re (\Sl\scrS\psi|\psi\Sr_{\scrE})-\Sl (\scrS+\scrS^{\dagger})\psi|\psi\Sr_{\scrE} \bigg\}\beta\,d\mu_{\scrN_{\tau}} d\tau\, .
	\end{align}
	For the right-hand side of~Equation \eqref{eq:Stokes}, we first claim that for any vector field $X\in C^{\infty}(\scrM,T\scrM)$, it holds that
	\begin{align*}
		i_{\tau}^{\ast}(X\lrcorner d\mu_{\scrM})=-g(X,\nu)d\mu_{\scrN_{\tau}}\, ,
	\end{align*} 
	where $\nu$ denotes the future-directed timelike normal vector of $\scrN_{\tau}$ in $\scrM$, as usual. Indeed, decomposing $X$ into a normal component $X^{\perp}:=-g(X,\nu)\nu$ and tangential component $X^{\top}:=X-X^{\perp}$, we have that
	\begin{align*}
		i_{\tau}^{\ast}(X\lrcorner d\mu_{\scrM})=i_{\tau}^{\ast}(X^{\perp}\lrcorner d\mu_{\scrM})=-g(X,\nu)i_{\tau}^{\ast}(\nu\lrcorner d\mu_{\scrM})=-g(X,\nu)d\mu_{\scrN_{\tau}}\, ,
	\end{align*}
	as claimed above. In particular, when choosing coordinates adapted to the splitting obtained from our Cauchy temporal function $t\colon\scrM\to\R$, i.e.~$(x^{\mu})_{\mu=0,\dots,k}$ with $x^{0}=t$, where $(x^{i})_{i=1,\dots,k}$ are local coordinates on $\scrN$, we find $\nu=\beta^{-1}\partial_{t}$ and hence, we conclude that
	\begin{align*}
		-g(X,\nu)=\Sl\sigma_{\scrS}(\beta dt)\psi|\psi\Sr_{\scrE}\qquad\text{ for }\qquad X:=\Sl\sigma_{\scrS}(dx^{\mu})\psi|\psi\Sr_{\scrE} \partial_{\mu}\end{align*} 
	To sum up, we have shown that 
	\begin{align}\label{eq:RHSProof}
		\int_{\scrN_{\tau}}i_{\tau}^{\ast}\omega=\int_{\scrN_{\tau}}\Sl\sigma_{\scrS}(\beta dt)\psi|\psi\Sr_{\scrE}d\mu_{\scrN_{\tau}}=\Vert\psi_{\tau}\Vert_{\tau}^{2}
	\end{align}
	for all $\tau\in\R$. Using Equation~\eqref{eq:LHSProof} as well as Equation~\eqref{eq:RHSProof} for the right-hand side, we obtain from~\eqref{eq:Stokes} the equation
	\begin{align*}
		\int_{0}^{t}\int_{\scrN_{\tau}}\bigg\{2\Re (\Sl\scrS\psi|\psi\Sr_{\scrE})-\Sl (\scrS+\scrS^{\dagger})\psi|\psi\Sr_{\scrE} \bigg\}\beta\,d\mu_{\scrN_{\tau}} d\tau=\Vert\psi_{t}\Vert_{t}^{2}-\Vert\psi_{0}\Vert_{0}^{2}\, .
	\end{align*}
	Taking the derivative with respect to the time variable $t$ yields the claimed result. For $t\leq 0$, we reverse the arguments by choosing a time strip $t^{-1}([t,0])$, which results into the equation
	\begin{align*}
		\int_{t}^{0}\int_{\scrN_{\tau}}\bigg\{2\Re (\Sl\scrS\psi|\psi\Sr_{\scrE})-\Sl (\scrS+\scrS^{\dagger})\psi|\psi\Sr_{\scrE} \bigg\}\beta\,d\mu_{\scrN_{\tau}} d\tau=\Vert\psi_{0}\Vert_{0}^{2}-\Vert\psi_{t}\Vert_{t}^{2}\, .
	\end{align*}
	Taking the time-derivative of this expression yields the same result.
\end{proof}

Now, in the following computations, it will be useful to rewrite the right-hand side of the energy equation in Proposition~\ref{Prop:energy} in terms of the time kernel $\B_{t,\tau}$ and the Hilbert space inner product $(\cdot|\cdot)_{t}$ (see~Equation~\eqref{eq:Ht}). As a first step, we introduce the concept of a \textit{weighted time kernel}.

\begin{Def}[Weighted Time Kernel]\label{eq:modTK}
	Let $t\colon\scrM\to\R$ be a Cauchy temporal function and $\B\colon C^{\infty}_{\rm{c}}(\scrM,\scrE)\to C^{\infty}(\scrM,\scrE)$ be a nonlocal potential with time kernel $(\B_{t,\tau})_{t,\tau\in\R}$. Then, we define the \emph{weighted potential} 
	\begin{align*}
		\V:=\beta\sigma_{\scrS}(\eta)^{-1}\B\colon C^{\infty}_{\rm{c}}(\scrM,\scrE)\to C^{\infty}(\scrM,\scrE)\, ,
	\end{align*}
	where $\eta:=\beta dt$ denotes the future-directed timelike normal covector, as usual. The time kernel of $\V$ is given by $
		\V_{t,\tau}:=\beta_{t}\sigma_{\scrS}(\eta_t)^{-1}\B_{t,\tau}$ where $\beta_{t}(\cdot):=\beta(t,\cdot)$.
\end{Def}

Note that this definition is well-defined, since $\Sl\sigma_{\scrS}(\eta)\cdot|\cdot\Sr_{\scrE}$ is pointwise positive definite, which implies that $\sigma_{\scrS}(\eta)\in C^{\infty}(\scrM,\mathrm{End}(\scrE))$ is pointwise invertible.

Using this notation, we derive the following energy estimates from Proposition~\ref{Prop:energy}, which will be relevant for our analysis.

\begin{Corollary}[Energy Estimates]\label{Cor:EnEstLoc}
	Consider a symmetric hyperbolic system $\scrS$ with the property that the zero-order operator $\mathcal{Z}_{\scrS}:=\beta\sigma_{\scrS}(\eta)^{-1}(\scrS+\scrS^{\dagger})$ is uniformly bounded in time, i.e.~there exists a constant $D>0$ such that $\Vert (\mathcal{Z}_{\scrS}\varphi)_{t}\Vert_{t}\leq D\Vert\varphi_{t}\Vert_{t}$ for all $\varphi\in C^{\infty}_{\mathrm{sc}}(\scrM,\scrE)$. Then, the unique solution $\psi\in C^{\infty}_{\mathrm{sc}}(\scrM,\scrE)$ to the Cauchy problem
	\begin{align*}
	\begin{cases}
		\scrS\psi&=\phi\\
		\psi\vert_{\scrN_{0}}&=\mathfrak{f}
	\end{cases}
\end{align*}	 
for compactly-supported Cauchy data $(\phi,\mathfrak{f})\in C^{\infty}_{\mathrm{c}}(\scrM,\scrE)\times C^{\infty}_{\mathrm{c}}(\scrN_{0},\scrE\vert_{\scrN_{0}})$ satisfies
\begin{align*}
	\Vert\psi_{t}\Vert_{t}\leq Me^{\frac{1}{2}D\vert t\vert}
\end{align*}
for all $t\in\R$ where $M>0$ is a constant only depending on $(\phi,\mathfrak{f})$.
\end{Corollary}

\begin{proof}
	Let us write $\phi^{\prime}:=\beta\sigma_{\scrS}(\eta)^{-1}\phi$. Then, taking the absolute value in  Proposition~\ref{Prop:energy} and using the Cauchy Schwartz inequality on the right-hand side yields
	\begin{align*}
		\bigg\vert\frac{d}{dt}\Vert\psi_{t}\Vert^{2}_{t}\bigg\vert\leq \Vert\psi_{t}\Vert_{t}\bigg(2\Vert \phi^{\prime}_{t}\Vert_{t}+\Vert(\mathcal{Z}_{\scrS}\psi)_{t}\Vert_{t}\bigg)\, ,
	\end{align*}
	We use the Leibniz rule on the left-hand side and the assumption on $\mathcal{Z}_{\scrS}$ to obtain the estimate
	\begin{align}\label{eq:EstEnGen}
		\bigg\vert\frac{d}{dt}\Vert\psi_{t}\Vert_{t}\bigg\vert\leq \Vert \phi^{\prime}_{t}\Vert_{t}+\frac{1}{2}\Vert\psi_{t}\Vert_{t}\, .
	\end{align}
	Now, we distinguish between two cases. First, note that the estimate~\eqref{eq:EstEnGen} implies
	\begin{align*}
		\frac{d}{dt}\bigg(e^{-\frac{1}{2}Dt}\Vert\psi_{t}\Vert_{t}\bigg)\leq e^{-\frac{1}{2}Dt}\Vert\phi_{t}^{\prime}\Vert_{t}\, .
	\end{align*}
	Then, for $t\geq 0$, integrating from $0$ to $t$ yields the estimate
	\begin{align*}
		\Vert\psi_{t}\Vert_{t}\leq e^{\frac{1}{2}Dt}\bigg\{\int_{0}^{t}\bigg(e^{-\frac{1}{2}D\tau}\Vert\phi^{\prime}_{\tau}\Vert_{\tau}\bigg)\,d\tau+\Vert\mathfrak{f}\Vert_{0}\bigg\}\leq e^{\frac{1}{2}Dt}\bigg(\int_{0}^{t}\Vert\phi^{\prime}_{\tau}\Vert_{\tau}\,d\tau+\Vert\mathfrak{f}\Vert_{0}\bigg)\, .
	\end{align*}
	For $t\leq 0$, we need to use a different sign in order to get the right exponential factor. In this case, we use the fact that the estimate~\eqref{eq:EstEnGen} also implies that
	\begin{align*}
		-\frac{d}{dt}\bigg(e^{\frac{1}{2}Dt}\Vert\psi_{t}\Vert_{t}\bigg)\leq e^{\frac{1}{2}Dt}\Vert\phi_{t}^{\prime}\Vert_{t}\, .
	\end{align*}
	Taking $t\leq 0$, integrating from $t$ to $0$ and using similar steps as above, results into the estimate
	\begin{align*}
		\Vert\psi_{t}\Vert_{t}\leq e^{-\frac{1}{2}Dt}\bigg(\int_{t}^{0}\Vert\phi^{\prime}_{\tau}\Vert_{\tau}\,d\tau+\Vert\mathfrak{f}\Vert_{0}\bigg)\, .
	\end{align*}
	Now, since $\phi$ (and hence also $\phi^{\prime}$) is compactly-supported, we can combine both estimates to find a constant $M>0$ such that $\Vert\psi_{t}\Vert_{t}\leq Me^{\frac{1}{2}D\vert t\vert}$, which gives the claimed result.
\end{proof}

\begin{Corollary}[Energy Estimates]\label{Cor:EnEst}
	Consider Set-up~\eqref{Setup} and let $\psi\in C^{\infty}_{\mathrm{sc}}(\scrM,\scrE)$ and $\phi\in C^{\infty}_{\mathrm{c}}(\scrM,\scrE)$ be such that $\scrS\psi=\B\phi$. Then,
	\begin{align*}
		\bigg\vert\frac{d}{dt}\Vert\psi_{t}\Vert_{t}\bigg\vert\leq \int_{\R}\Vert \V_{t,\tau}\phi_{\tau}\Vert_{t}\,d\tau+\frac{1}{2}\Vert(\mathcal{Z}_{\scrS}\psi)_{t}\Vert_{t}\, ,
	\end{align*}
	for all $t\in\R$, where $\mathcal{Z}_{\scrS}:=\beta\sigma_{\scrS}(\eta)^{-1}(\scrS+\scrS^{\dagger})$ is an operator of order zero.
\end{Corollary}

\begin{proof}
	First of all, using the defining relation of the time kernel in Definition~\ref{Def:NonLocPot} as well as the of the modified time kernel in Definition~\ref{eq:modTK} above, we can write		
	\begin{align*}
		\int_{\scrN_{t}}\Sl&\B\phi|\psi\Sr_{\scrE}\beta\,d\mu_{\scrN_{t}}=\int_{\R}\int_{\scrN_{t}}\Sl\beta_{t}\B_{t,\tau}\phi_{\tau}|\psi_{\tau}\Sr_{\scrE}\,d\mu_{\scrN_{t}}\,d\tau=\\&=\int_{\R}\int_{\scrN_{t}}\Sl\sigma_{\scrS}(\eta_{t})\V_{t,\tau}\phi_{\tau}|\psi_{t}\Sr_{\scrE}\,d\mu_{\scrN_{t}}\,d\tau=\int_{\R}\, (\V_{t,\tau}\phi_{\tau}|\psi_{t})_{t}\, d\tau\, .
	\end{align*}
	Using this equation as well as the notation $\mathcal{Z}_{\scrS}:=\beta\sigma_{\scrS}(\eta)^{-1}(\scrS+\scrS^{\dagger})$, we take the absolute value in Proposition~\ref{Prop:energy} applied to $\psi$ to obtain the estimate
	\begin{align*}
		\bigg\vert\frac{d}{dt}\Vert\psi_{t}\Vert^{2}_{t}\bigg\vert\leq \Vert\psi_{t}\Vert_{t}\bigg(2\int_{\R}\Vert \V_{t,\tau}\phi_{\tau}\Vert_{t}\,d\tau+\Vert(\mathcal{Z}_{\scrS}\psi)_{t}\Vert_{t}\bigg)\, ,
	\end{align*}
	where we used the triangle and Cauchy-Schwartz inequalities. Rewriting the left-hand side using the Leibniz rule yields the claimed result.
\end{proof}

\subsection{Extended Symmetric Hyperbolic System Including Derivatives}
\label{Subsec:ExtSHS}
In this section, we show that any symmetric hyperbolic system with a nonlocal potential can be reformulated as an equivalent first-order system with nonlocal potential for an extended variable that includes both the section and its covariant derivative. This is done by introducing an enlarged bundle, on which the corresponding enlarged system and potential are defined. The construction will be used later to establish regularity of solutions, adapting ideas from the theory of symmetric hyperbolic systems on $\mathbb{R}^{n}$ as explained in \cite[Sec.~5.3]{john} and \cite[Sec.~13.3]{intro} to the setting of vector bundles and nonlocal potentials considered in this article.

Throughout this section, we consider Set-up~\ref{Setup}. As in Section~\ref{SobolevSpaces}, we introduce the vector bundles
\begin{align*}
	\scrE_{i}:=\scrE\otimes T^{\ast}\scrM^{\otimes i}
\end{align*}
and equip them with the bundle metrics $\Sl\cdot|\cdot\Sr_{\scrE_{i}}$ and connections $\nabla^{\scrE_{i}}$ induced from the bundle metrics and connection on $\scrE$ as well as the auxiliary \textit{Riemannian} metric $g_{0}=dt\otimes dt+h_{t}$ on $\scrM$ and corresponding Levi-Civita connection $\nabla^{0}$ on $T^{\ast}\scrM$, i.e.
\begin{align*}
	\Sl\cdot|\cdot\Sr_{\scrE_{i}}:=\Sl\cdot|\cdot\Sr_{\scrE}\otimes g_{0}^{\otimes i},\qquad \nabla^{\scrE_{i}}:=\nabla^{\scrE}\otimes(\nabla^{0})^{\otimes i}\, .
\end{align*} 
Now, for any section $\psi\in C^{\infty}(\scrM,\scrE)$, we define the \textit{extended} variable $\Psi$ by
\begin{align}\label{eq:Extended}
	\Psi:=(\psi,\nabla^{\scrE}\psi)\in C^{\infty}(\scrM,\mathsf{E}_{1}),\qquad \mathsf{E}_{1}:=\scrE\oplus\scrE_{1}\, ,
\end{align}
where the bundle $\mathsf{E}_{1}$ comes equipped with a natural bundle metric $\Sl\cdot|\cdot\Sr_{\mathsf{E}_{1}}$ and metric connection $\nabla^{\mathsf{E}_{1}}$ induced from the ones on $\scrE$ and $\scrE_{1}$, i.e.
\begin{align*}
	\Sl\Psi|\Phi\Sr_{\mathsf{E}_{1}}:=\Sl\Psi_{0}|\Phi_{0}\Sr_{\scrE}+\Sl\Psi_{1}|\Phi_{1}\Sr_{\scrE_{1}}\qquad\text{and}\qquad \nabla^{\mathsf{E}_{1}}\Psi:=(\nabla^{\scrE}\Psi_{0},\nabla^{\scrE_{1}}\Psi_{1})\, .
\end{align*}

\begin{Remark}\label{Rem:SobSpL2}
	Let $\psi\in C^{1}(\scrM,\scrE)$ and write $\Psi=(\psi,\nabla^{\scrE}\psi)\in C^{0}(\scrM,\mathsf{E}_{1})$. Then, by definition, it holds that
	\begin{align*}
		\psi\in W^{2,1}(\mathcal{U},\scrE_{\beta})\qquad\Leftrightarrow\qquad \Psi\in L^{2}(\mathcal{U},\mathsf{E}_{1,\beta})\, .
	\end{align*}
	for every $\mathcal{U}\subset\scrM$ open, where $\scrM$ is equipped with the Riemannian metric $g_{0}$ and where $\mathsf{E}_{1,\beta}$ denotes the bundle $\mathsf{E}_{1}$ equipped with $\Sl\cdot|\cdot\Sr_{\beta,\mathsf{E}_{1}}:=\Sl\sigma_{\mathsf{S}_{1}}(\beta dt)\cdot|\cdot\Sr_{\mathsf{E}_{1}}$.
\end{Remark}

Now, the goal of the following is to show that there exists system $\mathsf{S}_{1}$ and nonlocal potential $\mathsf{B}_{1}$ on the Hermitian bundle $(\mathsf{E}_{1},\Sl\cdot|\cdot\Sr_{\mathsf{E}_{1}})$, such that for every smooth solution $\psi\in C^{\infty}(\scrM,\scrE)$ to the equation $(\scrS-\B)\psi=\phi$ for some given source $\phi\in C^{\infty}(\scrM,\scrE)$, the corresponding extended variables $\Psi$ and $\Phi$ as defined in~\eqref{eq:Extended} satisfy $(\mathsf{S}_{1}-\mathsf{B}_{1})\Psi=\Phi$.

We first establish the bigger system $\mathsf{S}_{1}$ for the \textit{local} problems.

\begin{Prp}\label{Prop:SHSDer} Let $\scrS=\sigma_{\scrS}(dx^{\mu})\nabla^{\scrE}_{\mu}-\scrS_{0}$ be a symmetric hyperbolic system over $(\scrE,\Sl\cdot|\cdot\Sr_{\scrE})$. Then, there exists a symmetric hyperbolic system $\mathsf{S}_{1}=\sigma_{\mathsf{S}_{1}}(dx^{\mu})\nabla^{\sf{E}_{1}}_{\mu}-\mathsf{S}_{1,0}$ on $(\mathsf{E}_{1},\Sl\cdot|\cdot\Sr_{\mathsf{E}_{1}})$ with principal symbol
\begin{align*}
	\sigma_{\mathsf{S}_{1}}(\xi):=\begin{pmatrix}
		\sigma_{\scrS}(\xi) &0\\ 0 & \sigma_{\scrS}(\xi)\otimes\mathrm{id}_{T^{\ast}\scrM}
	\end{pmatrix}\in \mathrm{End}(\mathsf{E}_{1}),\qquad\forall \xi\in T^{\ast}\scrM\, ,
\end{align*}
such that $\scrS\psi=\varphi$ if and only if $\mathsf{S}_{1}\Psi=\Phi$ for all $\psi,\varphi\in C^{\infty}(\scrM,\scrE)$, where we wrote $\Psi:=(\psi,\nabla^{\scrE}\psi)$ and $\Phi:=(\phi,\nabla^{\scrE}\phi)$ as in~\eqref{eq:Extended}. 
\end{Prp}

\begin{proof}
	First of all, it is clear by definition of the principal symbol $\sigma_{\mathsf{S}_{1}}\colon T^{\ast}\scrM\to\mathrm{End}(\mathsf{E}_{1})$ and the bundle metric on $\mathsf{E}_{1}$ that $\mathsf{S}_{1}$ is again a symmetric hyperbolic system, i.e.~that both the conditions (H) and (S) of Definition~\ref{Def:SHS} are satisfied. 
	
	For uniqueness, assume that there exists another symmetric hyperbolic system $\mathsf{S}_{1}^{\prime}$ with the same property and set $\mathsf{T}:=\mathsf{S}_{1}-\mathsf{S}_{1}^{\prime}$. Clearly, $\mathsf{T}$ is an operator of order zero. Now, let $\psi\in C^{\infty}(\scrM,\scrE)$ be arbitrary and set $\varphi:=\scrS\psi$. Then, $\mathsf{T}\Psi=\Phi-\Phi=0$ and hence $\mathsf{S}_{1}\Psi=\mathsf{S}_{1}^{\prime}\Psi$. Since $\psi$ was arbitrary and $\mathsf{T}$ an operator of order zero, it holds that $\mathsf{T}=0$ on all of $C^{\infty}(\scrM,\mathsf{E}_{1})$ and we conclude that $\mathsf{S}_{1}=\mathsf{S}_{1}^{\prime}$.
	
	Now, let us turn to the construction of $\mathsf{S}_{1}$. As a first step, we choose a local coordinates chart $(x^{\mu})_{\mu=0,\dots,k}$ on some open subset $\mathcal{U}$ of $\scrM$ and write $\scrS=\sigma_{\scrS}(dx^{\mu})\nabla^{\scrE}_{\mu}-\scrS_{0}$ for some zero-order operator $\scrS_{0}$. Now, the commutator $[\nabla_{\mu}^{\scrE},\scrS]$ for some fixed basis vector $\partial_{\mu}$ is an operator of order zero on $C^{\infty}(\mathcal{U},\scrE)$ and a straightforward computation shows that it is explicitly given by
	\begin{align}\label{eq:Comutator}
		[\nabla_{\mu}^{\scrE},\scrS]\psi=(\nabla^{\mathrm{End}(\scrE)}_{\mu}\sigma_{\scrS}(dx^{\alpha}))\nabla_{\alpha}^{\scrE}\psi+\sigma_{\scrS}(dx^{\alpha})\mathcal{F}_{\mu\alpha}(\psi)-(\nabla^{\mathrm{End}(\scrE)}_{\mu}\scrS_{0})(\psi)
	\end{align}
	 for all $\psi\in C^{\infty}(\scrM,\scrE)$, where $\nabla^{\mathrm{End}(\scrE)}$ denotes the induced connection on  $\mathrm{End}(\scrE)$, as usual, and where $\mathcal{F}\in\Omega^{2}(\scrM,\mathrm{End}(\scrE))$ defined by $\mathcal{F}(X,Y,\bullet):=[\nabla_{X}^{\scrE},\nabla_{Y}^{\scrE}]-\nabla^{\scrE}_{[X,Y]}$ for all $X,Y\in C^{\infty}(\scrM,\scrE)$ denotes the curvature $2$-form of the bundle $(\scrE,\nabla^{\scrE})$. Now, let $\psi\in C^{\infty}(\scrM,\scrE)$ and write $\nabla^{\scrE}\psi\in C^{\infty}(\scrM,\scrE_{1})$ locally as $\nabla^{\scrE}\psi=\nabla_{\mu}^{\scrE}\psi\otimes dx^{\mu}$. Then, we obtain
\begin{align*}
	\sigma_{\mathsf{S}_{1}}(dx^{\mu})\nabla_{\mu}^{\mathsf{E}_{1}}\Psi&=\begin{pmatrix}
	\sigma_{\scrS}(dx^{\mu})\nabla^{\scrE}_{\mu}\psi\\
	(\sigma_{\scrS}(dx^{\mu})\otimes\mathrm{id})\nabla^{\scrE_{1}}_{\mu}(\nabla^{\scrE}\psi))
	\end{pmatrix}\\&=\begin{pmatrix}
	(\scrS+\scrS_{0})\psi\\
	\sigma_{\scrS}(dx^{\mu})\nabla^{\scrE}_{\mu}\nabla^{\scrE}_{\alpha}\psi\otimes dx^{\alpha}+ \sigma_{\scrS}(dx^{\mu})\nabla^{\scrE}_{\alpha}\psi\otimes \nabla^{0}_{\mu}dx^{\alpha}
	\end{pmatrix}\\&=\begin{pmatrix}
	(\scrS+\scrS_{0})\psi\\
	(\scrS+\scrS_{0})(\nabla^{\scrE}_{\alpha}\psi)\otimes dx^{\alpha}-(\Gamma^{0})_{\mu\beta}^{\alpha}\sigma_{\scrS}(dx^{\mu})\nabla_{\alpha}^{\scrE}\psi\otimes dx^{\beta}
	\end{pmatrix}\\&=
	\begin{pmatrix}
		\scrS\psi\\
		\nabla^{\scrE}(\scrS\psi)
	\end{pmatrix}+\underbrace{\begin{pmatrix}
	\scrS_{0}\psi\\
	([\scrS,\nabla^{\scrE}_{\mu}]\psi+\scrS_{0}\nabla^{\scrE}_{\mu}\psi-(\Gamma^{0})_{\mu\beta}^{\alpha}\sigma_{\scrS}(dx^{\mu})\nabla_{\alpha}^{\scrE}\psi)\otimes dx^{\mu}
	\end{pmatrix}}_{=:\mathsf{S}_{1,0}\Psi}\, ,
\end{align*}
where $\Gamma^{0}$ are the Christoffel symbols of the auxiliary Riemannian manifold $(\scrM,g_{0})$. The operator $\mathsf{S}_{1,0}$ defined above clearly has order zero and we set 
\begin{align*}
	\mathsf{S}_{1}:=\sigma_{\mathsf{S}_{1}}(dx^{\mu})\nabla_{\mu}^{\mathsf{E}_{1}}-\mathsf{S}_{1,0}\, .
\end{align*}	
	 By construction, it holds that $\scrS\psi=\phi$ if and only if $\mathsf{S}_{1}\Psi=\Phi$ for all $\psi,\varphi\in C^{\infty}(\scrM,\scrE)$, where $\Psi:=(\psi,\nabla^{\scrE}\psi)$ and $\Phi:=(\phi,\nabla^{\scrE}\phi)$, which concludes the proof.
\end{proof}

Note that the zero-order term $\mathsf{S}_{1,0}\in C^{\infty}(\scrM,\mathrm{End}(\mathsf{E}_{1}))$ of the system $\sf{S}_{1}$ as constructed in the previous proof involves the derivatives of the principal symbol and $0^{\mathrm{th}}$-order terms of $\scrS$ as well as terms involving the bundle curvature of $(\scrE,\nabla^{\scrE})$. More precisely, using~\eqref{eq:Comutator}, it is explicitly given by
\begin{align}\label{eq:ZeroORderExt}
	&\mathsf{S}_{1,0}=\begin{pmatrix}
		\scrS_{0} & 0 \\ \scrA &\scrB
	\end{pmatrix}\colon C^{\infty}(\scrM,\mathsf{E}_{1})\to C^{\infty}(\scrM,\mathsf{E}_{1})
\end{align}
with zero-order operators $\scrA\colon C^{\infty}(\scrM,\scrE)\to C^{\infty}(\scrM,\scrE_{1})$ and $\scrB\colon C^{\infty}(\scrM,\scrE_{1})\to C^{\infty}(\scrM,\scrE_{1})$ given by
\begin{align*}
	\scrA\psi&:= \bigg(\sigma_{\scrS}(dx^{\alpha})\mathcal{F}_{\mu\alpha}-\nabla^{\mathrm{End}(\scrE)}_{\mu}\scrS_{0}\bigg)\psi\otimes dx^{\mu}\\
	\scrB\omega&:=\bigg(\delta^{\gamma}_{\alpha}\nabla^{\mathrm{End}(\scrE)}_{\mu}\sigma_{\scrS}(dx^{\alpha})-(\Gamma^{0})_{\alpha\mu}^{\gamma}\sigma_{\scrS}(dx^{\alpha})+\delta^{\gamma}_{\mu}\scrS_{0}\bigg)\omega_{\gamma}\otimes dx^{\mu}	
\end{align*}
where we expanded $\omega=\omega_{\alpha}\otimes dx^{\alpha}$ for coefficients $\omega_{\alpha}\in C^{\infty}(\mathcal{U},\scrE)$ in a local coordinate chart $(\mathcal{U},(x^{\mu})_{\mu=0,\dots,k})$ of $\scrM$.

\begin{Example}\label{Example:MinkoExtSys}
	Consider Minkowski spacetime $\scrM = \mathbb{R} \times \mathbb{R}^{k}$ and the trivial vector bundle $\scrE=\scrM\times\mathbb{C}^{N}$ of rank $N\in\N$, equipped with the standard complex inner product on its fibres, as discussed in Example~\ref{Ex:MinkSHS}. In this case, any symmetric hyperbolic system can be written in the form $\scrS = A^{\mu} \partial_{\mu} - \scrS_{0}$, where each $A^{\mu} \in C^{\infty}(\scrM, \mathbb{C}^{N \times N})$ is a smooth, pointwise Hermitian matrix-valued function, and $\scrS_{0}$ is a zero-order operator. In this case, $\mathsf{S}_{1}$ is given by
	\begin{align*}
		\mathsf{S}_{1}
		\begin{pmatrix}
			\psi\\ \omega_{\alpha}\otimes dx^{\alpha}
		\end{pmatrix}=
		\begin{pmatrix}
			A^{\mu}\partial_{\mu}\psi\\
			(A^{\mu}\partial_{\mu}\omega_{\alpha})\otimes dx^{\alpha}
		\end{pmatrix}-
		\begin{pmatrix}
			\scrS_{0}\psi\\
			\big((\partial_{\alpha}A^{\mu})\omega_{\mu}+\scrS_{0}\omega_{\alpha}-(\partial_{\alpha}\scrS_{0})\psi\big)\otimes dx^{\alpha}
		\end{pmatrix}
\end{align*}	 
for all $(\psi,\omega_{\alpha}\otimes dx^{\alpha})\in C^{\infty}(\scrM,\mathrm{End}(\mathsf{E}_{1}))$. Hence, we see that $\sf{S}_{1}$ is constructed out of the principal symbol $\sigma_{\scrS}(\xi)=A^{\mu}\xi_{\mu}$ and zero-order terms $\scrS_{0}$ of $\scrS$ and their derivatives. If the matrix-valued functions do actually not depend on the point and hence $\partial_{\alpha}A_{\beta}=0$ and $\partial_{\alpha}\scrS_{0}=0$, and if $\scrS_{0}$ is an imaginary number, it holds that $\scrS+\scrS^{\dagger}=0$ as well as $\mathsf{S}_{1}+\mathsf{S}_{1}^{\dagger}=0$. This is for example the case for the Dirac operator in Minkowski spacetime.
\end{Example}

Now, applying the procedure of the previous proposition inductively, we obtain a whole family of symmetric hyperbolic systems $\{\mathsf{S}_{k}\}_{k\geq 1}$ on the bundles
\begin{align*}
	\mathsf{E}_{0}:=\scrE,\qquad \mathsf{E}_{i+1}:=\mathsf{E}_{i}\oplus (\mathsf{E}_{i}\otimes T^{\ast}\scrM)\quad\forall i\geq 0
\end{align*}
equipped with the obvious bundle metrics and connections such that 
\begin{align*}
	\scrS\psi=\phi\qquad\Leftrightarrow\qquad\mathsf{S}_{i}\Psi_{i}=\Phi_{i},\qquad\Psi_{i}:=(\Psi_{i-1},\nabla^{\mathsf{E}_{i-1}}\Psi_{i-1})
\end{align*}

Now, suppose that $\B\colon C^{\infty}_{\mathrm{c}}(\scrM,\scrE)\to C^{\infty}(\scrM,\scrE)$ is a semi-regular linear operator. Note that the composition $\nabla^{\scrE}\B\colon C^{\infty}_{\mathrm{c}}(\scrM,\scrE)\to C^{\infty}(\scrM,\scrE_{1})$ is a gain a continuous operator with a Schwartz kernel. In the case of symmetric hyperbolic systems with nonlocal potentials, we obtain the following generalisation of the previous proposition.

\begin{Prp}\label{Prop:SHSDerNOnLoc}	Let $\scrS$ be a symmetric hyperbolic system over $(\scrE,\Sl\cdot|\cdot\Sr_{\scrE})$ and let $\B\colon C^{\infty}_{\mathrm{c}}(\scrM,\scrE)\to C^{\infty}(\scrM,\scrE)$ be a nonlocal potential. If $\mathsf{S}_{1}$ denotes the corresponding symmetric hyperbolic system on $\mathsf{E}_{1}$, then
\begin{align*}
	(\scrS-\B)\psi=\varphi\qquad\Leftrightarrow\qquad (\mathsf{S}_{1}-\mathsf{B}_{1})\Psi=\Phi\, ,
\end{align*}
where $\mathsf{B}_{1}\colon C^{\infty}_{\mathrm{c}}(\scrM,\mathsf{E}_{1})\to C^{\infty}(\scrM,\mathsf{E}_{1})$ denotes the operator defined by
\begin{align}\label{eq:ExtB1}
	\mathsf{B}_{1}:=\begin{pmatrix}
	\B & 0\\  \nabla^{\scrE}\circ\B- (\B\otimes\mathrm{id})\circ\nabla^{\scrE}  & \B\otimes\mathrm{id}
	\end{pmatrix}
\end{align}
\end{Prp}
\begin{proof}
	The proof follows the same steps as in Proposition~\ref{Prop:SHSDer}.
\end{proof}

\begin{Remark}
	Note that the operator $\sf{B}_{1}$ is in general not a nonlocal potential in the sense of Definition~\ref{Def:NonLocPot}, since it might not be locally integrable in time and hence not posses a time kernel according to our definition. If $k_{\B}$ is sufficiently regular in time, however, this is the case. Furthermore, in some cases, $\sf{B}_{1}$ can be split into a nonlocal potential with time kernel and a local operator, that can be reabsorbed in the zero-order terms $\sf{S}_{1,0}$ of the symmetric hyperbolic system. 
\end{Remark}

Similar to above, we define inductively the nonlocal potentials $\mathsf{B}_{i}\colon C^{\infty}_{\mathrm{c}}(\scrM,\mathsf{E}_{i})\to C^{\infty}(\scrM,\mathsf{E}_{i})$ in the bundles $\mathsf{E}_{i}=\mathsf{E}_{i-1}\oplus (\mathsf{E}_{i-1}\otimes T^{\ast}\scrM)$ by
\begin{align*}
	\mathsf{B}_{i}:=\begin{pmatrix}
	\mathsf{B}_{i-1} & 0\\ \nabla^{\mathsf{E}_{i-1}}\circ\mathsf{B}_{i-1}-(\mathsf{B}_{i-1}\otimes\mathrm{id})\circ\nabla^{\mathsf{E}_{i-1}} & \mathsf{B}_{i-1}\otimes\mathrm{id}
	\end{pmatrix}\, .
\end{align*}

\subsection{Retarded Nonlocalities} \label{Subsection:Ret}
In this section, we present the first main result concerning the analysis of the Cauchy problem for symmetric hyperbolic systems with nonlocal potentials. As a first result, we focus in this section on the case where the system is coupled to a \textit{retarded} nonlocal potential, as defined in Definition~\ref{Def:Potentials}(i). This setting is particularly suited for initial investigation due to its inherent causal structure. The central objective of this section is to proof that the Cauchy problem~\eqref{eq:SymNonLoc} is well-posed in this case. 

To start with, consider Set-up~\ref{Setup} and let $\overline{\scrM}_{T}:=[0,T]\times\scrN$ be the closed time strip of size $T>0$. If $\B$ is a retarded nonlocal potential, our aim is to study the Cauchy problem
\begin{align}\label{eq:CauchyRetarted}
	\begin{cases}
		(\scrS-\B)\psi &=\phi\\
		\psi\vert_{\scrN_{0}}&=\mathfrak{f}
	\end{cases}
\end{align}	
for a given source $\phi\in C_{\mathrm{c}}^{\infty}(\scrM,\scrE)$ with $\mathrm{supp}(\phi)\subset \overline{\scrM}_{T}$ and initial datum $\mathfrak{f}\in C^{\infty}_{\mathrm{c}}(\scrN_{0},\scrE\vert_{\scrN_{0}})$. As already specified at the beginning of Section~\ref{Sec:SHSNonLoc}, we additionally assume that the kernel $k_{\B}$ has \emph{past compact support}, meaning that $\B_{t,\tau} = 0$ whenever $\tau < -t_0$ for some fixed $t_0 > 0$. In the following, we restrict to the case $t_0 = 0$, i.e.~we prescribe initial data on the slice where the nonlocality is effectively ``switched on''. We comment on the more general case $t_0 \neq 0$, i.e.~in which the initial data are assigned slightly in the future of the switch-on time $t_{0}$, below in Remark~\ref{Remark:NonZeroIni}.

The main objective of this section is to prove the following theorem on the existence and uniqueness of the solutions to the Cauchy problem~\eqref{eq:CauchyRetarted}. 

\begin{Thm}[Cauchy Problem for Retarded Nonlocalities]\label{Thm:Ret} 
	Assume~Set-up~\ref{Setup} and suppose that the nonlocal potential $\B$ satisfies the following assumptions:
	\begin{itemize}
		\item[(i)]$\B$ is retarded and hence $\B_{t,\tau}=0$ for $\tau>t$ (see~Definition~\ref{Def:Potentials}(i)) and past compactly supported with switch-one time $t_{0}=0$, i.e.~$\B_{t,\tau}=0$ for $\tau<0$.
		\item[(ii)]$\B$ is uniformly bounded in the time strip, i.e.~there exists a constant $C_{T}>0$ such that $\Vert\V_{t,\tau}\varphi_{\tau}\Vert_{t}\leq C_{T}\Vert\varphi_{\tau}\Vert_{\tau}$ for all $\varphi\in C^{\infty}_{\mathrm{c}}(\scrN_{\tau},\scrE\vert_{\scrN_{\tau}})$ and $t,\tau\in [0,T]$.
		\end{itemize}
		Then, there exists a strong solution of the Cauchy problem~\ref{eq:CauchyRetarted} (see~Definition~\ref{Def:StrongSol}(R)). Furthermore, if we assume in addition that
	\begin{itemize}
		\item[(iii)]there is a $j\geq 1$ such that also $\mathsf{B}_{i}$ for $i=1,\dots,j$ are nonlocal potentials with time kernels satisfying the boundedness assumption (ii) as well\footnote{If $\mathsf{B}_{i}$ is a nonlocal potential with time kernel, then it is automatically retarded provided $\mathsf{B}_{i-1}$ is retarded, as one can easily see from the explicit expression.},
	\end{itemize}
	then, $\psi\in W^{2,j}(\scrM_{T},\scrE_{\beta})$ and in particular $\psi\in C^{l}(\overline{\scrM}_{T},\scrE)$ for $l\in\N$ such that $j\geq \frac{k+1}{2}+l$. In other words, $\psi$ is a classical solution to the Cauchy problem (see~Definition~\ref{Def:ClassSol}(R)). Furthermore, the solution $\psi$ is unique in this case and satisfies
	\begin{align*}
			\mathrm{supp}(\psi)\cap J^{+}(\scrN_{0})\subset J^{+}(\mathrm{supp}(\phi)\cup\mathrm{supp}(\mathfrak{f}))\, .
	\end{align*}
\end{Thm}

We will split the proof of this theorem in several parts. 

\begin{Prp}\label{Prop:DysonRet}
	Assume Set-up.~\ref{Setup} together with assumptions (i)-(ii) in Theorem~\ref{Thm:Ret} and let $\psi^{(n)}\in C^{\infty}(\scrM,\scrE)$ be the sections inductively defined as the (unique) solutions to the local Cauchy problems
	\begin{align*}
		\begin{cases}
			\scrS\psi^{(0)} &=\phi\\
			\psi^{(0)}\vert_{\scrN_{0}}&=\mathfrak{f}
		\end{cases}\qquad\text{and}\qquad 
		\begin{cases}
			\scrS\psi^{(n+1)} &=\B\psi^{(n)}\\
			\psi^{(n+1)}\vert_{\scrN_{0}}&=0
		\end{cases},\quad n\geq 0\, .
	\end{align*}
	Then, $\psi_{t}:=\sum_{n=0}^{\infty}\psi^{(n)}_{t}$ is absolutely convergent in the Hilbert space $\H_{t}$ for every $t\in [0,T]$. Furthermore, the map $\psi:t\mapsto\psi_{t}$ is a well-defined element in the space $\H_{\overline{\scrM}_{T}}$.
	
	Furthermore, if instead of assumption (ii), $\V_{t,\tau}$ is uniformly bounded on all of $\R$, i.e.~if there exists a constant $C>0$ such that $\Vert\V_{t,\tau}\varphi_{\tau}\Vert_{t}\leq C\Vert\varphi_{\tau}\Vert_{\tau}$ for all $\varphi\in C^{\infty}_{\mathrm{c}}(\scrN_{\tau},\scrE\vert_{\scrN_{\tau}})$ and $t\in\R$, then $\psi$ is a well-defined element of $L_{\mathrm{loc}}^{2}([0,\infty),\H_{\bullet})$.
\end{Prp}

\begin{Remark}
	The collection $(\H_{t})_{t\in\R}$ can be thought of as a bundle of Hilbert spaces over $\R$ and spaces such as $L^{2}_{\mathrm{loc}}([0,T],\H_{\bullet})$ are hence defined as the $L^{2}_{\mathrm{loc}}$-sections of this bundle.
\end{Remark}

\begin{proof}[Proof of Proposition~\ref{Prop:DysonRet}]
	Let $\psi^{(n)}$ be the solutions to the local Cauchy problems as defined in the statement of the proposition. By finite speed of propagation, see~Theorem~\ref{Thm:WellPosedCauchy}, together with assumptions (i) and (ii) of the nonlocal potential $\B$, it holds that 
	\begin{align*}
		\mathrm{supp}(\psi^{(0)})\cap J^{+}(\scrN_{0}) \subset J^{+}(\mathrm{supp}(\phi)\cup\mathrm{supp}(\mathfrak{f}))\quad\text{and}\quad \mathrm{supp}(\psi^{(n)}) \subset J^{+}(\mathrm{supp}(\phi)\cup\mathrm{supp}(\mathfrak{f}))
	\end{align*}
	 for all $n\geq 1$. In particular, note that all the support of all the sections $\psi^{(n)}$ is contained within the same set in the future of $\scrN_{0}$. Furthermore, note that $\B\psi^{(n)}$ is well-defined, since $\B$ extends to an operator acting on smooth sections, which are in $\H_{t}$ at any fixed $t\in \R$, as a consequence of assumptions (i), (ii) and the standard energy estimates.
	Now, the goal of the following discussion is to show that the perturbative ansatz
	\begin{align}\label{eq:SeriesAnsatz}
		\psi=\sum_{n=0}^{\infty}\psi^{(n)}
	\end{align}
	is absolutely convergent in $\mathcal{H}_{t}$ for every $t\in [0,T]$. Using the energy estimates from Corollary~\ref{Cor:EnEst}, we have
	\begin{align*}
		\frac{d}{dt}\Vert\psi^{(n+1)}_{t}\Vert_{t}\leq \int_{\R}\Vert \V_{t,\tau}\psi^{(n)}_{\tau}\Vert_{t}\,d\tau+\frac{1}{2}\Vert(\mathcal{Z}_{\scrS}\psi^{(n+1)})_{t}\Vert_{t}
	\end{align*}
	for all $n\in\N$ and $t\in [0,T]$, where $\mathcal{Z}_{\scrS}:=\beta\sigma_{\scrS}(\eta)^{-1}(\scrS+\scrS^{\dagger})$ is an operator of order zero. Since all the sections $\psi^{(n)}$ are contained within the same compact set on $\overline{\scrM}_{T}$, there exists a $D_{T}>0$ such that
	\begin{align}\label{eq:ZeroOrder}
		\Vert (\mathcal{Z}_{\scrS}\psi^{(n)})_{t}\Vert_{t}\leq D_{T}\Vert\psi^{(n)}_{t}\Vert_{t}
	\end{align}
	for all $t\in [0,T]$ and $n\in\N$. Using the assumptions (i)-(ii) as well as the fact that $\psi^{(n)}=0$ for $t<0$ and $n\geq 1$, we obtain the estimate
	\begin{align*}
		\frac{d}{dt}\Vert\psi^{(n+1)}_{t}\Vert_{t}&\leq C_{T}\int_{0}^{t}\Vert \psi^{(n)}_{\tau}\Vert_{\tau}\, d\tau+\frac{1}{2}D_{T}\Vert\psi_{t}^{(n+1)}\Vert_{t}\, ,
	\end{align*}
	for all $n\in\mathbb{N}$ and $t\in [0,T]$. Multiplying by $e^{-\frac{1}{2}Dt}$, this estimate can be written in the compact form
	\begin{align*}
		\frac{d}{dt}(e^{-\frac{1}{2}D_{T}t}\Vert\psi^{(n+1)}_{t}\Vert_{t})&\leq C_{T}e^{-\frac{1}{2}D_{T}t}\int_{0}^{t}\Vert \psi^{(n)}_{\tau}\Vert_{\tau}\, d\tau\leq C_{T}\int_{0}^{t}\Vert \psi^{(n)}_{\tau}\Vert_{\tau}\, d\tau
	\end{align*}
	and integrating over time from $0$ to $t\in [0,T]$, we obtain the inequality
	\begin{align*}
		\Vert\psi^{(n+1)}_{t}\Vert_{t} \leq C_{T}e^{\frac{1}{2}D_{T}t}\int_{0}^{t}\int_{0}^{\tau}\Vert \psi^{(n)}_{s}\Vert_{s}\, ds\,d\tau\,,
	\end{align*}
	where we used again the fact that $\psi^{(n)}=0$ for $t\leq 0$. Last but not least, we bound $e^{\frac{1}{2}D_{T}t}\leq e^{\frac{1}{2}D_{T}T}$ for all $t\in [0,T]$ to obtain the estimate
	\begin{align}\label{eq:Estimate}
		\Vert\psi^{(n+1)}_{t}\Vert_{t} \leq K_{T}\int_{0}^{t}(t-\tau)\Vert \psi^{(n)}_{\tau}\Vert_{\tau}\,d\tau
	\end{align}
	for all $n\in\mathbb{N}$ and $t\in [0,T]$, where we set $K_{T}:=C_{T}e^{\frac{1}{2}D_{T}T}>0$.

	Now, consider the induction hypothesis $\Vert\psi_{t}^{(n)}\Vert_{t}\leq \kappa(n)t^{2n}M_{T}$ for all $t\in [0,T]$, where \begin{align*}M_{T}:=\max_{t\in [0,T]}\Vert\psi^{(0)}_{t}\Vert_{t}\,.\end{align*} By finite speed of propagation, $M_{T}$ is clearly finite, and the ansatz trivially holds true for $n=0$ with $\kappa(0)=1$. By induction, we assume that it is true for $n$. Using the estimate~\eqref{eq:Estimate}, we obtain
	\begin{align*}
		\Vert\psi^{(n+1)}_{t}\Vert_{t} &\leq K_{T}\int_{0}^{t}(t-\tau)\Vert \psi^{(n)}_{t}\Vert_{\tau}\, d\tau\leq K_{T}M_{T}\kappa(n)\frac{t^{2n+2}}{(2n+1)(2n+2)}\, .
	\end{align*}
	This completes the inductive step and allows us to determine the coefficients $\kappa(n)$ in $\Vert\psi_{t}^{(n)}\Vert_{t}\leq \kappa(n)t^{2n}M_{T}$ as $\kappa(n)=\frac{K_{T}^n}{(2n)!}$. To sum up, we have shown that 
	\begin{align*}
		\forall t\in [0,T]\,,n\in\N:\quad \Vert\psi_{t}^{(n)}\Vert_{t}\leq \frac{K_{T}^n}{(2n)!}t^{2n}
	\end{align*}
	In particular, this implies that the perturbative series~\eqref{eq:SeriesAnsatz} is absolutely convergent in the $\H_{t}$-topology for every $t\in [0,T]$, since
	\begin{align*}
	\sum_{n=0}^{\infty}\Vert\psi^{(n)}_{t}\Vert_{t}\leq \sum_{n=0}^{\infty}\frac{M_{T}K_{T}^{n}t^{2n}}{(2n)!}=M_{T}\mathrm{cosh}(\sqrt{K_{T}}t)<\infty\, .
	\end{align*}
	
	Now, we have constructed a map $\psi:t\mapsto\psi_{t}$, where $\psi_{t}$ is the limit of the series~\ref{eq:SeriesAnsatz} in the $\H_{t}$-topology. Now, it is clear that $t\mapsto \psi_{t}$ is measurable, since it is the pointwise limit of the smooth (and in particular strongly measurable) functions $t\mapsto \sum_{k=0}^{n}\psi^{(k)}_{t}$ for $n\to\infty$, see e.g.~\cite[Thm.~1.14]{AmannEscherIII}. To see that $\psi$ is an element of $L^{2}([0,T],\H_{\bullet})$, we note that the above estimates imply
\begin{align*}
		\Vert\psi\Vert_{\scrM_{T}}^{2}=\int_{0}^{T}\Vert\psi_{t}\Vert_{t}^{2}\,dt\leq M_{T}^{2}\int_{0}^{T} \mathrm{cosh}(\sqrt{K_{T}}t)^{2}\,dt=M_{T}^{2}\bigg(\frac{T}{2}+\frac{\mathrm{sinh}(2\sqrt{K_{T}}T)}{4\sqrt{K_{T}}}\bigg)<\infty\,.
	\end{align*}
	It is easy to see that the series~\eqref{eq:SeriesAnsatz} converges also in the $\H_{\overline{\scrM}_{T}}$-topology and by the dominant convergence theorem for series (see e.g.~\cite[Thm.~3.16]{AmannEscherIII}), we conclude that the corresponding limit coincides with $\psi$.
	
	If $\V_{t,\tau}$ is uniformly bounded for all $t\in [0,\infty)$ by a constant $C>0$ independent of time, we can repeat the same argument as above for every choice of $T>0$. In particular, we conclude that the series~\eqref{eq:SeriesAnsatz} converges in the $\H_{\overline{\scrM}_{T}}$ topology for every fixed $T>0$ in this case, which proves the claim.
\end{proof}

\begin{Remark}\label{Remark:NonZeroIni} 
	Let us remark that in the proof of in Proposition~\ref{Prop:DysonRet} , it was crucial that the retarded nonlocal potential $\B$ had past compact support. In fact, we needed two additional assumptions to be considered separately:
	\begin{itemize}
		\item[(i)]The nonlocal potential $\B$ has past compact support, i.e.~$\B_{t,\tau}=0$ for $\tau<-t_{0}$ and some fixed $t_{0}>0$.
		\item[(ii)]We assign initial data on a Cauchy slice that is in the past of $\scrN_{-t_{0}}$. Without loss of generality, we hence choose the foliation such that $t_{0}=0$ and assign initial data on $\scrN_{0}$.
	\end{itemize}
	Now, assumption (i) is crucial for the the local solutions $\psi^{(n)}$ constructed in Proposition~\ref{Prop:DysonRet} to be well-defined. Recall that our solution $\psi$ to the Cauchy problem~\eqref{eq:CauchyRetarted} was defined as the limit of the series of the solutions $\psi^{(n)}$ to the local Cauchy problems
	\begin{align*}
		\begin{cases}
			\scrS\psi^{(0)} &=\phi\\
			\psi^{(0)}\vert_{\scrN_{0}}&=\mathfrak{f}
		\end{cases}\qquad\text{and}\qquad 
		\begin{cases}
			\scrS\psi^{(n+1)} &=\B\psi^{(n)}\\
			\psi^{(n+1)}\vert_{\scrN_{0}}&=0
		\end{cases},\quad n\geq 0\, .
	\end{align*}
	for a given source $\phi\in C^{\infty}_{\mathrm{c}}(\overline{\scrM}_{T},\scrE)$ and initial datum $\mathfrak{f}\in C^{\infty}_{\mathrm{c}}(\scrN_{0},\scrE\vert_{\scrN_{0}})$. Now, in general, the solutions $\psi^{(n)}$ also propagate to the past and thus, without assumption (i), already at the step $n = 1$, it is not clear whether the right-hand side $\B\psi^{(0)}$ can be consistently defined. Using the energy estimates (see Corollary~\ref{Cor:EnEstLoc}), we know that $\Vert\psi^{(0)}_{t}\Vert_{t} \sim e^{|t|}$, and hence,
\begin{align*}
	\Vert(\B\psi^{(0)})_{t}\Vert_{t}= \int_{-\infty}^{t} \Vert\B_{t,\tau} \psi^{(0)}_{\tau}\Vert_{t} \, d\tau
\end{align*}
is still well-defined, provided that the kernel $\B_{t,\tau}$ decays sufficiently fast as $\tau \to -\infty$. In general, however, it is unclear whether one can formulate a criterion on the kernel ensuring that $\B\psi^{(n)}$ is consistently defined for all $n \in \mathbb{N}$, beyond examining specific examples.\medskip

	Concerning the case where the nonlocal potential is ``switched on'' at a time  $-t_0 < 0$,  it is possible to readapt the estimates performed in the proof of Proposition~\ref{Prop:DysonRet} to prove the convergence of the perturbative series $\sum_n \psi^{(n)}$, provided that $t_0$ and $C$, the uniform bound of the nonlocal potential $\B$, satisfy
$$ t_0^2 e^{\frac{1}{2}Dt_0}  C <1 \,.$$ 
\end{Remark}

Next, we show that the element $\psi\in L^{2}_{\mathrm{loc}}([0,T],\H_{\bullet})$ constructed in the previous proposition defines a strong solution to the Cauchy problem~\eqref{eq:CauchyRetarted}.

\begin{Prp}\label{Prp:Ret} Assume~Set-up~\ref{Setup} together with assumptions (i)-(ii) in Theorem~\ref{Thm:Ret}. Then, there exists a strong solution of the Cauchy problem~\eqref{eq:SymNonLoc} on $\overline{\scrM}_{T}$. 
\end{Prp}

\begin{proof}
	Consider the sequence $(\psi^{(n)})_{n\in\N}$ as constructed in Proposition~\ref{Prop:DysonRet}. As shown in this proposition, the corresponding series converges absolutely in the $\H_{\overline{\scrM}_{T}}$ topology to some element $\psi\in\H_{\overline{\scrM}_{T}}$. We will now show that $\psi$ is a {strong} solution to the Cauchy problem~\eqref{eq:SymNonLoc}. For this, we write
	\begin{align*}
		\psi_{k}:=\sum_{n=0}^{k}\psi^{(n)}\in C^{\infty}(\scrM,\scrE)
	\end{align*}
	for $k\in\N$. As shown above, it holds that $\Vert\psi_{k}-\psi\Vert_{\scrM_{T}}\to 0$ as $k\to\infty$. Furthermore, it is clear that also $\Vert\psi_{k}\vert_{\scrN_{0}}-\mathfrak{f}\Vert_{0}\xrightarrow{k\to\infty} 0$, since $\psi_{k}\vert_{\scrN_{0}}=\mathfrak{f}$ for all $k\in\N$, by construction. Note also that
	\begin{align*}
		\Vert\psi^{(n)}\Vert_{\scrM_{T}}\xrightarrow{n\to\infty}0\, ,
	\end{align*}
	 since otherwise the series~\eqref{eq:SeriesAnsatz} would not converge absolutely. 

	On the other hand, it holds that $(\scrS-\B)\psi_{k}=\phi-\B\psi^{(k)}$, by construction, and hence, we get
	\begin{align*}
		\|(\scrS-\B)\psi_{k}-\phi\|^2_{\scrM_{T}}&=(\B\psi^{(k)}|\B\psi^{(k)})_{\scrM_{T}}= \int_{0}^{T}\,dt\,\int_{0}^{t}\,d\tau\,\int_{0}^{t}\,ds\,(\B_{t,\tau}\psi_{\tau}^{(k)}|\B_{t,s}\psi_{s}^{(k)})_{t}\\
		&\leq (C^{\prime}_{T})^{2}\int_{0}^{T}\,dt\,\int_{0}^{t}\,d\tau\,\int_{0}^{t}\,ds\,\Vert\psi_{\tau}^{(k)}\Vert_{\tau}\Vert\psi_{s}^{(k)}\Vert_{s} \\
		 &\leq T^{2}(C^{\prime}_{T})^{2} \Vert \psi^{(k)}\Vert_{\scrM_{T}}^2\xrightarrow{k\to\infty}0\, ,
	\end{align*}
	where we used Hölder's inequality. Furthermore, we used the following: By assumption, it holds that $\Vert\V_{t,\tau}\psi^{(k)}_{\tau}\Vert_{t}\leq C_{T}\Vert\psi^{(k)}_{\tau}\Vert_{\tau}$ for all $k\in\mathbb{N}$ and $t,\tau\in [0,T]$, where $\V=\beta\sigma_{\scrS}(\eta)^{-1}\B$. Now, since $\psi^{(k)}$ for arbitrary $k\in\mathbb{N}$ are supported within the same set as explained in the proof of Proposition~\ref{Prop:DysonRet}, the norm $\Vert\B_{t,\tau}\psi^{(k)}_{\tau}\Vert_{t}$ can be bounded by $\Vert\V_{t,\tau}\psi^{(k)}_{\tau}\Vert_{t}$ on $[0,T]$ and vice versa. Hence, we obtain a similar estimate for $\B_{t,\tau}$ for some constant $C^{\prime}_{T}>0$. 
	This conclude our proof.
\end{proof}
\black

\begin{Prp}[Regularity of Solutions]\label{Prp:RetReg} 		If condition (iii) in Theorem~\ref{Thm:Ret} is satisfied for some $j\in\N$ holds, then the strong solution $\psi$ constructed in Proposition~\eqref{Prp:Ret} is in the Sobolev space $W^{2,j}(\scrM_{T},\scrE_{\beta})$. In particular, by the Sobolev embedding theorem, $\psi$ is a classical solution with $\psi\in C^{l}(\overline{\scrM}_{T},\scrE)$ for $j\geq \frac{k+1}{2}+l$.
\end{Prp}

\begin{proof}
	We proof the claim by induction: Let $\psi\in\H_{\overline{\scrM}_{T}}$ be the strong solution constructed in Proposition~\ref{Prp:Ret}, i.e.~$\psi=\sum_{k=0}^{\infty}\psi^{(k)}$, where $\psi^{(k)}\in C^{\infty}_{\mathrm{sc}}(\scrM,\scrE)$ are the unique solutions to the local Cauchy problems
	\begin{align*}
		\begin{cases}
			\scrS\psi^{(0)} &=\phi\\
			\psi^{(0)}\vert_{\scrN_{0}}&=\mathfrak{f}
		\end{cases}
		\qquad\text{and}\qquad  
		\begin{cases}
			\scrS\psi^{(n+1)} &=\B\psi^{(n)}\\
			\psi^{(n+1)}\vert_{\scrN_{0}}&=0
		\end{cases},\quad n\geq 0\, .
	\end{align*} 
	Now, consider the system $\mathsf{S}_{1}$ on 	$\mathsf{E}_{1}$ as defined in Proposition~\ref{Prop:SHSDer} and Proposition~\ref{Prop:SHSDerNOnLoc}. If the nonlocal potential $\mathsf{B}_{1}$ satisfies the same conditions as $\B$, we can construct a strong solution $\Psi$ of the nonlocal Cauchy problem
	\begin{align}\label{sdfgdS}
		\begin{cases}
			(\mathsf{S}_{1}-\mathsf{B}_{1})\Psi=\Phi\\
			\Psi\vert_{\scrN_{0}}=\mathfrak{F}
		\end{cases}
	\end{align}
	with $\Phi=(\phi,\nabla^{\scrE}\phi)$ and where the initial data $\mathfrak{F}\in C^{\infty}(\scrN_{0},\mathsf{E}_{1}\vert_{\scrN_{0}})$ is given by
	\begin{align*}
		\mathfrak{F}=(\mathfrak{f},\sigma_{\scrS}(dt)^{-1}(\phi\vert_{\scrN_{0}}+\scrS_{0}\mathfrak{f}-\sigma_{\scrS}(dx^{i})\nabla_{i}^{\scrE}\mathfrak{f})\otimes dt+\nabla_{\partial_{i}}^{\scrE}\mathfrak{f}\otimes dx^{i})
	\end{align*}		
	 by considering the perturbative ansatz $\Psi=\sum_{k=0}^{\infty}\Psi^{(k)}$ with $\Psi^{(k)}$ being the unique solutions to the local problems
	\begin{align*}
		\begin{cases}
			\mathsf{S}_{1}\Psi^{(0)} &=\Phi\\
			\Psi^{(0)}\vert_{\scrN_{0}}&=\mathfrak{F}
		\end{cases}
		\qquad\text{and}\qquad  
		\begin{cases}
			\mathsf{S}_{1}\Psi^{(n+1)} &=\mathsf{B}_{1}\Psi^{(n)}\\
			\Psi^{(n+1)}\vert_{\scrN_{0}}&=0
		\end{cases},\quad n\geq 0\, .
	\end{align*} 
	Now, due to the past-compactness assumption on $\B$, it holds that $(\B\psi^{(k)})\vert_{\scrN_{0}}=0$ for all $k\in\N$. In particular, this implies that the local solutions $\psi^{(n)}$ satisfy
	\begin{align*}
		\nabla^{\scrE}\psi^{(n)}\vert_{\scrN_{0}}=0\, .
	\end{align*}
	In particular, by uniqueness of solutions to the local equations, we conclude that in fact $\Psi^{(n)}=(\psi^{(n)},\nabla^{\scrE}\psi^{(n)})$ for all $n\in\N$.
	
	Now, let us write $\eta:=\mathrm{pr}_{2}\Psi$. By construction, $\nabla^{\scrE}\psi_{n}\to\eta$ where $\psi_{n}:=\sum_{k=0}^{n}\psi^{(k)}$ as $k\to\infty$ in $\H_{\overline{\scrM}_{T}}$. In particular, this shows that the sequence $(\psi_{n})_{n}$ converges also in $W^{2,1}(\scrM_{T},\scrE_{\beta})$, since 
	\begin{align*}
		\Vert \psi_{n}-\psi_{m}\Vert_{W^{2,1}}^{2}=\Vert \psi_{n}-\psi_{m}\Vert_{L^{2}}^{2}+\Vert \nabla^{\scrE}\psi_{n}-\nabla^{\scrE}\psi_{m}\Vert_{L^{2}}^{2}\to 0\, ,
	\end{align*}
	cf.~Remark~\ref{Rem:SobSpL2}. Now, it is clear that the limit of the sequence $(\psi_{n})_{n}$ in $W^{2,1}(\scrM_{T},\scrE_{\beta})$ coincides with $\psi$, which is the limit of $(\psi_{n})_{n}$ in $\H_{\scrM_{t}}$, since $\Vert\cdot\Vert_{\scrM_{T}}\leq \Vert\cdot\Vert_{W^{2,1}(\scrM_{T},\scrE_{\beta})}$. Hence, we conclude that $\psi\in W^{2,1}(\scrM_{T},\scrE_{\beta})$. Proceeding by induction, we obtain $\psi\in W^{2,j}(\scrM_{T},\scrE_{\beta})$ whenever $\mathsf{S}_{k}$ and $\mathsf{B}_{k}$ satisfy the assumptions (i)-(iii) up to order $j$.

	In order to show that the solution is actually a classical solution, we choose a partition of unity $\{\rho_i\}$ subordinated to precompact chart $\mathcal{U}_i$ and a corresponding local trivialisation of $\scrE$. Then it is clear that $\psi_i=\rho_i \psi \in W^{2,j}(\mathcal{U}_i\,, \C^N)$ and by the Sobolev embedding theorem we can conclude that $\psi_i\in C^l(\mathcal{U}_i\,,\C^N)$ for $l\leq j-\frac{k+1}{2}$, for any $i$. Since $\rho_i$ are smooth, it follows that $\psi\in C^l(\mathcal{U}_i\,,\C^N)$ and in particular we can conclude our proof.
\end{proof}

\begin{Prp}[Uniqueness and Finite Speed of Propagation]\label{Prop:UniquenessRetarded} 
	Assume~Set-up~\ref{Setup} and suppose that the nonlocal potential $\B$ satisfies the assumptions (i) and (ii) in Theorem~\ref{Thm:Ret}. Then, if there exists a (classical) $C^{1}$-solution $\psi$ of the Cauchy problem~\eqref{eq:SymNonLoc} on the time strip $\overline{\scrM}_{T}$, it is unique and and propagates at most with the speed of light, i.e.
	\begin{align*}
		\mathrm{supp}(\psi)\cap J^{+}(\scrN_{0})\subset J^{+}(\mathrm{supp}(\phi)\cup\mathrm{supp}(\mathfrak{f}))\, .
	\end{align*}
\end{Prp}

\begin{proof}
	We first prove uniqueness. Suppose that $\psi\in C^{1}(\overline{\scrM}_{T},\scrE)$ is a solution of the homogeneous nonlocal equation $(\scrS-\B)\psi=0$ of with initial condition $\psi\vert_{\scrN_{0}}=0$. Then, following exactly the same steps as in the proof of Proposition~\ref{Prp:Ret}, cf.~\eqref{eq:Estimate} with replacing $\psi^{(n+1)}$ and $\psi^{(n)}$ by $\psi$, we obtain the estimate 
	\begin{align*}
		\Vert\psi_{t}\Vert_{t} \leq K_{T}\int_{0}^{t}(t-\tau)\Vert \psi_{\tau}\Vert_{\tau}\,d\tau \leq K_{T}T\int_{0}^{t}\Vert \psi_{\tau}\Vert_{\tau}\,d\tau=:K_{T}TF(t)
	\end{align*}
	for all $t\in [0,T]$, where $K_{T}>0$ is some constant only depending on $T$. Now, observe that $\dot{F}(t)=\Vert\psi\Vert_{t}$, where the dot denotes the time-derivative, which allows us to rewrite the above inequality as $\dot{F}(t)\leq K_{T}TF(t)$. In other words, we find that
	\begin{align*}
		\frac{d}{dt}(e^{-K_{T}Tt}F(t))\leq 0
	\end{align*}
	for all $t\in [0,T]$. Setting $G(t):=e^{-K_{T}Tt}F(t)$, we observe that $G(t)\geq 0$ and $\dot{G}(t)\leq 0$ as well as $G(0)=0$. The only such $C^{1}$-function is given by $G(t)=0$. We conclude that $\Vert\psi\Vert_{t}=0$ for all $t\in [0,T]$ and hence $\psi=0$ on $\scrM_{T}$.
	
	Now, if $\psi\in\H_{\overline{\scrM}_{T}}\cap C^{1}(\scrM_{T},\scrE)$ is a $C^{1}$-solution to the Cauchy problem $(\scrS-\B)\psi=\phi$ with $\psi\vert_{\scrN_{0}}=\mathfrak{f}$, uniqueness implies that $\psi$ is constructed via the perturbative series
	\begin{align*}
		\psi=\sum_{n=0}^{\infty}\psi^{(n)}\, ,
\end{align*}	
which is absolutely convergent in $\H_{\overline{\scrM}_{T}}$, where $\psi^{(n)}$ are the solutions to the local problems $\scrS\psi^{(0)}=\phi$ and $\scrS\psi^{(n+1)}=\B\psi^{(n)}$ with $\psi^{(n)}\vert_{\scrN_{0}}=0$, as explained in Proposition~\ref{Prp:Ret}. As explained in the proof of Proposition~\ref{Prop:DysonRet}, by the assumptions on $\B$, we know that 
\begin{align*}
	\mathrm{supp}(\psi^{(n)})\cap J^{+}(\scrN_{0})\subset J^{+}(\mathrm{supp}(\phi)\cap\mathrm{supp}(\mathfrak{f}))
\end{align*}
for all $n\in\N$. Hence, also the limit $\psi$, which is assumed to be $C^{1}$, is supported within the same set $J^{+}(\mathrm{supp}(\phi)\cup\mathrm{supp}(\mathfrak{f}))$. 
\end{proof}

Collecting all the results obtained, we have all the ingredients to proof the main theorem of this section.

\begin{proof}[Proof of Theorem~\ref{Thm:Ret}]
	By Proposition~\ref{Prp:Ret}, we know that the nonlocal Cauchy problem~\eqref{eq:SymNonLoc} has a strong solution provided its nonlocal potentials satisfies assumptions (i) and (ii). Then, in Proposition~\ref{Prp:RetReg}, we discussed the claimed regularity of solutions provided assumption (iii) holds true. Uniqueness and finite speed of propagation was then shown in Proposition~\ref{Prop:UniquenessRetarded}.
\end{proof}

\subsection{Nonlocal Potential with Short Time Range}
As a second case, we turn our attention to the Cauchy problem~\eqref{eq:SymNonLoc} in the presence of a nonlocal potential with short time range, as specified in Definition~\ref{Def:Potentials}(ii). This class of potentials, being non-retarded, introduces additional analytical challenges due to the lack of a purely causal structure. 

To start with, consider Set-up~\ref{Setup} and let $\overline{\scrM}_{T}:=[0,T]\times\scrN$ be the closed time strip of size $T>0$. If $\B$ is a nonlocal potential with short time range $\delta>0$, our aim is to study the Cauchy problem
\begin{align}\label{eq:CauchySmall}
	\begin{cases}
		(\scrS-\B)\psi &=\phi\\
		\psi\vert_{\scrN_{0}}&=\mathfrak{f}
	\end{cases}
\end{align}	
for a given source $\phi\in C_{\mathrm{c}}^{\infty}(\scrM,\scrE)$ with $\mathrm{supp}(\phi)\subset \overline{\scrM}_{T}$ and initial datum $\mathfrak{f}\in C^{\infty}_{\mathrm{c}}(\scrN_{0},\scrE\vert_{\scrN_{0}})$. 

Our aim in this section is to proof the following theorem.

\begin{Thm}\label{Thm:Small}
	Assume~Set-up~\ref{Setup} and suppose that the nonlocal potential $\B$ satisfies the following assumptions:
	\begin{itemize}
		\item[(i)]$\B$ has short time range $\delta>0$, i.e.~$\V_{t,t^{\prime}}=0$ if $\vert t-t^{\prime}\vert>\delta$.
		\item[(ii)]$\B$ is uniformly bounded in time with a constant that is small compared to $\delta$ and $D$ in the following sense: There exists a constant $C<(8e\delta^{2})^{-1}$ such that 
		\begin{align*}
			\Vert\V_{t,\tau}\varphi_{\tau}\Vert_{t}\leq Ce^{-\frac{1}{2}D\vert\tau\vert}\Vert\varphi_{\tau}\Vert_{\tau}\end{align*} 
		for all $\varphi\in C^{\infty}_{\mathrm{c}}(\scrN_{\tau},\scrE\vert_{\scrN_{\tau}})$ and $t,\tau\in\R$.
		\item[(iii)]The zero-order operator $\mathcal{Z}_{\scrS}:=\beta\sigma_{\scrS}(\eta)^{-1}(\scrS+\scrS^{\dagger})$ is uniformly bounded in time, i.e.~there exists a $D>0$ such that \begin{align*}
			\Vert(\mathcal{Z}_{\scrS}\psi)_{t}\Vert_{t}\leq D\Vert\psi_{t}\Vert_{t}
		\end{align*}
		for all $\varphi\in C^{\infty}_{\mathrm{sc}}(\scrM,\scrE)$ and $t\in\R$. Furthermore, also $\sigma_{\scrS}(dt)\in C^{\infty}(\scrM,\mathrm{End}(\scrE))$ is uniformly bounded in time.
	\end{itemize}
	Then, there exists a strong solution of the Cauchy problem~\eqref{eq:SymNonLoc}. \end{Thm}

\begin{Remark}
Let us highlight the key differences in the assumptions compared to the retarded case discussed in Theorem~\ref{Thm:Ret}:
\begin{itemize}
	\item[$\bullet$]Recall that strong solutions are constructed perturbatively by generating a sequence of solutions $\psi^{(n)}$ to \emph{local} equations, see~Proposition~\ref{Prop:DysonSmall} below. In this case, due to the non-retarded nature of the potential, the domain of integration in the energy estimates of Corollary~\eqref{Cor:EnEst} expands by $\pm\delta$ at each step. As a consequence, we must impose an exponential decay factor on the bounds for the time kernel. This decay compensates for the exponential growth induced by the zero-order terms associated with the operator $\scrS + \scrS^{\dagger}$.
	\item[$\bullet$] Finally, note that in this case, we do not have any uniqueness or finite speed of propagation results. The standard energy estimates used in the retarded case are not applicable here. However, in Section~\ref{secdirunique}, we will show that in the special case of a symmetric kernel, a uniqueness statement can still be established by considering a nonlocal modification of the inner product $(\cdot|\cdot)_{t}$.
\end{itemize}
\end{Remark}

As in the retarded case, we start first by constructing a sequence of solutions to local problems converging at any fixed time in the corresponding Hilbert space of the Cauchy slice.

\begin{Prp}\label{Prop:DysonSmall}
	Assume Set-up.~\ref{Setup} together with assumptions (i)-(iii) in Theorem~\ref{Thm:Small} and let $\psi^{(n)}\in C^{\infty}(\scrM,\scrE)$ be the sections inductively defined as the (unique) solutions to the local Cauchy problems
	\begin{align*}
		\begin{cases}
			\scrS\psi^{(0)} &=\phi\\
			\psi^{(0)}\vert_{\scrN_{0}}&=\mathfrak{f}
		\end{cases}\qquad\text{and}\qquad 
		\begin{cases}
			\scrS\psi^{(n+1)} &=\B\psi^{(n)}\\
			\psi^{(n+1)}\vert_{\scrN_{0}}&=0
		\end{cases},\quad n\geq 0\, .
	\end{align*}
	Then, $\psi_{t}:=\sum_{n=0}^{\infty}\psi^{(n)}_{t}$ is absolutely convergent in the Hilbert space $\H_{t}$ for every $t\in\R$. Furthermore, the map $\psi:t\mapsto\psi_{t}$ is a well-defined element in the space $L^{2}_{\mathrm{loc}}(\R,\H_{\bullet})$. 
\end{Prp}

\begin{proof}[Proof of Proposition~\ref{Prop:DysonSmall}.]
Let $\psi^{(n)}$ be the solutions to the local Cauchy problems as defined in the statement of the proposition. By finite speed of propagation, see~Theorem~\ref{Thm:WellPosedCauchy}, it holds that 
	\begin{align*}
		\mathrm{supp}(\psi^{(0)}) \subset J(\mathrm{supp}(\phi)\cup\mathrm{supp}(\mathfrak{f}))\quad\text{and}\quad \mathrm{supp}(\psi^{(n+1)}) \subset J(\mathrm{supp}(\B\psi^{(n)}))
	\end{align*}
	 for all $n\in\N$. Note that $\B\psi^{(n)}$ is well-defined, since $\B$ extends to an operator acting on smooth sections, which are in $\H_{t}$ at any fixed $t\in\R$, as a consequence of assumption (iii) and the standard energy estimates. Moreover, unlike in the retarded case, observe that $\psi^{(n)}$ will in general also include a component propagating into the past, both due to the non-zero initial datum $\mathfrak{f}$ for $\psi^{(0)}$ and the nonlocality, which increases the support of the source at each step $n$ by $\pm\delta$ in time, and hence, at some order, the source will thus intersect the initial Cauchy slice $\scrN_{0}$.
	
	Now, the goal of the following discussion is to show that the perturbative ansatz
\begin{align}\label{ansatzpertur2}
		\psi=\sum_{n=0}^{\infty}\psi^{(n)}
	\end{align}	
	is absolutely convergent in $\mathcal{H}_{t}$ for all $t\in\R$. Using the energy estimates from Corollary~\ref{Cor:EnEst}, we have
\begin{align}\label{eq:EnergyEstSmall}
		\bigg\vert\frac{d}{dt}\Vert\psi^{(n+1)}_{t}\Vert_{t}\bigg\vert\leq C\int_{t-\delta}^{t+\delta}e^{-\frac{1}{2}D\vert\tau\vert}\Vert \psi^{(n)}_{\tau}\Vert_{\tau}\, d\tau+\frac{1}{2}D\Vert\psi^{(n+1)}_{t}\Vert_{t}
	\end{align}
	for all $n\in\N$ and $t\in\R$, where we used assumptions (i)-(iii). As a first step, we absorb the last term in~\eqref{eq:EnergyEstSmall} using an appropriate exponential factor, by noting that the energy estimate~\eqref{eq:EnergyEstSmall} imply
	\begin{subequations}
	\begin{align}
		\frac{d}{dt}\bigg(e^{-\frac{1}{2}D t}\Vert\psi^{(n+1)}_{t}\Vert_{t}\bigg)\leq Ce^{-\frac{1}{2}D t}\int_{t-\delta}^{t+\delta}e^{-\frac{1}{2}D\vert\tau\vert}\Vert \psi^{(n)}_{\tau}\Vert_{\tau}\, d\tau\,,\quad &\text{for }t\geq 0\label{eq:geqt}\\
		-\frac{d}{dt}\bigg(e^{\frac{1}{2}D t}\Vert\psi^{(n+1)}_{t}\Vert_{t}\bigg)\leq Ce^{\frac{1}{2}D t}\int_{t-\delta}^{t+\delta}e^{-\frac{1}{2}D\vert\tau\vert}\Vert \psi^{(n)}_{\tau}\Vert_{\tau}\, d\tau\,,\quad &\text{for }t\leq 0\label{eq:leqt}
	\end{align}
	\end{subequations}
	for all $n\in\N$. Now, we distinguish between to cases. If $t\geq 0$, we integrate~\eqref{eq:geqt} from $0$ to $t$ and, using the fact that $\psi^{(n+1)}_{t=0}=0$, we obtain the estimate\footnote{In the first estimate, we changed the order of integration by using that for every choice of $\delta>0$ and $t\geq 0$ it holds that \begin{align*}\begin{cases}0 \leq \tau \leq t\\ \tau-\delta \leq s\leq\tau+\delta\end{cases}\qquad\Leftrightarrow\qquad\begin{cases}-\delta\leq s \leq t+\delta\\ \mathrm{max}\{s-\delta,0\} \leq \tau\leq\mathrm{min}\{s+\delta,t\}\end{cases}\, .\end{align*}
	Note also that $\mathrm{max}\{s-\delta,0\}\leq \mathrm{min}\{s+\delta,t\}$ on the range $s\in [-\delta,t+\delta]$ for all $\delta>0$, $t\geq 0$.}
		\begin{align}
		\Vert\psi_{t}^{(n+1)}\Vert_{t}&=Ce^{\frac{1}{2}Dt}\int_{0}^{t}\underbrace{e^{-\frac{1}{2}D\tau}}_{\leq 1}\int_{\tau-\delta}^{\tau+\delta}e^{-\frac{1}{2}D\vert s\vert}\Vert\psi_{s}\Vert_{s}\,ds\,d\tau\nonumber\\&\leq Ce^{\frac{1}{2}Dt}\int_{-\delta}^{t+\delta}\bigg(\underbrace{\mathrm{min}\{s+\delta,t\}-\mathrm{max}\{s-\delta,0\}}_{\leq 2\delta}\bigg)e^{-\frac{1}{2}D\vert s\vert}\Vert\psi^{(n)}_{s}\Vert_{s}\,ds\nonumber\\&\leq (2\delta C)e^{\frac{1}{2}Dt}\int_{-\delta}^{t+\delta}e^{-\frac{1}{2}D\vert s\vert}\Vert\psi^{(n)}_{s}\Vert_{s}\,ds\label{eq:EstPost}
	\end{align}
	for all $t\geq 0$ and $n\in\mathbb{N}$. 
	
	Similarly, in the case $t\leq 0$, we integrate over the estimate~\eqref{eq:leqt} from $t$ to $0$. Using again that $\psi^{(n+1)}_{t=0}=0$, we obtain the similar estimate
	\begin{align}
		\Vert\psi_{t}^{(n+1)}\Vert_{t}&=Ce^{-\frac{1}{2}Dt}\int_{t}^{0}\underbrace{e^{\frac{1}{2}D\tau}}_{\leq 1}\int_{\tau-\delta}^{\tau+\delta}e^{-\frac{1}{2}D\vert s\vert}\Vert\psi_{s}\Vert_{s}\,ds\,d\tau\nonumber\\&\leq Ce^{-\frac{1}{2}Dt}\int_{t-\delta}^{\delta}\bigg(\underbrace{\mathrm{min}\{s+\delta,0\}-\mathrm{max}\{s-\delta,t\}}_{\leq 2\delta}\bigg)e^{-\frac{1}{2}D\vert s\vert}\Vert\psi^{(n)}_{s}\Vert_{s}\,ds\nonumber\\&\leq (2\delta C)e^{-\frac{1}{2}Dt}\int_{t-\delta}^{\delta}e^{-\frac{1}{2}D\vert s\vert}\Vert\psi^{(n)}_{s}\Vert_{s}\,ds\label{eq:EstNegt}
	\end{align}
	for all $t\leq 0$ and $n\in\mathbb{N}$. Now, note that the previous two estimates~\eqref{eq:EstPost} and~\eqref{eq:EstNegt} for $t\geq 0$ and $t\leq 0$ can be combined and written in the compact form
	\begin{align}
		\forall t\in\R,\,n\in\N:\quad\Vert\psi^{(n+1)}_{t}\Vert_{t}\leq (2\delta C)e^{\frac{1}{2}D\vert t\vert}\int_{-\delta}^{\vert t\vert+\delta}e^{-\frac{1}{2}D\vert s\vert}\Vert\psi^{(n)}_{\mathrm{sgn}(t)s}\Vert_{\mathrm{sgn}(t)s}\, ds \label{ites}
	\end{align}
	
	After deriving the relevant estimates for the $\H_{t}$-norm of $\psi^{(n+1)}$ in terms of the $\H_{t}$-norm of $\psi^{(n)}$, we derive suitable estimates for $\Vert\psi^{(n)}_{t}\Vert_{t}$. As a first step, by Corollary~\ref{Cor:EnEstLoc}, we can find a constant $M>0$ such that 
	\begin{align*}
		\Vert\psi^{(0)}_{t}\Vert_{t}\leq e^{\frac{1}{2}D\vert t\vert}M
	\end{align*}
	for all $t\in\R$. Now, for given $n\in\N$, we consider the induction hypothesis 
	\begin{align*}
		\Vert\psi^{(n)}_{t}\Vert_{t}\leq M\kappa(n)e^{\frac{1}{2}D\vert t\vert}(\vert t\vert+2n\delta)^{n},\qquad\forall t\in\R\, .
	\end{align*}	
	 Clearly, this inequality holds true for $n=0$ with $\kappa(0)=1$. By induction, we assume that it is true for $n$. Using the estimate~\eqref{ites} and the induction hypothesis, we obtain
	\begin{align*}
		\Vert\psi^{(n+1)}_{t}\Vert_{t}&\leq (2\delta C)e^{\frac{1}{2}D\vert t\vert}\int_{-\delta}^{\vert t\vert+\delta}e^{-\frac{1}{2}D\vert s\vert}\Vert\psi^{(n)}_{\mathrm{sgn}(t)s}\Vert_{\mathrm{sgn}(t)s}\, ds\\&\leq (2\delta C)M\kappa(n)e^{\frac{1}{2}D\vert t\vert}\int_{-\delta}^{\vert t\vert+\delta}(\vert s\vert+2n\delta)^{n}\, ds\\&= (2\delta C)M\kappa(n)e^{\frac{1}{2}D\vert t\vert}\frac{1}{n+1}\bigg\{(\vert t\vert+2\delta n+\delta)^{n+1}-2(2n\delta)^{n+1}+(2\delta n+\delta)^{n+1}\bigg\}\\&\leq (4\delta C)M\kappa(n)e^{\frac{1}{2}D\vert t\vert}\frac{(\vert t\vert+2(n+1)\delta)^{n+1}}{n+1}\, .
	\end{align*}
	This allows us to determine the coefficients $\kappa(n)$ in $\Vert\psi^{(n)}_{t}\Vert_{t}\leq\kappa(n)e^{\frac{1}{2}D\vert t\vert}(\vert t\vert+n\delta)^{n}M$ and we obtain $\kappa(n)=\frac{1}{n!}(4C\delta)^{n}$. To sum up, we have shown that
	\begin{align}\label{eq:EstimatesSmallTR}
		\forall t\in\R\,,n\in\N:\quad \Vert\psi^{(n)}_{t}\Vert_{t} \leq \frac{M}{n!}(4C\delta)^{n}e^{\frac{1}{2}D\vert t\vert}(\vert t\vert+2n\delta)^{n}\, .
	\end{align}

	To show that the perturbative ansatz~\eqref{ansatzpertur2} is absolutely convergent in the $\H_{t}$-topology, we have to study the bounds on the $\H_{t}$-norms of $\psi^{(n)}$ for for large $n$. In this case, using Stirling's formula for the factorial, one obtains
	\begin{align*}
		\Vert\psi^{(n)}_{t}\Vert_{t} &\leq \frac{M}{n!}(4C\delta)^{n}e^{\frac{1}{2}D\vert t\vert}(\vert t\vert +2n\delta)^{n}= \frac{M}{n!}(8C\delta^2)^{n}n^{n}e^{\frac{1}{2}D\vert t\vert}\left(1+\frac{\vert t\vert}{2n\delta}\right)^n \\&\sim\frac{Me^{\frac{1}{2\delta}\vert t\vert(D\delta+1)}}{\sqrt{2\pi n}}(8e\delta^{2}C)^{n}\left(1 + \mathcal{O}\left(\frac{1}{n}\right)\right) \quad \text{as} \quad n \to \infty.
	\end{align*}
	In particular, using this inequality and the asymptotic expansion as written in the last line, it follows that 
	\begin{align*}
		\lim_{n\to\infty}\frac{\Vert\psi^{(n+1)}_{t}\Vert_{t}}{\Vert\psi^{(n)}_{t}\Vert_{t}}=8e\delta^{2}C
	\end{align*}
	for all $t\in\R$, which shows that the series $\sum_{n=0}^{\infty}\Vert\psi^{(n)}_{t}\Vert_{t}$ is convergent whenever $8e\delta^{2}C<0$, by the ratio test of convergence of series. 
	
	Now, we have constructed a map $\psi:t\mapsto\psi_{t}$, where $\psi_{t}$ is the limit of the series~\ref{ansatzpertur2} in the $\H_{t}$-topology. Now, it is clear that $t\mapsto \psi_{t}$ is measurable, since it is the pointwise limit of the smooth (and in particular strongly measurable) functions $t\mapsto \sum_{k=0}^{n}\psi^{(k)}_{t}$ for $n\to\infty$, see e.g.~\cite[Thm.~1.14]{AmannEscherIII}. The fact that $\psi$ is $L^{2}_{\mathrm{loc}}$ in time can be seen as follows: Consider the compact time interval $[0,T]$ for $T>0$. Using~\eqref{eq:EstimatesSmallTR}, we obtain
	\begin{align*}
		\Vert\psi_{t}^{(n)}\Vert_{\scrM_{T}}^2&=\int_{0}^{T}\Vert\psi^{(n)}_{t}\Vert^2_{t}\,dt \leq \frac{M^2}{(n!)^2}(4C\delta)^{2n}\int_{0}^{T}e^{D\vert t\vert}(\vert t\vert+2n\delta)^{2n}\,dt\\&\leq \frac{TM^2e^{DT}}{(n!)^2}(4C\delta)^{2n}(T+2n\delta)^{2n}\, .
	\end{align*}
	Using similar arguments as before, we conclude that the series of $\Vert\psi_{t}^{(n)}\Vert_{\scrM_{T}}^2$ is convergent. Hence, by the dominant convergence theorem applied to series (see e.g.~\cite[Thm.~3.16]{AmannEscherIII}), we conclude that the series~\eqref{ansatzpertur2} is also absolutely convergent in the $\H_{\overline{\scrM}_{T}}$-topology and 
	\begin{align*}
		\Vert\psi\Vert_{\scrM_{T}}^2=\int_{0}^{T}\Vert\psi_{t}\Vert_{t}^2\,dt=\int_{0}^{T}\sum_{n=0}^{\infty}\Vert\psi_{t}^{(n)}\Vert_{t}^2\,dt=\sum_{n=0}^{\infty}\int_{0}^{T}\Vert\psi_{t}^{(n)}\Vert_{t}^2\,dt=\sum_{n=0}^{\infty}\Vert\psi^{(n)}\Vert_{\scrM_{T}}^2<\infty\, .
	\end{align*}
	This holds true for any compact time strip, which concludes the proof.
\end{proof}

With the previous proposition, we are able to prove the main theorem of this section.

\begin{proof}[Proof of Theorem~\ref{Thm:Small}]
Consider the sequence $(\psi^{(n)})_{n\in\N}$ as constructed in Proposition~\ref{Prop:DysonSmall}. As shown in this proposition, the corresponding series converges absolutely in the $\H_{[-\delta,T+\delta]\times\scrN}$-topology to some element $\psi\in\H_{[-\delta,T+\delta]\times\scrN}$. To show that $\psi$ is a strong solution to the Cauchy problem~\eqref{eq:CauchySmall} we follows from similar steps as in Proposition~\ref{Prp:RetReg}, with the involved norms and time strips adapted properly.
\end{proof}

The following example demonstrates that Theorem~\ref{Thm:Small} cannot be expected to hold in a significantly broader sense for nonlocal potentials that do not satisfy the condition $C<(8e\delta^{2})^{-1}$ on the time kernel, even in the case $\scrS+\scrS^{\dagger}=0$.

\begin{Example}\label{Counterexample} (A simple Counterexample)\newline
	We consider $(1+1)$-dimensional Minkowski spacetime $\scrM=\R_{t}\times\R_{x}$ and the trivial line bundle over $\C$ equipped with the standard complex inner product. As a symmetric hyperbolic system we consider the operator
	\begin{align*}
		\scrS\colon C^{\infty}(\scrM)\to C^{\infty}(\scrM)\,,\qquad \scrS:=\partial_{t}\, .
	\end{align*}
	In particular, $\H_{t}\cong L^{2}(\R_{x},dx)$ for all $t\in \R_{t}$ and $\H_{\scrM_{T}}\cong L^{2}([0,T]\times\R_{x},dt\,dx)$ for $T>0$. Furthermore, we choose a real-valued function $f\in C^{\infty}_{\mathrm{c}}(\scrM)$ such that $\Vert f\Vert_{L^{2}(\R_{t}\times\R_{x})}=1$. As a nonlocal potential, we choose the integral operator $\B:C^{\infty}_{\mathrm{c}}(\scrM)\to C^{\infty}(\scrM)$ defined by
	\begin{align*}
		(\B\psi)(t,x):=-f(t,x)\int_{\R}\,d\tau\,\int_{\R}\,dy\,\dot{f}(\tau,y)\psi(\tau,y)\,,\qquad\forall \psi\in C^{\infty}_{\mathrm{c}}(\scrM)\, ,
	\end{align*}
	or in bra/ket notation, $\B=|f\ket\bra\scrS^\dagger f|=-|f\ket\bra \dot{f}|$, where $<\cdot|\cdot>$ denotes the inner product induced by the bundle metric, as usual, which in this case coincides with the $L^{2}$-inner product on $\scrM=\R_{t}\times\R_{x}$. The Schwartz kernel $k_{\B}\in C^{\infty}(\scrM\times\scrM)$ of $\B$ is regular and given by $k_{\B}(t,x;\tau,y)=-f(t,x)\dot{f}(\tau,y)\in C^{\infty}(\scrM\times\scrM)$ and the corresponding time kernel $\B_{t,\tau}:C^{\infty}_{\mathrm{c}}(\R_{x})\to C^{\infty}(\R_{x})$ can be written as
	\begin{align*}
		\B_{t,\tau}\psi_{\tau}(x)=-f(t,x)\int_{\R}\,dy\,\dot{f}(\tau,y)\psi(\tau,y)\,,\qquad\forall \psi\in C^{\infty}_{\mathrm{c}}(\scrM)\, .
	\end{align*}
	Now, in order for $\B$ to satisfy the smallness condition, we have to choose $f$ such that $\mathrm{supp}(f)\subset (0,\delta/2)\times K$, where $K\subset\R_{x}$ is some compact subset. In this way, it holds that $\B_{t,\tau}=0$ for $\vert t-\tau\vert >\delta$. Next, let us show that $\B_{t,\tau}$ is uniformly bounded in time, i.e.~that there exists a constant $C>0$ such that $\Vert\B_{t,\tau}\psi_{\tau}\Vert_{t}\leq C\Vert\psi_{\tau}\Vert_{\tau}$ for all $\psi\in C^{\infty}_{\mathrm{c}}(\scrM)$ and all $t,\tau\in\R$. For this, let $\psi\in C^{\infty}_{\mathrm{c}}(\scrM)$ be arbitrary. Then
	\begin{align*}
		\Vert\B_{t,\tau}\psi_{\tau}\Vert_{\tau}^{2}&=\int_{\R}\,dx\,\vert \B_{t,\tau}\psi_{\tau}(x)\vert^{2}=\int_{\R}\,dx\,\bigg\vert f(t,x)\int_{\R}\,dy\,\dot{f}(\tau,y)\psi(\tau,y)\bigg\vert^{2}\stackrel{\text{Hölder ineq.}}{\leq}\\&\leq \bigg(\int_{\R}\,dx\,\vert f(t,x)\vert^{2}\bigg)\bigg(\int_{\R}\,dy\,\vert \dot{f}(\tau,y)\vert^{2}\bigg)\Vert \psi_{\tau}\Vert_{\tau}^{2}\leq\\&\leq \underbrace{\max_{t,t^{\prime}\in (0,\delta/2)}\bigg(\int_{\R}\,dx\,\vert f(t,x)\vert^{2}\bigg)\bigg(\int_{\R}\,dy\,\vert \dot{f}(t^{\prime},y)\vert^{2}\bigg)}_{=:C^{2}}\Vert \psi_{\tau}\Vert_{\tau}^{2}\, .
	\end{align*} 
	Next, we will show that there exists a $\psi\in C^{\infty}_{\mathrm{c}}(\scrM)$ and a choice of $t,\tau\in\R$ such that $\Vert\B_{t,\tau}\psi_{\tau}\Vert_{t}\geq\frac{2}{\delta^{2}}\Vert\psi_{\tau}\Vert_{\tau}$, which shows that the constant $C$ cannot be chosen smaller than $\frac{2}{\delta^{2}}$ and hence that assumption (ii) in Theorem~\ref{Thm:Small} is not fulfilled for the nonlocal potential $\B$. For this, we choose 
$\psi:=\dot{f}$ and compute
	\begin{align*}
		\Vert\B_{t,t}\dot{f}_{t}\Vert_{t}^{2}&=\int_{\R}\,dx\,\bigg\vert f(t,x)\int_{\R}\,dy\,\vert\dot{f}(t,y)\vert^{2}\bigg\vert^{2}=\\&=\bigg(\int_{\R}\,dx\,\vert f(t,x)\vert^{2}\bigg)\bigg(\int_{\R}\,dy\,\vert \dot{f}(t,y)\vert^{2}\bigg)\Vert\dot{f}_{t}\Vert_{t}^{2}\stackrel{\text{Hölder ineq.}}{\geq}\\&\geq \bigg(\int_{\R}\,dx\,\vert f(t,x)\dot{f}(t,x)\vert\bigg)^{2}\Vert\dot{f}_{t}\Vert_{t}^{2}\geq\bigg(\frac{1}{2}\partial_{t}\int_{\R}\,dx\,\vert f(t,x)\vert^{2}\bigg)^{2}\Vert\dot{f}_{t}\Vert_{t}^{2}\, ,
	\end{align*}
	which holds true for any value of $t\in\R$. Now, let us set $g(t):=\int_{\R}dx\,\vert f(t,x)\vert^{2}$. By definition, we note that $g(t)\geq 0$ for all $t$ and $\Vert g\Vert_{L^{1}(0,\delta/2)}=1$. Let $t_{\mathrm{max}}\in (0,\delta/2)$ be such that $g(t)\leq g(t_{\mathrm{max}})$ for all $t\in (0,\delta/2)$. The normalisation $\int_{\R}dt\,g(t)=\int_{0}^{\delta/2}\,dt\,g(t)=1$ implies that $g(t_{\mathrm{max}})\geq \frac{2}{\delta}$. By the mean value theorem we conclude that there exists a $t_{0}\in (0,t_{\mathrm{max}})$ such that
	\begin{align*}
		\partial_{t}g(t_{0})=\frac{g(t_{\mathrm{max}})-g(0)}{t_{\mathrm{max}}-0}\geq \frac{4}{\delta^{2}}\,, 
	\end{align*} 
	since $t_{\mathrm{max}}\leq \frac{\delta}{2}$. Hence, we have shown that there exists a $t_{0}\in (0,\delta/2)$ such that 
	\begin{align*}\Vert\B_{t_{0},t_{0}}\dot{f}_{t_{0}}\Vert_{t_{0}}\geq \frac{2}{\delta^{2}}\Vert \dot{f}_{t_{0}}\Vert_{t_{0}}\, .\end{align*}
	In particular, our nonlocal potential violates the condition $C<(8e\delta^{2})^{-1}$ of Theorem~\ref{Thm:Small}. We will now show that the nonlocal equation $(\scrS-\B)\psi=\phi$ does not always admit a solution. Suppose there exists a strong solution $\psi\in L^{2}(\R_{t}\times\R_{x})$ to the nonlocal equation $(\scrS-\B)\psi=f$ with source $f$. Then, for any test function $\varphi\in C^{\infty}_{\mathrm{c}}(\scrM)$ we must have
	\begin{align}\label{eq:Cexample}
		\bra \psi|(\scrS-\B)^{\dagger}\varphi\ket=\bra f|\varphi\ket\, .
	\end{align}
	Now, note that the left-hand side can be written as
	\begin{align*}
		\bra \psi|(\scrS-\B)^{\dagger}\varphi\ket=\bra\psi|\scrS^{\dagger}f\ket-\bra\psi|\scrS^{\dagger}f\ket\bra f|\varphi\ket=\bra\psi|\scrS^{\dagger}f\ket(1-\bra f|\varphi\ket)\, .
	\end{align*}
	But now if we choose $\varphi=f$ as a test function, the left-hand side of~\eqref{eq:Cexample} gives $0$ while the right-hand side 1, which shows that such a solution $\psi$ cannot exist.
\end{Example}
\section{Application I: Maxwell's Equations in Linear Dispersive Media}\label{sec:Maxwell}
	As an application of a symmetric hyperbolic system with a retarded nonlocal potential, we consider Maxwell's equations in a linear dispersive medium. For simplicity, we assume that the Lorentzian manifold is \emph{ultrastatic}, i.e.
\begin{align*}
\scrM = \R \times \scrN, \qquad g = -dt \otimes dt + h\, ,
\end{align*}
where $(\scrN,h)$ is a smooth Riemannian $3$-manifold. It is well known that $(\scrM,g)$ is globally hyperbolic if and only if $(\scrN,h)$ is complete, see e.g.~\cite[Thm.~3.1]{SanchezUltra} and \cite[Prop.~5.2]{KayUltraStatic}. We define the divergence and curl operator on the manifold $(\scrN,h)$ in local coordinates $(x^{i})_{i=1,2,3}$ of $\scrN$ by
\begin{align*}
	\mathrm{div}\,X:=\nabla_{i}X^{i},\qquad (\mathrm{curl}\,\vec{X})_{i}:=[(*dX^{\flat})^{\sharp}]_{i}:=\epsilon_{ijk}\nabla^{j}\vec{X}^{k}
\end{align*}
for all $X\in C^{\infty}(\scrN,T\scrN)$, where $\nabla$ and $\epsilon$ denote the Levi-Civita connection and pseudotensor with sign convention $\epsilon_{123}=\sqrt{\mathrm{det}(h)}$ on $\scrN$, respectively. 

Time-dependent vector fields on $\scrN$ can be identified with sections of the pull-back bundle $\scrF:=\pi^{\ast}(T\scrN)\to\scrM$, where $\pi\colon\scrM\to\scrN$ denotes the natural projection. We equip the bundle $\scrF$ with the bundle metric induced by the Riemannian metric $h$, i.e.
\begin{align*}
	\Sl \psi|\varphi\Sr_{\scrF}:=h_{ij}\psi^{i}\varphi^{j},\qquad \forall \psi,\varphi\in C^{\infty}(\scrM,\scrF)\, ,
\end{align*}
where we wrote $\psi=\psi^{i}\partial_{i}$ and $\varphi=\varphi^{i}\partial_{i}$ in a local coordinate chart $(x^{i})_{i=1,2,3}$ of $\scrN$.

While Maxwell's equations on manifolds are often dealt with in the manifestly covariant formulation using differential forms, they can also be understood as a symmetric hyperbolic systems in terms of the electric and magnetic fields, once a time function has been fixed. On ultrastatic manifolds, this has for example been studied in \cite[App.~D]{Capoferri} and \cite{BaerCurl}, see also the lecture notes \cite[Ex.~3.7.3]{BaerLectureGeo}. On more generally manifolds (possibly with timelike boundary), Maxwell's equations as a symmetric hyperbolic system in terms of the electric and magnetic part of the Faraday tensor have been discussed in \cite{DragoGinouxMurroMaxwell}. We denote the \emph{electric} and \emph{magnetic fields} on $\scrM$ by
\begin{align*}
	E,B \in C^{\infty}(\scrM,\scrF)\,.
\end{align*} 
Additionally, we consider external sources: the \emph{electric charge density} $\rho^{(\mathrm{ext})} \in C^{\infty}(\scrM)$ and the external \emph{current density} $j^{(\mathrm{ext})} \in C^{\infty}(\scrM,\scrF)$, related by the \emph{continuity equation} 
\begin{align}\label{eq:ContEqMax}
	\partial_{t}\rho^{(\mathrm{ext})}+\mathrm{div}\,j^{(\mathrm{ext})}=0\, .
\end{align}

Now, in the presence of matter, the electric and magnetic fields $E$ and $B$ are not entirely independent of the charge and current distributions; rather, they interact with bound charges and currents within the medium. These interactions give rise to \emph{induced fields}, making the total sources inseparable from the fields themselves.  To handle this coupling in a tractable way, the standard approach is to decompose the total charge and current densities into external and induced parts, i.e.~$\rho = \rho^{(\mathrm{ext})} + \rho^{(\mathrm{ind})}$ and $j = j^{(\mathrm{ext})} + j^{(\mathrm{ind})}$, and to introduce the auxiliary fields $\vec{D}$ (\emph{electric displacement}) and  $\vec{H}$ (\emph{magnetizing field}), in Gaußian units and with speed of light $c=1$, by
\begin{align}\label{eq:Aux}
D := E + 4\pi P\,, \qquad
H := B - 4\pi M\, ,
\end{align}
where $P, M \in C^{\infty}(\scrM,\scrF)$ are the \emph{polarisation} and \emph{magnetisation} vector fields, defined in terms of the induced charge and current by the system\footnote{Note that $P$ and $M$ are only specified \emph{up to gauge}, that is, up to $(P,M)\mapsto (P+\mathrm{curl}\,G,M-\partial_{t}G)$ for any vector field $G\in C^{\infty}(\scrM,\scrF)$, as one can easily see from the defining equations.}
\begin{align*}
\begin{cases}
\rho^{(\mathrm{ind})} &= -\mathrm{div}\, P \\
j^{(\mathrm{ind})} &= \partial_t P + \mathrm{curl}\, M
\end{cases}\, .
\end{align*}
With these definitions, the \emph{macroscopic Maxwell's equations} can be written in the form
\begin{align}\label{eq:Max}
	\begin{cases}
			\partial_{t}D-\mathrm{curl}\,H&=-4\pi j^{(\mathrm{ext})}\\
			\partial_{t}B+\mathrm{curl}\,E&=0
		\end{cases},\qquad \begin{cases}
			\mathrm{div}\,D=4\pi\rho^{(\mathrm{ext})}\\
			\mathrm{div}\,B=0
		\end{cases}\, .
	\end{align}
The first pair of equations, the \emph{Ampère–Faraday–Maxwell equations}, represent the dynamical laws of the system, while the second pair, the \emph{Gauß equations}, act as constraint conditions. As written, Maxwell's equations~\eqref{eq:Max}, viewed as a system for the fields $E, D, B, H$ in terms of the external sources $\rho^{(\mathrm{ext})}$ and $j^{(\mathrm{ext})}$, are of course underdetermined. To resolve this issue, one must supplement the system with a suitable \emph{matter model} in which $P$ and $M$ are given functions of $E$ and $B$, hence determining the auxiliary fields $D$ and $H$ in terms of $E$ and $B$ via Equation~\eqref{eq:Aux}. The precise form of these \emph{constitutive equations} is determined by the physical properties of the medium and when chosen appropriately, ensure a well-posed formulation of Maxwell's equations.

The class of matter models we consider in this section are \emph{linear dispersive medium}, see e.g.~\cite{LandLifMax,Cessenat}. More precisely, we consider a \emph{linear response function} $\chi\in C^{\infty}(\R)$ with the property that $\chi(0)=0$. Then, we consider the constitutive equations	
\begin{align}\label{eq:Response}
	P(t,\vec{x})=\frac{1}{4\pi}\int_{-t_{0}}^{t}\chi(t-\tau)E(\tau,\vec{x})\,d\tau\,,\quad M(t,\vec{x})=0\,,
\end{align}
where $t_{0}\geq 0$ is fixed and where the right-hand side has to be understood in the distributional sense, i.e.~by integrating against a test section of $\scrF$ that is contracted with above integrand via the bundle metric. An explicit example of such a model is the well-known \emph{Drude-Lorentz model} that describes the response of electrons in a material to external electromagnetic fields, for which the response function $\chi$ is given by
\begin{align*}
	\chi(t)\propto e^{-c_{1}t}\mathrm{sin}(c_{2}t)
\end{align*}
for some constants $c_{1},c_{2}>0$, see e.g.~\cite{Marletta,Engstrom,Cassier} for some recent discussions of this model in various of its mathematical aspects. The dynamical part of the Maxwell equations~\eqref{eq:Max} together with~the constitutive equations~\eqref{eq:Response} can be written as
\begin{align*}
	\bigg[\partial_{t}+\begin{pmatrix}
		0 & -\mathrm{curl}\\
		\mathrm{curl} &0
	\end{pmatrix}\bigg]
	\begin{pmatrix}
		E(t,\cdot)\\B(t,\cdot)
	\end{pmatrix}+\int_{-t_{0}}^{t}
	\begin{pmatrix}
		\dot{\chi}(t-\tau)E(\tau,\cdot)\\  0
	\end{pmatrix}
	\,d\tau=
	\begin{pmatrix}
	-4\pi j^{(\mathrm{ext})}(t,\cdot)\\ 0
	\end{pmatrix}\, ,
\end{align*}
where we used that $\chi(0)=0$, which implies that there is no boundary term when taking the $t$-derivative of the nonlocal part in the definition of $P$. The constraint equations, i.e.~the second pair of equations in~\eqref{eq:Max}, are also nonlocal and can be written as
\begin{align}\label{eq:ConstraintsMax}
	\mathrm{div}\,E(t,\cdot)+\int_{-t_{0}}^{t}\chi(t-\tau)\,\mathrm{div}\,E(\tau,\cdot)\,d\tau=4\pi\rho^{(\mathrm{ext})}(t,\cdot)\,,\qquad \mathrm{div}\,B=0\, .
\end{align} 

After this preliminary discussion, we show how the above equations fits in our framework developed in Section~\ref{Subsection:Ret}.

\begin{Lemma}\label{Prop:Maxwell} Let $(\scrM=\R\times\scrN,g=-dt\otimes dt+h)$ be a globally-hyperbolic ultrastatic Lorentzian manifold. Set $\scrE:=\scrF\oplus\scrF$ equipped with the obvious bundle metric $\Sl\cdot|\cdot\Sr_{\scrE}=\Sl\cdot|\cdot\Sr_{\scrF}^{\oplus 2}$. Then,
\begin{align*}
	\scrS_{\mathrm{el}}:=\partial_{t}+\begin{pmatrix}
		0 & -\mathrm{curl}\\
		\mathrm{curl} &0
	\end{pmatrix}\colon C^{\infty}(\scrM,\scrE)\to C^{\infty}(\scrM,\scrE)
\end{align*}
is a symmetric hyperbolic system in the sense of Definition~\ref{Def:SHS}. Furthermore, let $\chi\in C^{\infty}(\R)$ be such that $\chi(0)=0$. Then, the piecewise continuous kernel $k_{\chi}(x,y)\colon \scrE_{y}\to\scrE_{x}$ defined by
\begin{align*}
	k_{\chi}((t,\vec{x}),(\tau,\vec{y}))=\theta(\tau+t_{0})\theta(t-\tau)\begin{pmatrix}\dot{\chi}(t-\tau)\delta(\vec{x}-\vec{y}) & 0 \\ 0& 0\end{pmatrix}
\end{align*}
is the kernel of a retarded and past compact nonlocal potential $\B_{\chi}\colon C^{\infty}_{\mathrm{c}}(\scrM,\scrE)\to C^{\infty}(\scrM,\scrE)$.
\end{Lemma}

\begin{proof}
	Following Example~\ref{Exam:RetUltraStat}, it is clear that $\B_{\chi}$ is well-defined nonlocal retarded potential with time kernel in the sense of Definition~\ref{Def:NonLocPot} and Definition~\ref{Def:Potentials}.
	
	To show that $\scrS_{\mathrm{el}}$ is a symmetric hyperbolic system, we first note that the principal symbol $\sigma_{\mathrm{curl}}\colon T^{\ast}\scrN\to\mathrm{End}(\scrF)$ of the operator $\mathrm{curl}$ is given by
	\begin{align*}
		(\sigma_{\mathrm{curl}}(\vec{\xi})\psi)_{i}=\epsilon_{ijk}\xi^{j}\psi^{k}\,,\qquad\forall\vec{\xi}=\xi_{i}dx^{i}\in T^{\ast}\scrN\, ,\psi\in C^{\infty}(\scrM,\scrF)\, ,
	\end{align*}
	where $(x^{i})_{i=1,2,3}$ are local coordinates on $\scrN$. Now, it is easy to see that the endomorphism $\sigma_{\mathrm{curl}}(\vec{\xi})$ is antisymmetric with respect to $\Sl\cdot|\cdot\Sr_{\scrF}$, since
	\begin{align*}
		\Sl	\sigma_{\mathrm{curl}}(\vec{\xi})\psi|\varphi\Sr_{\scrF}=\epsilon_{ijk}\xi^{j}\psi^{k}\varphi^{i}=-\epsilon_{kji}\xi^{j}\psi^{k}\varphi^{i}=-\Sl\psi|\sigma_{\mathrm{curl}}(\vec{\xi})\varphi\Sr_{\scrF}
	\end{align*}		
	for all $\vec{\xi}\in T^{\ast}\scrN$ and $\psi,\varphi\in C^{\infty}(\scrM,\scrF)$. In particular, this shows that 
	\begin{align*}
		\sigma_{\scrS_{\mathrm{el}}}(\xi)=\begin{pmatrix}
			\xi_{0}\mathrm{id}_{\scrF} & -\sigma_{\mathrm{curl}}(\vec{\xi})\\ \sigma_{\mathrm{curl}}(\vec{\xi}) & \xi_{0}\mathrm{id}_{\scrF}
		\end{pmatrix}\in\mathrm{End}(\scrE)\,,\qquad\forall \xi=(\xi_{0},\vec{\xi})\in T^{\ast}\scrM
	\end{align*}
	is symmetric with respect to $\Sl\cdot|\cdot\Sr_{\scrE}$. For hyperbolicity, we take a general future-directed time-like vector of the form $\tau=dt+\vec{\alpha}$, where $\vec{\alpha}=\alpha_{i}dx^{i}\in C^{\infty}(\scrM,\scrF)$ is such that $\Vert\vec{\alpha}\Vert_{\scrF}^{2}=h^{ij}\alpha_{i}\alpha_{j}<1$. Then, for every $\omega=(\psi,\varphi)$ with $\psi,\varphi\in C^{\infty}(\scrM,\scrF)$, the Cauchy-Schwartz inequality implies
	\begin{align*}
		\Sl\omega| \sigma_{\scrS_{\mathrm{el}}}(\tau)\omega\Sr_{\scrE}&=\Vert\omega\Vert_{\scrE}^{2}-2\Sl\psi|\sigma_{\mathrm{curl}}(\vec{\alpha})\varphi\Sr_{\scrF}\\&\geq \Vert\omega\Vert_{\scrE}^{2}-2\Vert\psi\Vert_{\scrF}\Vert\varphi\Vert_{\scrF}\Vert\sigma_{\mathrm{curl}}(\vec{\alpha})\Vert_{\mathrm{op}}\\&\geq\Vert\omega\Vert_{\scrE}^{2}(1-\Vert\sigma_{\mathrm{curl}}(\vec{\alpha})\Vert_{\mathrm{op}})\, ,
	\end{align*}
	where we used that $-2\Vert\psi\Vert_{\scrF}\Vert\varphi\Vert_{\scrF}\geq -(\Vert\psi\Vert_{\scrF}^{2}+\Vert\varphi\Vert_{\scrF}^{2})=-\Vert\omega\Vert_{\scrE}^{2}$ in the last step. Hence, the quadratic form $\Sl\cdot| \sigma_{\scrS_{\mathrm{el}}}(\tau)\cdot\Sr_{\scrE}$ is positive-definite provided that the operator norm of $\sigma_{\mathrm{curl}}(\vec{\alpha})$ satisfies $\Vert\sigma_{\mathrm{curl}}(\vec{\alpha})\Vert_{\mathrm{op}}<1$. Now, we note that
	\begin{align*}
		\Vert \sigma_{\mathrm{curl}}(\vec{\alpha})\psi\Vert_{\scrF}^{2}=\epsilon_{ijk}\epsilon^{ilm}\alpha^{j}\alpha_{l}\psi^{k}\psi_{m}=\Vert\vec{\alpha}\Vert_{\scrF}^{2}\Vert\psi\Vert_{\scrF}^{2}-(\langle\vec{\alpha}|\psi\rangle_{\scrF})^{2}\leq \Vert\vec{\alpha}\Vert_{\scrF}^{2}\Vert\psi\Vert_{\scrF}^{2}\, ,
	\end{align*}		
	where we used that $\epsilon_{ijk}\epsilon^{ilm}=\delta_{j}^{l}\delta_{k}^{m}-\delta_{j}^{m}\delta_{k}^{l}$. In particular, it follows that
	\begin{align*}
		\Vert\sigma_{\mathrm{curl}}(\vec{\alpha})\Vert_{\mathrm{op}}^{2}&=\sup_{\Vert\psi\Vert_{\scrF}=1}\Vert \sigma_{\mathrm{curl}}(\vec{\alpha})\psi\Vert_{\scrF}^{2}\leq\Vert\vec{\alpha}\Vert_{\scrF}^{2}<1\, ,
	\end{align*}
by assumption on $\vec{\alpha}$, which concludes the proof.	
\end{proof}

Following this preliminary discussion, we now turn to the Cauchy problem for Maxwell's equations in linear dispersive media, as an example of a symmetric hyperbolic system with a retarded and past compact nonlocal potential.

\begin{Prp}\label{Thm:Maxwell}
	Let $(\scrM,g)$ be a four-dimensional ultrastatic globally hyperbolic manifold and consider the Maxwell operator $\scrS_{\mathrm{el}}$ as defined in~Lemma~\ref{Prop:Maxwell}. Furthermore, let $\chi\in C^{j}(\R)$ with $j\geq 2$,  $\chi(0)=0$, and consider the corresponding nonlocal potential $\B_{\chi}$ with $t_{0}=0$ as defined in~Lemma~\ref{Prop:Maxwell}. Then, for every choice of initial data $\mathfrak{f}_{1},\mathfrak{f}_{2}\in C^{\infty}_{\mathrm{c}}(\scrN_{0},\scrF\vert_{\scrN_{0}})$ and external sources supported on the time strip $\overline{\scrM}_{T}$, i.e.
	\begin{align*}
		j^{(\mathrm{ext})}\in C^{\infty}_{\mathrm{c}}(\overline{\scrM}_{T},\scrF),\,\rho^{(\mathrm{ext})}\in C^{\infty}_{\mathrm{c}}(\overline{\scrM}_{T})\,\qquad \partial_{t}\rho^{(\mathrm{ext})}+\mathrm{div}\,j^{(\mathrm{ext})}=0
	\end{align*}
	there exists a unique classical solution $\psi:=(E,B)\in C^{j-2}(\overline{\scrM}_{T},\scrE)$ on the time strip $\overline{\scrM}_{T}$ to the nonlocal Cauchy problem
	\begin{align*}
		\begin{cases}
			(\scrS_{\mathrm{el}}-\B_{\chi})\psi&=\phi\\
			\psi\vert_{\scrN_{0}}&=\mathfrak{f}
		\end{cases}\, ,\qquad\text{where}\qquad \psi:=\begin{pmatrix}
			E\\ B
		\end{pmatrix}\,,\quad\phi:=\begin{pmatrix}
			-4\pi j^{(\mathrm{ext})}\\ 0
		\end{pmatrix}\,,\qquad\mathfrak{f}:=\begin{pmatrix}
		\mathfrak{f}_{1}\\\mathfrak{f}_{2}
		\end{pmatrix}
	\end{align*}
	that propagates at most with the speed of light, i.e.
	\begin{align*}
		\mathrm{supp}(\psi)\cap J^{+}(\scrN_{0})\subset J^{+}(\mathrm{supp}(\phi)\cap\mathrm{supp}(\mathfrak{f}))\, .
\end{align*}		
	If, furthermore, the initial data are chosen such that $\mathrm{div}\,\mathfrak{f}_{1}=\rho^{(\mathrm{ext})}\vert_{t=0}$ and $\mathrm{div}\,\mathfrak{f}_{2}=0$, then also the constraint equations~\eqref{eq:ConstraintsMax} are satisfied.
\end{Prp}

\begin{proof}
	The existence of a strong solution is a direct application of the first part of Theorem~\ref{Thm:Ret}, where we note that the condition $\Vert\V_{t,\tau}\psi_{\tau}\Vert\leq C\Vert\psi_{\tau}\Vert_{\tau}$ on $[0,T]$ is automatically satisfies given that $\dot{\chi}$ is continuous. 
	
	For regularity, observe the following: If we look at the the operator $\mathsf{B}_{1}$ of the extended system acting on $(\psi,\nabla^{\scrE}\psi)$ as defined in~\eqref{eq:ExtB1}, the only relevant terms to consider is the commutator $[\partial_{t},\B]$. Now, for any $\varphi\in C^{\infty}_{\mathrm{pc}}(\scrM,\scrF)$, it holds that
	\begin{align*}
		&\partial_{t}\int_{0}^{t}\dot{\chi}(t-\tau)\varphi(\tau,\cdot)\,d\tau-\int_{0}^{t}\dot{\chi}(t-\tau)\partial_{\tau}\varphi(\tau,\cdot)=\\&=\ddot{\chi}(0)\varphi(t,\cdot)+\int_{0}^{t}\ddot{\chi}(t-\tau)\varphi(\tau,\cdot)\,d\tau-\bigg(\ddot{\chi}(0)\varphi(t,\cdot)-\ddot{\chi}(t)\varphi(0,\cdot)-\int_{0}^{t}\underbrace{(\partial_{\tau}\dot{\chi}(t-\tau))}_{-\ddot{\chi}(t-\tau)}\varphi(\tau,\cdot)\,d\tau\bigg)\\&=\ddot{\chi}(t)\mathfrak{f}
	\end{align*}
	Hence, up to a local operator that can be added to the zero-order terms of $\mathsf{S}_{1}$, the nonlocal potential $\mathsf{B}_{1}$ just acts component-wise as $\B$ on $\psi$. Hence, if $\chi$ is regular enough, the extended systems can be solved in the same way and we obtain a solution in $\psi\in C^{j-2}(\overline{\scrM}_{T},\scrE)$. Uniqueness of finite speed of propagation then follows from the second part of Theorem~\ref{Thm:Ret}.
	
	Now, to show that also the constraint equations are satisfied provided the initial data is chosen such that $\mathrm{div}\,\mathfrak{f}_{1}=\rho^{(\mathrm{ext})}\vert_{t=0}$ and $\mathrm{div}\,\mathfrak{f}_{2}=0$, we take the divergence of $(\scrS_{\mathrm{el}}-\B_{\chi})\psi=\phi$, which implies
	\begin{align*}
		\partial_{t}\bigg(\mathrm{div}\,E(t,\cdot)+\int_{0}^{t}\chi(t-\tau)\,\mathrm{div}\,E(\tau,\cdot)\,d\tau\bigg)=4\pi\partial_{t}\rho^{(\mathrm{ext})}(t,\cdot)\,,\qquad \partial_{t}\mathrm{div}\,B=0\, ,
	\end{align*} 
	where we used the continuity equation $\partial_{t}\rho^{(\mathrm{ext})}+\mathrm{div}\,j^{(\mathrm{ext})}=0$.  In particular, this implies that there are constants $C_{1},C_{2}$ such that 
	\begin{align*}
		\mathrm{div}\,E(t,\cdot)+\int_{0}^{t}\chi(t-\tau)\,\mathrm{div}\,E(\tau,\cdot)\,d\tau-4\pi\rho^{(\mathrm{ext})}(t,\cdot)=C_{1}\,,\qquad \mathrm{div}\,B=C_{2}\, ,
	\end{align*} 
	Now, constraint equations~\eqref{eq:ConstraintsMax} are fulfilled if and only if $C_{1}=C_{2}=0$. Due to the choice $t_{0}=0$, the restriction of the above equations $t=0$ gives the local conditions
	\begin{align*}
		\mathrm{div}\,\mathfrak{f}_{1}-4\pi\rho^{(\mathrm{ext})}\vert_{t=0}=C_{1}\,,\qquad \mathrm{div}\,\mathfrak{f}_{2}=C_{2}\, .
	\end{align*}
	Hence, we conclude that the constraint equations~\eqref{eq:ConstraintsMax} are fulfilled if and only if the initial data is chosen such that $\mathrm{div}\,\mathfrak{f}_{1}=\rho^{(\mathrm{ext})}\vert_{t=0}$ and $\mathrm{div}\,\mathfrak{f}_{2}=0$.
\end{proof}
\section{Application II: Dirac Equation and Potentials with Short Time Range}\label{sec:Dirac}
In the final section, we focus on an important special case of symmetric hyperbolic systems, namely the classical \textit{Dirac operator} on globally hyperbolic manifolds. After introducing the necessary spin geometric background, we illustrate the main results of Section~\ref{Sec:SHSNonLoc} in this example. In addition, we establish a uniqueness result for nonlocal potentials with short time range by endowing the space of solutions with a modified (nonlocal) inner product. We conclude by reviewing recent applications of the nonlocal Dirac equation to causal fermion systems.

\subsection{Dirac Equations and Nonlocal Potentials with Short Time Range}
Let $(\scrM,g)$ be a globally hyperbolic manifold of dimension $k+1$ (with $k\geq 1$) equipped with a \textit{spin structure}, i.e.~a pair consisting of a ${\rm Spin}_0(1,k)$-principal bundle $P_{\rm Spin_0}$ over $\scrM$ together with a twofold covering map $\Lambda$ from $P_{\rm Spin_0}$ to the bundle of positively-oriented tangent frames $P_{\rm SO^+}$ of $\scrM$ such that the following diagram commutes:

\begin{equation*}
	\begin{tikzcd}
		P_{{\rm Spin}_0} \times {\rm Spin}_0(1,k) \arrow[r,"\cdot"]\arrow[d,swap,"\Lambda\times\lambda"] & P_{\rm Spin} \arrow[rd,"\pi"]\arrow[d,"\Lambda"] & \\
		P_{\rm SO^+} \times \textnormal{SO}(1,k)\arrow[r,"\cdot"]  & P_{\rm SO^+} \arrow[r,swap,"\pi"] & \scrM
	\end{tikzcd}
\end{equation*}
where $\lambda$ denotes the double covering at the group level, $\pi$ refers to the respective bundle projections, and the horizontal arrows represent the right group actions on the corresponding principal bundles.

The existence of spin structures is related to the topology of $\scrM$ and there are well-known topological obstructions. A sufficient -- but not necessary -- condition for the existence of a spin structure is parallelizability of $\scrM$. In particular, since any $3$-dimensional  orientable manifold is parallelizable, it follows by Theorem~\ref{thm:Sanchez} that any $(1+3)$-dimensional globally hyperbolic manifold admits a spin structure.

Now, given a fixed spin structure on $(\scrM,g)$, one can use the spin representation $\rho\colon {\rm Spin}_0(1,k) \to \textnormal{Aut}(\mathbb{C}^N)$ to construct the \textit{spinor bundle}, that is, the associated $\C$-vector bundle
$$S\scrM:={\rm Spin}_0(1,k)\times_\rho \mathbb{C}^N\, ,$$
of rank $N:= 2^{\lfloor \frac{k+1}{2}\rfloor}$ (where~$\lfloor\cdot\rfloor$ is the Gau{\ss} bracket; for details see~\cite{baum,lawson+michelsohn}). 
The spinor bundle comes together with the following additional structures:
\begin{itemize}
\item[$\bullet$] a natural ${\rm Spin}_0(1, k)$-invariant indefinite fibre metric
\begin{equation}\label{eq: spin prod}
\Sl\cdot\vert \cdot	\Sr_{{S_p}\scrM} : {S_p}\scrM \times S_p\scrM \to \mathbb{C}\,;
\end{equation}
\item[$\bullet$] a \textit{Clifford multiplication}, i.e.~a vector bundle homomorphism
$$\gamma\colon T\scrM\to \text{End}(S\scrM)\, ,$$ 
which satisfies $\Sl{\gamma(u)\psi}\vert{\phi}\Sr_{{S_p}\scrM}=\Sl{\psi}\vert{\gamma(u)\phi}\Sr_{{S_p}\scrM}$ and the Clifford relations
\begin{equation*}\label{eq:gamma symm}
 \gamma(u)\gamma(v) + \gamma(v)\gamma(u) = -2g(u, v)\1_{S_p\scrM}
\end{equation*}
for all $p \in \scrM$, $u, v \in T_p\scrM$ and $\psi,\phi\in S_p\scrM$.
\end{itemize}

After this preliminary discussion of the relevant geometric structures, we now turn to the (classical) Dirac operator. 

\begin{Def} \label{Def:Dirac}
The \textit{(classical) Dirac operator} $\Dir$ is the operator defined as the 
composition of the metric connection $\nabla^{S\scrM}$ on $S\scrM$, obtained as a lift 
of the Levi-Civita connection on $T\scrM$, and the Clifford multiplication, i.e.
$$\Dir=\imath \gamma\circ\nabla^{S\scrM} \colon C^\infty(\scrM,S\scrM) \to C^\infty(\scrM,S\scrM)\,.$$
\end{Def}

\begin{Remark}
	The classical Dirac operator can also more generally be defined on \emph{Dirac bundles}, in which the spinor bundle is replaced by a more general Clifford module bundle equipped with a Hermitian metric and compatible connection that is also compatible with the Clifford multiplication, see e.g.~\cite[Sec.~5.5]{RudolphSchmidt2}.
\end{Remark}

In local coordinates and with a local trivialisation of the spinor bundle $S\scrM$, the 
Dirac operator $\Dir$ as defined in~\eqref{Def:Dirac} can be written as
\begin{align*}
\Dir \psi = \imath \sum_{\mu=0}^{k}  \epsilon_\mu \gamma(e_\mu) \nabla^{S\scrM}_{e_\mu} 
\psi\, 
\end{align*}
where  $\{e_\mu\}$ is a local Lorentzian orthonormal frame of the tangent bundle $T\scrM$ with $\epsilon_\mu:=g(e_\mu,e_\mu)$, i.e.~$\epsilon_{0}=-1$ and $\epsilon_{i}=1$ for $i=1,\dots,k$. For a given real parameter~$m \in \R$ (the ``mass''), the \textit{Dirac equation} is the linear first-order equation 
\begin{align*}
	(\Dir - m)\psi=0
\end{align*}
for $\psi\in C^{\infty}(\scrM,S\scrM)$. Multiplying the Dirac operator $\Dir$ by a suitable factor of the imaginary unit, it can be realized as a symmetric hyperbolic system on the Hermitian bundle $(S\scrM,\Sl\cdot|\cdot\Sr_{S\scrM})$.

\begin{Prp}\label{Prop:Dirac} For any globally hyperbolic spin manifold $(\scrM,g)$, the operator  $\scrS_{m}:=-\imath (\Dir -m)$ is a symmetric hyperbolic system on $(S\scrM,\Sl\cdot\vert\cdot\Sr_{S\scrM})$ in the sense of Definition~\ref{Def:SHS}. Furthermore, $\scrS_{m}=-\scrS_{m}^{\dagger}$, where $\scrS_{m}^{\dagger}$ denotes the formal adjoint of $\scrS_{m}$ with respect to integrated bundle metric (see~\eqref{eq:SesqDirac} below), as usual.
\end{Prp}

\begin{proof}
	The principal symbol $\sigma_{\scrS_{m}}$ of the operator $\scrS_{m}:=-\imath (\Dir -m)$ is given by 
	\begin{equation*}
		\sigma_{\scrS_{m}}(\xi)=\gamma(\xi^\sharp)\, ,
	\end{equation*}
	for all covectors $\xi\in T^{\ast}\scrM$. Therefore, Property (S) of Definition~\ref{Def:SHS} is verified on account of
$$  \Sl{\gamma(\xi^{\sharp})\psi}\vert{\phi}\Sr_{{S}\scrM}=\Sl{\psi}\vert{\gamma(\xi^{\sharp})\phi}\Sr_{{S}\scrM}\, ,$$
 while Property (H) follows from~\cite[Proposition 1.1]{dimock3}, provided that the spin product was chosen with the appropriate sign.
 
To show that $\scrS_{m}$ is formally skew-adjoint with respect to the sesquilinear form
	\begin{align}\label{eq:SesqDirac}
		\bra\cdot\vert\cdot\ket_{S\scrM}:=\int_{\scrM}\Sl\cdot\vert\cdot\Sr_{S\scrM}\,d\mu_{\scrM}
	\end{align}
	defined on $C^{\infty}_{\mathrm{c}}(\scrM, S\scrM)$, we begin by noting that the zero-order term $\scrS_{m,0}:=-\imath m$ is trivially skew-adjoint. According to Proposition~\ref{Prop:AdjointSHS}, the only remaining contribution to $\scrS_{m} + \scrS_{m}^{\dagger}$ hence arises from the covariant derivatives of the principal symbol. Now, recall that the spin connection $\nabla^{S\scrM}$ is by definition a \emph{Clifford connection} (see~e.g.~\cite[Def.~5.5.10]{RudolphSchmidt2}), which means that it is a Clifford mudole derivation, i.e.
\begin{align*}
	\nabla_{X}^{S\scrM}(\gamma(Y)\psi)=\gamma(\nabla_{X}Y)\psi+\gamma(Y)\nabla_{X}^{S\scrM}\psi
\end{align*}	
for all $\psi\in C^{\infty}(\scrM,S\scrM)$ and vector fields $X,Y\in C^{\infty}(\scrM, T\scrM)$, where $\nabla$ denotes the Levi-Civita connection of $(\scrM,g)$. Denoting by $\nabla^{\mathrm{End}(S\scrM)}$ the induced connection on $\mathrm{End}(S\scrM)$, this is equivalent to say that
\begin{align*}
	\nabla^{\mathrm{End}(S\scrM)}_{X}(\gamma(Y))=\gamma(\nabla_{X}Y)\, .
\end{align*}
In particular, viewing the Clifford  multiplication $\gamma\colon T\scrM\to\mathrm{End}(S\scrM)$ equivalently as a smooth section of the bundle $\mathrm{End}(S\scrM)\otimes T^{\ast}\scrM$ this is equivalent to say that 
	\begin{align*}
		\nabla^{\mathrm{End}(S\scrM) \otimes T^{\ast}\scrM} \gamma = 0\, ,
	\end{align*}
	where $\nabla^{\mathrm{End}(S\scrM) \otimes T^{\ast}\scrM}$ is the tensor product connection constructed out of $\nabla^{\mathrm{End}(S\scrM)}$ and the Levi-Civita connection. Consequently, $\mathscr{D}\sigma_{\scrS_{m}}=0$, where $\mathscr{D}$ is the divergence-like operator as introduced in Proposition~\ref{Prop:AdjointSHS}. This completes the proof.
\end{proof}

Following Proposition~\ref{Prop:Dirac}, the operator $\scrS_{m}:=-\imath (\Dir-m)$ is a symmetric hyperbolic system and hence, following Section~\ref{subsec:SHS}, we can define the Hilbert space 
\begin{align}
\label{L2sprod}
	\H_t:=\overline{C^{\infty}_{\mathrm{c}}(\scrN_{t},S\scrM\vert_{\scrN_{t}})}^{\Vert\cdot\Vert_{t}},\qquad (\psi | \phi)_t := \int_{\scrN_{t}} \Sl \psi | \gamma(\nu)\phi \Sr_{S\scrM}\: d\mu_{\scrN_{t}} \: ,
\end{align}
where $\nu$ denotes the future-directed timelike normal to $\scrN_{t}$, as usual. Due to current conservation, the inner product $(\cdot|\cdot)_{t}$ is in fact independent of the choice of Cauchy hypersurface\footnote{This can also be seen from the fact that $\scrS_{m}+\scrS_{m}^{\dagger}=0$, as proven in Proposition~\ref{Prop:Dirac}, and the energy evolution equation in Proposition~\ref{Prop:energy}.}, for details see~\cite[Section~2]{finite}. 

For a given nonlocal potential $\B\colon C^{\infty}_{\mathrm{c}}(\scrM,S\scrM)\to C^{\infty}(\scrM,S\scrM)$, we shall consider the following the \textit{Dirac equation with nonlocal potential}, i.e.~the Cauchy problem
\begin{align}\label{eq:DiracNonLoc}
	\begin{cases}
		(\Dir + \B -m)\psi &=\phi\\
		\psi\vert_{\scrN_{0}}&=\mathfrak{f}
	\end{cases}
\end{align}
on the time strip $\overline{\scrM}_{T}$ for a given source $\phi\in C^{\infty}_{\mathrm{c}}(\scrM,S\scrM)$ with $\mathrm{supp}(\phi)\subset\overline{\scrM}_{T}$ and initial datum $\mathfrak{f}\in C^{\infty}_{\mathrm{c}}(\scrN_{0},\scrE\vert_{\scrN_{0}})$.

Theorem~\ref{Thm:Small} on the Cauchy problem for symmetric hyperbolic systems with nonlocal potentials with short time range implies the following results:

\begin{Corollary}\label{Cor.Dirac}
	Let $(\scrM,g)$ be a globally hyperbolic manifold such that $\sigma_{\scrS_{m}}(dt)=-\beta^{-2}\gamma(\partial_{t})$ is uniformly bounded in time and let $\B$ be a nonlocal potential with short time range $\delta>0$. If $\B$ is uniformly bounded in time in the sense that
		\begin{align*}
			\exists C\,\,\text{ with }\,\,0<C<(8e\delta^{2})^{-1}:\quad \Vert\V_{t,\tau}\psi_{\tau}\Vert_{t}\leq C\Vert\psi_{\tau}\Vert_{\tau}\,\quad\forall\psi\in C^{\infty}_{\mathrm{sc}}(\scrM,\scrE)\,,\forall t,\tau\in\R\, ,\end{align*} 
			where $\V:=\beta\gamma(\nu)^{-1}\B$ denotes the modified potential, then there exists a strong solution of the Cauchy problem~\eqref{eq:DiracNonLoc} on the time strip $\overline{\scrM}_{T}$.
\end{Corollary}

\begin{proof}
	This is a special case of Theorem~\ref{Thm:Small}, noting that the conditions on the zero-order terms are trivially fulfilled on account of $\scrS_{m}+\scrS_{m}^{\dagger}=0$, as shown in Proposition~\ref{Prop:Dirac}.
\end{proof}

We remark that the uniform boundedness assumption on $\gamma(\partial_{t})$ is essentially a condition on the lapse function $\beta$, since $\gamma(\partial_{t})^{2}=-g(\partial_{t},\partial_{t})\mathrm{id}_{S\scrM}=\beta^{2}\mathrm{id}_{S\scrM}$ by the Clifford relations.

\subsection{A Symmetric Nonlocal Potential and Current Conservation}
For the probabilistic interpretation of the Dirac equation, it is important to endow the solution space with a scalar product, which can be represented by an integral over a Cauchy surface of the form as in~\eqref{L2sprod}. Then, evaluating the scalar product~$(\psi | \psi)_\scrN$ on a Cauchy surface~$\scrN$ for a normalized Dirac solution, the integrand has the interpretation as the probability density of the Dirac particle to be at a certain position on the Cauchy surface. As shown in~\cite{baryogenesis} and used subsequently in~\cite{collapse, fockdynamics}, the probabilistic interpretation of the Dirac equation carries over to the case with a nonlocal potential~$\B$, provided that the nonlocal potential is {\em{symmetric}} in the following sense. We write the operator $\B\colon C^{\infty}_{\mathrm{c}}(\scrM,S\scrM)\to C^{\infty}(\scrM,S\scrM)$ as an integral operator with integral kernel~$k_{\B}(x,y)$,
\beq \label{Bnonlocal}
\big( \B \psi \big)(x) = \int_\scrM k_{\B}(x,y)\, \psi(y)\: d\mu_\scrM(y) \:.
\eeq
Moreover, the integral kernel is assumed to be symmetric in the sense that
\beq \label{Bsymm}
k_{\B}(x,y)^\dagger = k_{\B}(x,y) \:,
\eeq
where the star denotes the adjoint with respect to the fibre metric~\eqref{eq: spin prod} (which for Dirac spinors is also referred to as the spin inner product). In this setting, the scalar product on Dirac solutions~\eqref{L2sprod} needs to be modified by terms involving the nonlocal potential. It takes the form
\begin{subequations} \label{c}
\begin{align}
\la \psi | \phi \ra_\scrN &:= (\psi | \phi)_\scrN \label{c1} \\
&\quad\;\; -i \int_\Omega d\mu_\scrM(x) \int_{\scrM \setminus \Omega} d\mu_\scrM(y)\;
\Sl \psi(x) \,|\, k_{\B}(x,y)\, \psi(y) \Sr_{S_x\scrM} \label{c2} \\
&\quad\;\; +i \int_{\scrM \setminus \Omega} d\mu_\scrM(x) \int_\Omega d\mu_\scrM(y)\;
\Sl \psi(x) \,|\, k_{\B}(x,y)\, \psi(y) \Sr_{S_x\scrM} \:, \label{c3}
\end{align}
\end{subequations}
where~$\Omega \subset \scrM$ is chosen as the past of the Cauchy surface~$\scrN$. This inner product has the structure of a so-called {\em{surface layer integral}} (for the general context see for example~\cite[Section~9.1]{intro}). It follows by direct computation that this inner product is indeed conserved in time (for details see~\cite[Proposition~B.1]{baryogenesis}).

\subsection{Uniqueness of Solutions of the Cauchy Problem} \label{secdirunique}
The conservation law of the previous section can be used to prove uniqueness of solutions to the Cauchy problem. We again need to assume that the nonlocal potential has short time range~$\delta$ and is small compared to~$\delta$, so that Theorem~\ref{Thm:Small} applies. Under these assumptions, we have the estimate
\beq \label{diffes}
\Big| \la \psi | \psi \ra_\scrN - (\psi | \psi)_{\scrN_t} \Big| \\
\leq 2 C \delta\: \sup_{\tau \in [t-\delta, t+\delta]} \: \|\psi_\tau\|^2_{\tau} \:.
\eeq

\begin{Lemma} 
	For all the solutions constructed in Theorem~\ref{Thm:Small}, the surface layer inner
product~$\la .|. \ra_\scrN$ (introduced in \eqref{c}) is positive definite and can be estimated in terms of the $L^2$-scalar product~$(.|.)_\scrN$ in~\eqref{L2sprod} by
	\beq \label{sprodbounds}
		\frac{1}{2}\: (\psi | \psi)_\scrN \leq \la \psi | \psi \ra_\scrN \leq 2\: (\psi | \psi)_\scrN \:.
	\eeq
	\end{Lemma}
	
\Proof We again consider the series ansatz~\eqref{ansatzpertur2}. Estimating each summand according to~\eqref{ites} and using~\eqref{diffes} gives the result.
\QED

Using this result, all the solutions constructed in Theorem~\ref{Thm:Small} endowed with the scalar product~$\la .|. \ra_\scrN$ form a pre-Hilbert space. Forming the completion gives a Hilbert space denoted by~$(\H, \la .|. \ra_\scrN$).
Due to the conservation law, restricting solutions to different surface layers
gives rise to a unitary mapping
\beq \label{unit}
U_{t_0}^{t_1} : \H_{\scrN_{t_0}} \rightarrow \H_{\scrN_{t_1}} \:.
\eeq
These mappings make it possible to identify the Hilbert spaces in different surface layers, giving rise to one Hilbert space denoted by~$(\H, \la .|. \ra)$.

It is not clear whether every solution of the Dirac equation can be represented as the series~\eqref{ansatzpertur2}. Therefore, it remains to treat the class of solutions which do {\em{not}} have this series representation. This can be done as follows. Let~$\phi(x)$ be a general strong solution of the nonlocal Dirac equation. Since we do not assume any decay assumptions at spatial infinity, this solution can in general not be considered as a vector in the Hilbert space~$(\H, \la .|. \ra)$ (simply because the spatial integral for example in~\eqref{c1} will in general diverge). In order to single out those solutions which can be identified with Hilbert space vectors, we restrict attention to solutions~$\phi$ for which  the linear functional
\beq \label{F1}
\Phi : \H_\scrN \rightarrow \C \:,\qquad \psi \mapsto \la \phi | \psi \ra_\scrN
\eeq
is well-defined and bounded. In this case, the Fr{\'e}chet-Riesz theorem allows us to~$\phi$ with a vector in~$(\H, \la .|. \ra)$. Again, according to our conservation law, this identification does not depend on the choice of the Cauchy surface. In this way, the unitary time evolution~\eqref{unit} also extends to the solution~$\phi$.

The above method raises the immediate question for which solutions~$\phi$ the linear functional~$\Phi$ in~\eqref{F1} is bounded. There does not seem to be a simple characterisation of these solutions. But at least, the following lemma gives a convenient sufficient condition. Instead of demanding that the solution be square integrable on a Cauchy surface, we need the stronger assumption that it is square integral over every Cauchy surface in the time range of~$\scrN$:
\begin{Lemma}
	Assume that the assumptions of Theorem~\ref{Thm:Small} again hold. Let~$\phi$ be a strong solution of the Dirac equation with the property that there is a constant~$C>0$ with
\[ (\psi | \phi)_t < C \qquad \text{for all~$t \in (t_0-\delta, t_0+\delta)$}\:. \]
Then the solution~$\phi$ satisfies the inequality~\eqref{sprodbounds} with~$\scrN := \scrN_{t_0}$. Moreover, the linear functional~\eqref{F1} is well-defined and bounded.
\end{Lemma}

\Proof This follows immediately by estimating the surface layer integrals in~\eqref{c2} and~\eqref{c3}.
\QED

\subsection{Applications to Causal Fermion Systems} \label{seccfs}
The Dirac equation with a nonlocal potential as considered in the previous section appears in the analysis of the causal action principle for causal fermion systems~\cite{nonlocal}, with applications to collapse phenomena~\cite{collapse} and the quantum field theory limit~\cite{fockdynamics}. We now outline the connection and explain how the results of the present paper apply in this context.

The theory of {\em{causal fermion systems}} is a recent approach to fundamental physics (see the reviews~\cite{dice2014, review}, the textbooks~\cite{cfs, intro} or the website~\cite{cfsweblink}). In this approach, spacetime and all objects therein are described by a measure~$\rho$ on a set~$\F$ of linear operators on a Hilbert space~$(\H, \la .|. \ra_\H)$. The physical equations are formulated by means of the so-called {\em{causal action principle}}, a nonlinear variational principle where an action~$\Sact$ is minimized under variations of the measure~$\rho$. A minimizing measure satisfies corresponding {\em{Euler-Lagrange (EL) equations}}, being nonlinear equations for~$\rho$. The {\em{linearized field equations}} describe first variations of a minimizing measure which preserve the EL equations. Similar to usual linearisations of physical equations (like for example the equations of linearized gravity), the linearized field equations describe the dynamics of small perturbations of a causal fermion system. In~\cite{nonlocal} the linearized field equations were analysed in detail for causal fermion systems in Minkowski space. In this setting, the dynamics of the wave functions forming the causal fermion system can be described by a Dirac equation of the form
\beq \label{dirnonloc}
\big( \Dir + \B - m \big) \psi = 0 \:,
\eeq
where~$\B$ is a {\em{nonlocal}} potential, i.e.\ an integral operator of the form~\eqref{Bnonlocal}. Moreover, the integral kernel~$k_{\B}(x,y)$ is symmetric as in~\eqref{Bsymm}. It is nonlocal on a scale~$\ell_{\min}$, which lies between the length scale~$\varepsilon$ of the ultraviolet regularisation (which can be thought of as the Planck scale) and the length scale~$\ell_\macro$ of macroscopic physics (which can be thought of as the Compton scale),
\beq \label{ellmindef}
\varepsilon \ll \ell_{\min} \ll \ell_\macro \:.
\eeq
More specifically, the nonlocal potential~$\B$ is composed of a collection of vector potentials~$A_a$ with~$a  \in \{1,\ldots N\}$,
\beq \label{Bstochasticintro}
k_{\B}(x,y) = \sum_{a=1}^N \slashed{A}_a \Big( \frac{x+y}{2} \Big) \:L_a(y-x) \:,
\eeq
where the~$L_a$ are fixed complex-valued kernels. The number~$N$ of these fields is very large and scales like
\beq \label{Nscaleintro}
N \simeq \frac{\ell_{\min}}{\varepsilon} \:.
\eeq
The Cauchy problem for the Dirac equation~\eqref{dirnonloc} can be solved with the help of Theorem~\ref{Thm:Small}, provided that the range~$\ell_{\min}$ and the amplitudes of the potentials~$A_a$ in~\eqref{Bstochasticintro} are sufficiently small.

The nonlocal potential~$\B$ has far-reaching consequences for the dynamics of the Dirac wave functions. First, as worked out in~\cite{collapse}, describing the potentials~$A_a$ stochastically gives rise to a collapse mechanism which can describe the reduction of the wave function in a measurement process. Second, the resulting non-commutativity of the bosonic potentials can be associated to bosonic field operators satisfying the canonical commutation relations. In~\cite{fockdynamics} a limiting case is worked out in which the nonlocal Dirac equation~\eqref{dirnonloc} gives rise to perturbative quantum electrodynamics.

In the above-mentioned works, the Cauchy problem for the nonlocal Dirac equation was solved perturbatively with a Dyson series (for details see~\cite[Section~4.2]{collapse}), but without proving the convergence of this series. The methods worked out in the present paper put this perturbative description on a rigorous mathematical basis. Indeed, Theorem~\ref{Thm:Small} states that the Cauchy problem admits a solution, provided that the nonlocal potential~$\B$ is sufficiently small. In simple terms, our proof consists in showing that the Dyson series converges in suitable function spaces.

\bibliographystyle{amsplain}
\bibliography{cauchynonloc}
\end{document}